\documentclass[UTF-8,reqno]{amsart}
\usepackage{amsmath, amsfonts, amsthm, amssymb, stmaryrd, color, cite, mathrsfs}

\usepackage{hyperref}
\newtheorem{theorem}{Theorem}[section]
\newtheorem{lemma}{Lemma}[section]
\newtheorem{proposition}{Proposition}[section]
\newtheorem{corollary}{Corollary}[section]

\theoremstyle{definition}
\newtheorem{definition}{Definition}[section]

\theoremstyle{remark}
\newtheorem{remark}{Remark}[section]

\numberwithin{equation}{section}
\allowdisplaybreaks

\def\f{\frac}

\def\hf1{^\f{1}{1-\xi^2}}

\def\be{\begin{equation}}
\def\en{\end{equation}}
\def\bs{\begin{split}}
\def\es{\end{split}}
\def\ba{\begin{align}}
\def\ea{\end{align}}

\title[Homogenization of the  non-homogeneous incompressible
Navier-Stokes equations]
{Homogenization and corrector results for the stochastic non-homogeneous incompressible Navier-Stokes equations with rapid oscillation }

\author[Z. Qiu]{Zhaoyang Qiu}
\address{School of Applied Mathematics, Nanjing University of Finance and Economics, Nanjing, 210046, China.}
\email{zhqmath@163.com}

\author[J. Chen]{Junlong Chen}
\address{School of Sciences, Great Bay University,  Great Bay Institute for Advanced Study, Dongguan 523000, China.
}

\address{School of Mathematics, University of Science and Technology of China, Hefei, 230026, Anhui, China.}
\email{chenjunlong@gbu.edu.cn}

\author[J. Duan]{Jinqiao Duan}
\address{Department of Mathematics and Department of Physics, Great Bay University, Dongguan,
	Guangdong 523000, China}
\email{duan@gbu.edu.cn}


\date{\today}

\begin{document}
\begin{abstract}
In this paper we are concerned with the homogenization property of stochastic non-homogeneous incompressible
Navier-Stokes equations with rapid oscillation in a smooth bounded domain of $\mathbb{R}^d$, $d=2,3$, and driven by  multiplicative infinite-dimensional Wiener noise. Using two-scale convergence, stochastic compactness and the martingale representation theory, we first show the solutions of original equations converge to the solution of a stochastic non-homogeneous incompressible homogenized system. Also, the energy equation of the homogenized system is established.  Furthermore, a corrector result is proved which strengthens the two-scale convergence from weak to strong in the regularity space $H^1(\mathcal{O})$. Since the continuity equation which is of transport type cannot confer any regularization effect, there are some issues for proving the two results, including the difficulties for establishing the stochastic compactness and passing to the limit. We develop new regularity estimates, a stochastic version of lower semicontinuity as well as energy equation to overcome these difficulties.
\end{abstract}

\maketitle
\section{Introduction}
\subsection{Governing equations}
The non-homogeneous incompressible Navier-Stokes equations govern the motion of a fluid with spatially and temporally varying density under the assumption of incompressibility. These equations comprise a momentum equation subject to the incompressibility constraint, along with a continuity equation that expresses mass conservation for variable-density flows. They play a fundamental role in modeling fluid behavior where density variations are significant yet the incompressibility condition remains applicable, such as  thermal convection, buoyancy-driven flows, and multiphase systems. For further physical background, we refer  readers to \cite{lad, cww, sim1}.
In this paper, we study stochastic non-homogeneous incompressible Navier-Stokes equations featuring rapidly oscillating terms in the diffusion component and the external force, the specific form is as follows
\begin{eqnarray}\label{equ1}
\left\{\begin{array}{ll}
\!\!\!\partial_t\rho^{\varepsilon}+{\rm div}(\rho^\varepsilon \mathbf{u}^{\varepsilon})=0,\\\\
\!\!\!\partial_t (\rho^\varepsilon\mathbf{u}^{\varepsilon})+A^\varepsilon \mathbf{u}^{\varepsilon}+{\rm div}(\rho^\varepsilon\mathbf{u}^\varepsilon\otimes \mathbf{u}^\varepsilon)+\nabla \mathrm{\pi} =f^\varepsilon( \mathbf{u}^{\varepsilon})+g(\mathbf{u}^{\varepsilon})\frac{dW}{dt},\\\\
\!\!\! {\rm div}\mathbf{u}^{\varepsilon}=0,\\\\
\!\!\! \mathbf{u}^{\varepsilon}|_{\partial \mathcal{O}}=0, ~~~\rho^\varepsilon(0,x)=\rho_0,~~~ \mathbf{u}^{\varepsilon}(0,x)=\mathbf{u}_0,
\end{array}\right.
\end{eqnarray}
where $\mathcal{O}$ is a  bounded domain of class $C^2$ in $\mathbb{R}^d$, $d=2,3$, $\mathbf{u}^\varepsilon: \mathbb{R}^d \times \mathbb{R}^+\rightarrow \mathbb{R}^d$ is the velocity of the fluid flow and $\rho^\varepsilon: \mathbb{R}^d \times \mathbb{R}^+\rightarrow \mathbb{R}$ is the density of fluid flow, which account for the momentum equation and the mass equation respectively,  $\mathrm{\pi}: \mathbb{R}^d \times \mathbb{R}^+\rightarrow \mathbb{R}$ is the pressure.  $\varepsilon\in (0,1)$ is the scale parameter representing the ratio of the microscopic to macroscopic scales.  $f^\varepsilon$ is external force with  oscillation parameter $\varepsilon$
$$f^\varepsilon( \mathbf{u}^{\varepsilon})=f\left(\frac{x}{\varepsilon},  t, \mathbf{u}^\varepsilon\right).$$
$g$ is a noise intensity operator. $W$ is an $H$-valued $Q$-Wiener process defined on $\mathcal{S}:=(\Omega,\mathcal{F},\{\mathcal{F}_{t}\}_{t\geq0},\mathrm{P})$ which is adapted to the complete, right continuous filtration $\{\mathcal{F}_{t}\}_{t\geq 0}$. Assume that $\{\mathbf{e}_{k}\}_{k\geq 1}$ is a complete orthonormal basis of $H$ such that $Q\mathbf{e}_{k}=\lambda_{k}\mathbf{e}_{k}$, then $W$ can be written formally as the expansion $W(t,\omega)=\sum_{k\geq 1}\sqrt{\lambda_{k}}\mathbf{e}_{k}W_k(t,\omega)$, where $\{W_{k}\}_{k\geq 1}$ is a sequence of independent standard one-dimensional Brownian motions, see \cite{Zabczyk} for more details.
Let $H_{0}=Q^{\frac{1}{2}}H$, then $H_{0}$ is a Hilbert space with the inner product
\begin{equation*}
\langle h,\eta\rangle_{H_{0}}=\langle Q^{-\frac{1}{2}}h,Q^{-\frac{1}{2}}\eta\rangle_{H},~ \forall~~ h,\eta\in H_{0},
\end{equation*}
with the induced norm $\|\cdot\|_{H_{0}}^{2}=\langle\cdot,\cdot\rangle_{H_{0}}$. The imbedding map $i:H_{0}\rightarrow H$ is a Hilbert-Schmidt and hence  a compact operator with $ii^{\ast}=Q$, where $i^{\ast}$ is the adjoint of the operator $i$. Then, $W\in C([0,T]; H_0)$ almost surely.

 The term $A^\varepsilon \mathbf{u}^\varepsilon$ represents  the diffusion effect, where the differential operator $A^\varepsilon$ takes the form
$$A^\varepsilon=-\sum_{i,j=1}^{d}\frac{\partial}{\partial x_i}\left(a^\varepsilon_{i,j}\frac{\partial}{\partial x_j}\right).$$
Here the oscillatory coefficient
$$a^\varepsilon_{i,j}(x,t)=a_{i,j}\left(\frac{x}{\varepsilon}, t\right)$$
is symmetric, thus
$$a_{i,j}=a_{j,i},~ i,j=1,\cdots,d$$
and the function $a_{i,j}\in L^\infty (\mathbb{R}^d_y\times \mathbb{R}^+)$. The space $\mathbb{R}^d_y$ is the space $\mathbb{R}^d$ of variable $y=(y_1,y_2, \cdots, y_d)$. In the area of material science, the coefficient $a_{i,j}\left(\frac{x}{\varepsilon}, t\right)$ could be used to describe the microscopic characteristics. As the scale parameter $\varepsilon$ diminishes, it enables the revelation of the intrinsic properties of composite materials, thereby providing a theoretical foundation for their efficient utilization. The operator $A^\varepsilon$ is assumed to satisfy the uniformly elliptic condition, thus, there exists constant $\kappa>0$ such that
\begin{align}\label{1.2}
\sum_{i,j=1}^da^\varepsilon_{i,j}(x,t)\xi_i \xi_j\geq \kappa |\xi|^2,
\end{align}
for any $x, \xi\in \mathbb{R}^d, t\in \mathbb{R}^+$. Here $|\cdot|$ is the Euclidean norm in $\mathbb{R}^d$.
In the composite material, heterogeneity is minimized relative to the overall sample size, such that the mixture exhibits macroscopic homogeneity. This justifies the assumption of a uniform distribution of heterogeneities, which can be mathematically represented by periodicity. Therefore, the coefficient $a_{i,j}$ satisfies the periodicity hypothesis, thus for any $y\in \mathbb{R}_y^d$ and $\widetilde{y}\in \mathbb{Z}^d$,
\begin{align*}a_{i,j}(y+\widetilde{y}, t)=a_{i,j}(y,t).\end{align*}
The study of stochastic non-homogeneous incompressible equations \eqref{equ1} has advanced significantly over the last decade years. The pioneering stochastic result is due to H. Yashima \cite{yas} which established the global existence of martingale solutions of the system with
non-vacuum in the initial density, influenced by additive Gaussian noise. M. Sango \cite{san1} extended the result to the cases of non-Lipschitz multiplicative noise, meanwhile allowing the appearance of vacuum in the initial density. D. Wang et al. \cite{cww} proved the existence of global martingale weak solutions of the equations driven by multiplicative Le\'{v}y noise.

\subsection{Research Progress}
Homogenization is a method used to replace a heterogeneous (highly varying or complex) system with an equivalent homogeneous (uniform) system, while preserving its overall large-scale behavior. This approach is particularly valuable for analyzing systems with properties that vary on small scales, such as materials with microstructures or media with oscillatory coefficients. Research in homogenization not only facilitates numerical computation but also enhances the application of mathematics in dynamical and thermodynamic modeling.

The mathematical theory of homogenization is a rich and interdisciplinary field, which was initially developed by  A. Bensoussan,  J.L. Lions, et al. in the work \cite{ben} in the periodic setting. Then, in the 1970s, G. Nguetseng  \cite{ngu2} and G. Allaire \cite{all} formalized the concept of two-scale convergence,  allowing for the systematic study of partial differential equations (PDEs) with rapidly oscillating coefficients and provided a rigorous framework for deriving effective equations, which is a cornerstone of modern homogenization theory. Building on the two-scale convergence technique, G. Allaire et al. \cite{all2} further studied the homogenization of  nonlinear reaction-diffusion equations with a large oscillation reaction term. L. Signing \cite{sig,sig2} considered the homogenization of the unsteady Stokes-type equations and the unsteady Navier-Stokes equations. The second author and Y. Tang  \cite{CT} studied the homogenization of non-local nonlinear $p$-Laplacian equations with variable index and periodic structure.
In the case of including both heterogeneous coefficients $A^{\varepsilon}$ and the perforated domain,
	W. J\" ager and J.L. Woukeng used  two-scale $\sigma$-convergence to solve the homogenization problem  of the Richards equations and the
	Darcy-Lapwood-Brinkmann system, see \cite{Woukeng2021,Woukeng2023}. W. Niu et al.  studied the periodic homogenization and convergence rates of coefficients in linear elliptic systems and parabolic systems with several time and spatial scales in \cite{Niu2020,Niu2023,Niu2024}, i.e.
	$a^\varepsilon_{i,j}=a_{i,j}\left(x,t, \frac{x}{\varepsilon_1}, \frac{t}{\varepsilon_2}, \frac{x}{\varepsilon_3}, \frac{t}{\varepsilon_4}\cdots\right)$, $\varepsilon_i,  i <\infty$ is a function of $\varepsilon>0$.

In the theory of random homogenization, a key analytical tool is the stochastic two-scale convergence method, introduced by Bourgeat et al.  \cite{Bourgeat1994}.  It is worth mentioning that  S. Neukamm et al. \cite{neukamm2021,neukamm2018} proposed an equivalent characterization of stochastic two-scale convergence using the stochastic unfolding operator, and applied it to the homogenization of abstract linear time-dependent PDEs. Based on the stochastic two-scale convergence, M. Sango et al.\cite{raz} generalized the result \cite{all2} to stochastic reaction-diffusion equations with almost periodic framework, see also \cite{moh} for the type of linear hyperbolic stochastic PDEs.
J. Duan et al. \cite{duan,duan2} considered the homogenization of stochastic PDEs related to Hamiltonian systems etc. with L\'{e}vy noise, see also \cite{HJS, wang2} for the non-symmetric jump processes, and stochastic PDEs with dynamical boundary conditions. The first two authors and Y. Tang \cite{chen} proved the homogenization property for the stochastic abstract fluid models with multiplicative cylindrical Wiener process, including the homogeneous Navier-Stokes equations, the Boussinesq equations, the  Allen-Cahn equations etc. We refer  readers  to \cite{ bess, bess1, bess2, Bessaih2021, glo, ich,moh2, neu, par, san2, wang1, zhang} and references therein for more results.

As far as the authors
are aware, there appears to be no existing result in the literature concerning the homogenization of the stochastic non-homogeneous incompressible Navier-Stokes equations \eqref{equ1} with periodically oscillating coefficients.
The main goal of the present paper
is to consider this problem of system \eqref{equ1} for $d=2,3$ as $\varepsilon\rightarrow 0$. Moreover, the homogenization results established here are novel even in the context of deterministic equations.  In the future, we will study the coefficient
$A^{\varepsilon}$
from the periodic case to the context of almost periodic and stationary ergodicity.

\subsection{Definitions and Assumptions}
Before presenting the main results, we first introduce some basic notations, definitions of two-scale convergence and the assumptions. Throughout the paper, if \(\alpha_1, \alpha_2 \in \mathbb{R} \), we define  \(\alpha_1 \lesssim_\alpha C\alpha_2, \) means that the constant $C>0$ relies on $\alpha$ such that \( \alpha_1\leq C(\alpha)\alpha_2 \).
 Denote by $\mathcal{O}_t = \mathcal{O} \times [0,T]$ and $D=\left(-\frac{1}{2},\frac{1}{2}\right)^d$ which is the subset of $\mathbb{R}_y^d$. Denote by $L^p_{per}(D)$ all the  $D$-periodic functions in $L_{loc}^p(\mathbb{R}^d)$, endowed with the norm
$$\|\mathbf{f}\|^p_{L^p_{per}(D)}=\int_{D}|\mathbf{f}(y)|^pdy,$$
which is a Banach space.

Now, we recall the concepts of weak, strong two-scale convergence.
\begin{definition} A sequence of $L^p(\mathcal{O}_t)$-valued random variables $\{\mathbf{u}^\varepsilon\}_{\varepsilon\in (0,1)}$ is  weak-$\Sigma$ convergent in $L^p( \mathcal{O}_t)$ if there exists a certain $L^p(\mathcal{O}_t; L^p_{per}(D))$-valued random variable $\mathbf{u}$ such that as $\varepsilon\rightarrow 0$,
	$$\int_{\mathcal{O}_t}\mathbf{u}^\varepsilon(x,t)\psi\left(x,\frac{x}{\varepsilon},t\right)dxdt\rightarrow \int_{\mathcal{O}_t}\int_{D}\mathbf{u}(x,y,t)\psi(x,y,t)dxdydt,$$
	for any $\psi\in L^{p'}( \mathcal{O}_t; L^{p'}_{per}(D))$, where $\frac{1}{p}+\frac{1}{p'}=1$.
\end{definition}

\begin{definition}\label{def4.2} A sequence of $L^p(\mathcal{O}_t)$-valued random variables $\{\mathbf{u}^\varepsilon\}_{\varepsilon\in (0,1)}$ is strong-$\Sigma$ convergent in $L^p(\Omega\times \mathcal{O}_t)$ if there exists a certain $L^p(\mathcal{O}_t; L^p_{per}(D))$-valued random variable $\mathbf{u}$ such that as $\varepsilon\rightarrow 0$,
	$$\mathrm{E}\left(\int_{\mathcal{O}_t}\mathbf{u}^\varepsilon(x,t, \omega)\mathbf{v}^\varepsilon(x,t, \omega) dxdt\right)\rightarrow \mathrm{E}\left(\int_{\mathcal{O}_t}\int_{D}\mathbf{u}(x,y,t,\omega)\mathbf{v}(x,y,t,\omega)dxdydt\right),$$
	for any bounded $\mathbf{v}^\varepsilon\in L^{p'}(\Omega\times \mathcal{O}_t)$ with $\mathbf{v}^\varepsilon \rightarrow \mathbf{v}$ in $L^{p'}(\Omega\times \mathcal{O}_t)$, weak-$\Sigma$, where $\frac{1}{p}+\frac{1}{p'}=1$.
\end{definition}

{\bf Assumptions on $f$ and $g$}.
For the external force $f$,  we assume that the function 
$f: \mathbb{R}_y^{d}\times \mathbb{R}^+\times \mathbb{R}^{d}\rightarrow \mathbb{R}^{d}$ 
is $\mathbb{Z}^d $ periodic with respect to the variable $y$,
 moreover the Lipschitz and linear growth conditions hold, thus there exists two constants  $c_1, c_2>0$ such that

\qquad $(\mathbf{A.1})~~ |f(y, t, \xi_1)-f(y, t, \xi_2)|\leq c_1|\xi_1-\xi_2|, ~{\rm for}~ (y,t)\in \mathbb{R}_y^{d}\times \mathbb{R}^+, ~\xi_1, \xi_2\in  \mathbb{R}^{d}; $\\

\qquad $(\mathbf{A.2})~~ |f(y, t, \xi)|\leq c_2(1+|\xi|),  ~{\rm for}~ (y,t)\in \mathbb{R}_y^{d}\times \mathbb{R}^+,~\xi\in  \mathbb{R}^{d}$.
 
 For the operator $g$ we also assume that $g: H\rightarrow L_2(H; H)$ satisfies the Lipschitz and linear growth conditions, thus there exists two constants  $c_3, c_4>0$ such that

\qquad $(\mathbf{A.3})~~ \|g(\mathbf{u}_1)-g(\mathbf{u}_2)\|^2_{L_2(H; H)}\leq c_3\|\mathbf{u}_1-\mathbf{u}_2\|_{H}^2, ~{\rm for}~ \mathbf{u}_1, \mathbf{u}_2\in H; $\\

\qquad $(\mathbf{A.4})~~ \|g(\mathbf{u})\|^2_{L_2(H; H)}\leq c_4(1+\|\mathbf{u}\|_{H}^2),  ~{\rm for}~ \mathbf{u}\in H,$\\
 where  $L_2(H; H):=\left\{SQ^{\frac{1}{2}}: \, S\in L_{Q}(H_{0}; H)\right\}$, and $L_{Q}(H_{0}; H)$ is the space of all linear operators $S:H_{0}\rightarrow H$ such that $SQ^{\frac{1}{2}}$ is a linear Hilbert-Schmidt
operator from $H$ to itself,  endowed with
  the norm $$\|S\|_{L_{Q}}^{2}=tr(SQS^{\ast}) =\sum\limits_{k\geq 1}\| SQ^{\frac{1}{2}}\mathbf{e}_{k}\|_{H}^{2}.$$

\subsection{Main results} We now turn to a more precise statement of our
results.
\begin{theorem}\label{thm2.1} Assume that the assumptions \((\mathbf{A.i})\), \(\mathbf{i} = 1, 2, 3, 4\) hold and initial data $(\rho_0, \mathbf{u}_0)$ is deterministic with \(\mathbf{u}_0 \in H\), \(0 < m \leq \rho_0 \leq M<\infty\). Then, we have that there exists a subsequence of solutions $\{\rho^\varepsilon, \mathbf{u}^\varepsilon, \rho^\varepsilon\mathbf{u}^\varepsilon\}_{\varepsilon\in (0,1)}$ of equations \eqref{equ1} satisfying the convergence $\mathrm{P}$ a.s.
\begin{eqnarray}\label{2.2}
\left\{\begin{array}{ll}
\!\!\!\mathbf{u}^\varepsilon\rightarrow \mathbf{u},~
{\rm strongly ~in}~ L^2(0,T; H),\\\\
\!\!\!\nabla \mathbf{u}^\varepsilon\rightarrow \nabla_x \mathbf{\mathbf{u}}+\nabla_y \overline{\mathbf{u}}, ~{\rm weak }-\Sigma~{\rm in}~ L^2(0,T; H),\\\\
\!\!\! \rho^\varepsilon\rightarrow \rho, ~{\rm strongly ~in}~ C([0,T]; W^{-\alpha,p}(\mathcal{O})),\\\\
\!\!\! \rho^\varepsilon\mathbf{u}^\varepsilon\rightarrow \rho\mathbf{u}, ~{\rm strongly ~in}~ C([0,T]; W^{-\alpha, 2}(\mathcal{O})),
\end{array}\right.
\end{eqnarray}
for any $\alpha\in (0,1), p\in (1,\infty)$, where the corrector $\overline{\mathbf{u}}$ is given by Lemma \ref{lem4.6}. Moreover, for every $t\in [0,T]$, and $\phi\in H^1(\mathcal{O})$, $\varphi\in V$,  the limit  $(\rho(t), \mathbf{u}(t),  \rho(t)\mathbf{u}(t))$ satisfies the following homogenized Navier-Stokes equations 
\begin{eqnarray*}
\left\{\begin{array}{ll}
\!\!\!(\rho(t), \phi)-(\rho_0, \phi)-\int_{0}^{t}(\rho(s) \mathbf{u}(s), \nabla\phi)ds=0,\\\\
\!\!\!(\rho(t)\mathbf{u}(t), \varphi)=(\rho_0\mathbf{u}_0, \varphi)-\int_0^t(\overline{A }\mathbf{u}(s), \varphi)_{ V'\times V}ds+\int_0^t(\rho(s)\mathbf{u}(s)\otimes \mathbf{u}(s), \nabla\varphi)ds\nonumber\\
\!\!\!\qquad\qquad+\int_{0}^{t}(\overline{f}( \mathbf{u}(s)), \varphi)ds+\int_{0}^{t}(g(\mathbf{u}(s))dW, \varphi),
\end{array}\right.
\end{eqnarray*}
in which the homogenized operator $\overline{A }$ is given in \eqref{4.39}, the function $\overline{f}$ is given by
$$\overline{f}(\mathbf{u})=\int_{D}f(y, s, \mathbf{u})dy.$$
Furthermore, for every $t\in [0,T]$, $( \mathbf{u}(t), \overline{\mathbf{u}}(t),  \rho(t)\mathbf{u}(t))$ satisfies the following energy equation  $\mathrm{P}$ a.s.
\begin{align}\label{1.4}
&\| \sqrt{\rho(t)}\mathbf{u}(t)\|^2_{L^2(\mathcal{O})}- \| \sqrt{\rho_0}\mathbf{u}_0\|^2_{L^2(\mathcal{O})}\nonumber\\
&=-2\sum_{i,j=1}^d\int_{0}^{t}\int_{\mathcal{O}}\int_{D}a_{i,j}(y,s)\left(\frac{ \partial\mathbf{u}(x, s)}{\partial{x_i}}+\frac{ \partial\overline{\mathbf{u}}(x, y, s)}{\partial {y_i}}\right)\cdot \left(\frac{ \partial\mathbf{u}(x, s)}{\partial x_j}+\frac{ \partial\overline{\mathbf{u}}(x, y, s)}{\partial y_j}\right) dxdyds\nonumber\\
&\quad+2\int_{0}^{t}\int_{\mathcal{O}}\int_{D}f( y, s, \mathbf{u}(s)) \mathbf{u}(s)dxdyds+2\int_{0}^{t}(g(\mathbf{u}(s))dW,  \mathbf{u}(s))+
\int_{0}^{t}\|g(\mathbf{u}(s))\|_{L_2(H;H)}^2ds.
\end{align}
\end{theorem}

The Theorem \ref{thm2.1} is proved by combining the two-scale convergence, stochastic compactness method as well as the Jakubowski version of the Skorokhod representation theorem.
From a theoretical perspective,
the homogenization problem is significantly more complicated for non-homogeneous equations \eqref{equ1} compared with  reaction-diffusion equations, homogeneous hydrodynamic equations, etc..
Since the equations \eqref{equ1} contain the continuity equation which is a transport equation, several difficulties arise in proving the homogenization problem, including the establishment of a priori estimates and the passage to the limit in the approximating system.

Now, let us explain why  the establishment of a priori estimates and the passage to the limit in the approximating system are challenging. The papers \cite{moh, raz} established the homogenization for the stochastic homogeneous equations driven by the finite-dimensional Brownian motions, and the homogenized stochastic equations were given in the sense of expectation. Here, our goal is to deal with the  infinite-dimensional $Q$-Wiener process, and build the homogenized Navier-Stokes equations in the pathwise sense. Hence, 
the method used in \cite{moh, raz} is not suitable for our situation. To this end, the martingale representation method is invoked to overcome this difficulty. Following the strategy,  we need to show that for every $t\in [0,T]$, and  P a.s. $\omega$
$$\mathcal{G}_t^\varepsilon(\rho^\varepsilon, \mathbf{u}^{\varepsilon})\rightarrow \mathcal{G}_t(\rho, \mathbf{u}), ~~as~~ \varepsilon\rightarrow 0,$$
where $\mathcal{G}_t^\varepsilon, \mathcal{G}_t$ are two functionals defined in Step 2 of Proposition \ref{pro2.1}, for which we  need to overcome two  difficulties: (i) Due to the adverse effect of the continuity equation and the noise, the first difficulty is how to 
establish the time-regularity estimates of $\mathbf{u}^\varepsilon$; (ii) The  two-scale convergence of the stochastic version (Lemma \ref{lem4.2}) defined in the sense of expectation cannot be used here, hence the second difficulty is how to pass to the limit of the diffusion term.

Firstly, to establish the time-regularity estimates of $\mathbf{u}^\varepsilon$, we need to prove
$$\mathrm{E} \left(\sup_{\theta\in (0,T)}\left(\theta^{-\frac{1}{2}}\int^{T - \theta}_{0} \| \rho^{\varepsilon}(t + \theta) \mathbf{u}^{\varepsilon}(t + \theta) -\rho^{\varepsilon}(t) \mathbf{u}^{\varepsilon}(t) \|^2_{V^{\prime}} \, dt\right)\right)
	\lesssim_{m,M,c_2, c_4,\kappa,T} C.$$
Hence, it is unavoidable to prove the following estimate for noise
\begin{align}\label{e1}\mathrm{E} \left(\sup_{\theta\in (0,T)}\left(\theta^{-\frac{1}{2}}\int_{0}^{T-\theta}\left\|\int^{t + \theta}_{t} g( \mathbf{u}^{\varepsilon}(s)) \, dW\right\|^2_{V'}dt\right)\right)
	\lesssim_{m,M,c_2, c_4,\kappa,T} C.\end{align}
Note that the classical Burkholder-Davis-Gundy inequality cannot be applied directly to deal with it. In order to control the noise, we introduce a new random variable 
$$K(\omega)=\sup_{0\leq s<t\leq T}\frac{\|G(t)-G(s)\|_{V'}}{|t-s|^\gamma}, $$
where $G(t):=\int_0^tg( \mathbf{u}^{\varepsilon}(s)) \, dW$. We can show $K(\omega)<\infty$,~ {\rm P}~ a.s.
As a result,  the inequality $\|G(t)-G(s)\|_{V'}\leq K(\omega)|t-s|^\gamma$ is meaningful. Then, the right-hand side term of \eqref{e1} can be transformed  into a controllable form, see Lemma \ref{lem3.2} for details.

Secondly, in order to pass to the limit of the diffusion term, we use the deterministic two-scale convergence, Lemmas \ref{lem2.1}, \ref{lem4.2*} rather than the stochastic version. Hence we need the uniform estimates for every $s\in [0,T]$, and P a.s. $\omega$
$$\int_{0}^{t}\|\nabla \mathbf{u}^{\varepsilon}(s)\|_{L^2(\mathcal{O})}^2ds\lesssim_{\omega, m,M,c_2, c_4,\kappa,T} C.$$
It is worth  mentioning that the aforementioned uniform estimates are not the direct consequence of uniform estimates $\mathbf{u}^{\varepsilon}\in L^2(\Omega; L^2(0,T;V))$ in $\varepsilon$. In this paper, by fully exploiting the construction of the equations, we  provide a rigorous proof of the estimates, see Lemma \ref{lemapx.1}.

\begin{remark} We emphasize that the convergence result \eqref{2.2} holds only in a new probability space \((\Omega, \mathcal{F}, \mathrm{P})\), not in the original stochastic basis \(\mathcal{S}\), owing to an application of the Skorokhod-Jakubowski representation theorem. Thus, the convergence is weak both in the sense of probability and PDEs.
\end{remark}

\begin{remark} Unlike \cite{cww}, here the oscillation external force $f^\varepsilon$ cannot  depend on the density $\rho^\varepsilon$. In other words, we cannot even deal with the simple case $\rho^\varepsilon f^{\varepsilon}$. The reason is as follows: the weak
convergence inherited from the uniform bounds is not enough to identify the limit, thus, $\rho^\varepsilon\rightarrow \rho$ in $C([0,T]; W^{-\alpha, p}(\mathcal{O}))$ with $\alpha\in (0,1)$ and $f^\varepsilon(\mathbf{u}^\varepsilon)\rightarrow f(\cdot, \cdot, \mathbf{u})$ ~{\rm weak}-$\Sigma$,~ {\rm in}~$L^2( \mathcal{O}_t)$ are very weak, hence, we cannot find a suitable space to pass to the limit in the sense of two-scale convergence. However, the  fact that $f^\varepsilon$ does not depend on  $\rho^\varepsilon$ brings troubles in the a priori $p$-order moment estimates. In order to solve the  problem, we have to assume that the initial density $\rho_0$ is away from the vacuum.
\end{remark}


The following corrector result is proved which improves the weak-$\Sigma$ convergence of $\nabla\mathbf{u}^\varepsilon$ in $L^2(\mathcal{O}_t)$ to strong-$\Sigma$ in $L^2(\mathcal{O}_t)$.
\begin{theorem}\label{thm2} Under the same assumptions as those of in Theorem \ref{thm2.1}, we have
$$\frac{\partial\mathbf{u}^\varepsilon}{\partial x_i}\rightarrow \frac{\partial\mathbf{u}}{\partial x_i}+\frac{\partial\overline{\mathbf{u}}}{\partial y_i}, ~{\rm in }~ L^2(\Omega\times \mathcal{O}_t),~{\rm strong}-\Sigma, ~1\leq i\leq d$$
as $\varepsilon\rightarrow 0$.
\end{theorem}

 We already established the strong-$\Sigma$ convergence for stochastic homogeneous  hydrodynamic  models, Allen-Cahn equations, etc. in \cite{chen}. For the stochastic homogeneous  hydrodynamic  models, for every $t \in [0,T]$ the required convergence $\mathrm{E}\left(\|\mathbf{u}^\varepsilon(t)\|_{H}^2\right)\rightarrow \mathrm{E}\left(\|\mathbf{u}(t)\|_{H}^2\right)$ can be directly deduced from regularity estimates and compactness criterion. The proof of non-homogeneous case is non-trivial compared with \cite{chen}.
  Specifically, the proof relies on the  convergence \begin{align}\label{e2}\mathrm{E}\left(\|\sqrt{\rho^\varepsilon(t)}\mathbf{u}^\varepsilon(t)\|^2_{L^2(\mathcal{O})}\right)\rightarrow \mathrm{E}\left(\|\sqrt{\rho(t)}\mathbf{u}(t)\|^2_{L^2(\mathcal{O})} \right),\end{align} for every $t\in [0,T]$.
  Note that the continuity equation is of transport type and enforces mass conservation. As a transport equation, it is inherently inviscid, and thus one cannot expect the density function to confer any regularizing effect, especially under the physical assumption that the initial density is bounded only in 
$L^\infty(\mathcal{O})$. By exploiting properties of the transport equation, we can only obtain $\rho^\varepsilon(x,t,\omega)\in [m,M]$, for all $(x,t) \in \mathcal{O}_t$ and $\omega\in \Omega$. Then, by the lower and upper-bounds of density, we can derive the estimates $\rho^\varepsilon\mathbf{u}^\varepsilon\in L^p(\Omega; L^\infty(0,T;L^2(\mathcal{O})))$. Hence, unlike \cite{chen} we cannot use the standard compactness criterion such as the Aubin-Lions Lemma to achieve the convergence  due to the limited regularity of $\rho^\varepsilon\mathbf{u}^\varepsilon$. 
  
  We solve the problem by the idea of energy equation. A stochastic version of lower semicontinuity is established firstly, see Lemma \ref{lem5.1}. Based on the result and the energy equation, we show that \begin{align}\label{sup1}\limsup_{\varepsilon\rightarrow 0}\mathrm{E}\left(\|\sqrt{\rho^\varepsilon(t)}\mathbf{u}^\varepsilon(t)\|^2_{L^2(\mathcal{O})}\right)\leq  \mathrm{E}\left(\|\sqrt{\rho(t)}\mathbf{u}(t)\|^2_{L^2(\mathcal{O})}\right).\end{align}
In addition, from the boundedness of $\rho^\varepsilon\mathbf{u}^\varepsilon$ we derive
 $$\liminf_{\varepsilon\rightarrow 0}\mathrm{E}\left(\|\sqrt{\rho^\varepsilon(t)}\mathbf{u}^\varepsilon(t)\|^2_{L^2(\mathcal{O})}\right)\geq  \mathrm{E}\left(\|\sqrt{\rho(t)}\mathbf{u}(t)\|^2_{L^2(\mathcal{O})}\right), $$ which  together with \eqref{sup1} {leads to the desired convergence \eqref{e2}.}
Furthermore, applying the convergence \eqref{e2} and stochastic two-scale convergence, we could strengthen the weak convergence result to strong convergence, as presented in Theorem \ref{thm2}.

The rest of the paper is organized as follows. We introduce some preliminaries including functional spaces and operators, then establish the   a priori uniform estimates and the stochastic compactness in section 2.   In section 3, using two-scale convergence and stochastic compactness we prove the homogenization result by passing to $\varepsilon\rightarrow 0$.
We improve the convergence of $\nabla \mathbf{u}^\varepsilon$ in $L^2( \mathcal{O}_t)$,~weak-$\Sigma$ to the $L^2(\mathcal{O}_t)$, strong-$\Sigma$ in section 4, while a stochastic version of lower semicontinuity is established.

\section{Uniform estimates and stochastic compactness}

In this section, we establish the a priori uniform estimates and the stochastic compactness used for studying the homogenization property. We begin by recalling some preliminaries, including functional spaces, operators, and stochastic setting, which will be used in the sequel.

{\bf Functional spaces and operators}. For any $k\in \mathbb{N}^+, p\geq 1$, denote by $W^{k, p}(\mathcal{O})$ the Sobolev spaces of functions having distributional derivatives up to order $k\in \mathbb{N}^{+}$, which is integrable in $L^{p}(\mathcal{O})$. We denote by $W^{-k, p'}(\mathcal{O})$ the dual of $W^{k, p}(\mathcal{O})$, $p'$ is the conjugate index of $p$, and $H^1(\mathcal{O})=W^{1,2}(\mathcal{O})$.
Denote by $C_0^\infty(\mathcal{O})$ the space of all $\mathbb{R}^d$-valued functions of class $C^\infty(\mathcal{O})$ with compact supports contained in $\mathcal{O}$.
Let
$$C_{0, div}^\infty(\mathcal{O})=\{\mathbf{u}\in C_0^\infty(\mathcal{O}); {\rm div} \mathbf{u}=0\}.$$
Define by $H$ the closure of $C_{0, div}^\infty(\mathcal{O})$ in $L^2(\mathcal{O})$-norm, $V$ the closure of $C_{0, div}^\infty(\mathcal{O})$ in $H^1(\mathcal{O})$-norm, endowed with the $L^2(\mathcal{O})$-norm, $H^1(\mathcal{O})$-norm respectively.
Denote by $V'$ the dual of space $V$, then these spaces satisfy
the Gelfand inclusions $V\subset H\subset V'$. We denote by $(\cdot, \cdot), \|\cdot\|_{H}, \|\cdot\|_{V}$ the inner product of $L^2(\mathcal{O})$ and the norms $H, V$. The duality product between $V, V'$ is denoted by $(\cdot, \cdot)_{V\times V'}$. 

For the oscillation diffusion term,  we understand $A^\varepsilon$ as the bounded operator from $V$ into $V'$ with the duality product
$$(A^\varepsilon \mathbf{u}, \mathbf{v})_{V'\times V}=\sum_{i,j=1}^{d}\left(a^\varepsilon_{i,j}\frac{\partial \mathbf{u}}{\partial x_j}, \frac{\partial \mathbf{v}}{\partial x_i}\right), ~{\rm for}~\mathbf{u}, \mathbf{v}\in V.$$
Since the embedding $V$ into $H$ is compact, it follows that for every $\varepsilon\in (0,1)$, $(A^\varepsilon)^{-1}$ as a map from $H$ into $V$ is compact on $H$. From the symmetry and the compactness of the operator, we have the existence of a complete orthonormal basis $\{\mathbf{e}_{k}\}_{k\geq 1}$ for $H$ of eigenfunctions of $A^\varepsilon$. 

Let $X$ be a Banach space, for any $0<\alpha<1, p\geq 1$, denote by
 \begin{equation*}
W^{\alpha,p}(0,T;X):=\left\{\mathbf{u}\in L^{p}(0,T;X):\int_{0}^{T}\int_{0}^{T}\frac{\|\mathbf{u}(t)-\mathbf{u}(s)\|^p_{X}}{|t-s|^{\alpha p+1}}dsdt<\infty\right\},
\end{equation*}
where the norm is defined 
by $$\|\mathbf{u}\|^p_{W^{\alpha,p}(0,T;X)}=\int_{0}^{T}\|\mathbf{u}(t)\|_{X}^pdt+\int_{0}^{T}\int_{0}^{T}\frac{\|\mathbf{u}(t)-\mathbf{u}(s)\|^p_{X}}{|t-s|^{\alpha p+1}}dsdt.$$
If $\alpha=1$, denote by
 \begin{equation*}
W^{1,p}(0,T;X):=\left\{\mathbf{u}\in L^{p}(0,T;X):\frac{d\mathbf{u}}{dt}\in L^{p}(0,T;X)\right\},
\end{equation*}
which is the classical Sobolev space with its usual norm
\begin{equation*}
\|\mathbf{u}\|_{W^{1,p}(0,T;X)}^{p}=\int_{0}^{T}\left(\|\mathbf{u}(t)\|_{X}^{p}+\left\|\frac{d\mathbf{u}(t)}{dt}\right\|_{X}^{p}\right)dt.
\end{equation*}

{\bf Stochastic setting}. Denote by $L^p(\Omega; L^q(0,T;X)), p\in [1,\infty], q\in [1, \infty]$ the space of processes with values in $X$ defined on $\Omega\times [0,T]$ such that

i. $\mathbf{u}$ is measurable with respect to $(\omega, t)$, and for each $t\geq 0$, $\mathbf{u}(t)$ is $\mathcal{F}_{t}$-measurable;

ii. For almost all $(\omega, t)\in \Omega\times [0,T]$, $\mathbf{u}\in X$ and
\begin{align*}
\|\mathbf{u}\|^p_{L^p(\Omega; L^q(0,T;X))}\!=\!
\begin{cases}
\mathrm{E}\left(\int_{0}^{T}\|\mathbf{u}(t)\|_{X}^qdt\right)^\frac{p}{q},~& {\rm if}~q\in [1,\infty),\\\\
\mathrm{E}\left(\sup_{t\in[0,T]}\|\mathbf{u}(t)\|^p_X\right),~ & {\rm if}~q=\infty.
\end{cases}
\end{align*}
 If $p=\infty$, denote
$$L^\infty(\Omega; L^q(0,T; X)):=\inf\left\{\zeta; \mathrm{P}(L^q(0,T;X)> \zeta)=0\right\}.$$
Here,
\[
\mathrm{P}(L^q(0,T; X) > \zeta) = 0
\]
means that \( \rho: \Omega \to L^q(0,T; X) \) is essentially bounded.

We recall the following well-known Burkholder-Davis-Gundy inequality to control the martingale part: for any $g\in  L^{p}(\Omega;L^{2}_{loc}([0,\infty); L_{2}(H; X)))$, there exists a constant $c_p>0$ such that
\begin{eqnarray*}
\mathrm{E}\left(\sup_{t\in [0,T]}\left\|\int_{0}^{t}g(s) dW(s)\right\|_{X}^{p}\right)\leq c_{p}\mathrm{E}\left(\int_{0}^{T}\sum_{k\geq 1}\|g(t)Q^\frac{1}{2}\mathbf{e}_k\|_{X}^{2}dt\right)^{\frac{p}{2}},
\end{eqnarray*}
for any $p\in[1,\infty)$, see also \cite[Theorem 4.36]{Zabczyk}.

	The existence of a martingale weak solution is given by \cite{cww} for \( d = 3 \). Here, using the Galerkin approximate method, we also have the following existence result for \( d = 2, 3 \).
\begin{proposition}\label{lem2.1*}Assume that the assumptions \((\mathbf{A.i})\), \(\mathbf{i} = 1, 2, 3, 4\) hold and initial data \(\mathbf{u}_0 \in H\), \(0 < m \leq \rho_0 \leq M<\infty\). Then, for every \(\varepsilon\in (0,1)\), \(T > 0\), and \(d = 2, 3\), there exists a global martingale weak solution of the equations \eqref{equ1} in the following sense:
	
i. $(\Omega,\mathcal{F},\{\mathcal{F}_{t}\}_{t\geq0},\mathrm{P}, W)$ is a filtered probability space with a filtration $\{\mathcal{F}_{t}\}_{t\geq0}$, $W$ is a $Q$-Wiener process adapted to filtration $\{\mathcal{F}_{t}\}_{t\geq0}$.

ii. $\mathbf{u}^\varepsilon$ is $H$-valued $\mathcal{F}_t$-progressively measurable process with the regularity
$$\mathbf{u}^\varepsilon\in L^p(\Omega; L^\infty(0,T;H)\cap L^2(0,T;V))$$
for any $p\geq 2$, $\rho^\varepsilon$ is $L^\infty(\mathcal{O})$-valued $\mathcal{F}_t$-progressively measurable process with the regularity
$$\rho^\varepsilon\in L^\infty(\Omega\times \mathcal{O}_t).$$
Moreover, we have that $\rho^\varepsilon\mathbf{u}^\varepsilon$ is $L^2(\mathcal{O})$-valued $\mathcal{F}_t$-progressively measurable process with the regularity
$$\rho^\varepsilon\mathbf{u}^\varepsilon\in L^p(\Omega; L^\infty(0,T;L^2(\mathcal{O}))).$$

iii. For any $t\in [0,T]$, $\phi\in H^1(\mathcal{O}), \varphi\in V$, it holds $\mathrm{P}$ a.s.
$$(\rho^{\varepsilon}(t), \phi)-(\rho_0, \phi)-\int_{0}^{t}(\rho^\varepsilon(s) \mathbf{u}^{\varepsilon}(s), \nabla\phi)ds=0,$$
and
\begin{align}\label{2.1*}
&(\rho^\varepsilon(t)\mathbf{u}^\varepsilon(t), \varphi)=(\rho_0\mathbf{u}_0, \varphi)-\int_0^t(A^\varepsilon \mathbf{u}^\varepsilon(s), \varphi)_{ V'\times V}ds+\int_0^t(\rho^\varepsilon(s)\mathbf{u}^\varepsilon(s)\otimes \mathbf{u}^\varepsilon(s), \nabla\varphi)ds\nonumber\\
&\qquad\qquad\qquad+\int_{0}^{t}(f^\varepsilon( \mathbf{u}^{\varepsilon}(s)), \varphi)ds+\int_{0}^{t}(g(\mathbf{u}^{\varepsilon}(s))dW, \varphi). 
\end{align}

\end{proposition}

\subsection{A priori unform estimates}
	
In this subsection, we establish the uniform temporal and spatial a priori regularity estimates of solutions $\rho^\varepsilon, \mathbf{u}^\varepsilon, \rho^\varepsilon\mathbf{u}^\varepsilon$ in $\varepsilon$. Then, using the a priori regularity estimates, we will derive the tightness of a sequence of measures induced by the distributions of these solutions. 	

\begin{lemma}\label{lem3.1*} If the initial density $\rho_0$ satisfies $0<m\leq \rho_0\leq M<\infty$, then the sequence of solutions $\{\rho^\varepsilon\}_{\varepsilon\in (0,1)}$ in  equations \eqref{equ1}  has the following uniform estimates of $\varepsilon$
$$0<m\leq \rho^\varepsilon(x,t,\omega)\leq M<\infty,$$
for any $(x,t)\in \mathcal{O}_t, \omega\in \Omega$.
\end{lemma}
\begin{proof}
	Since the continuity equation is a type of transport equation, hence the solution $\rho^\varepsilon$ shares the same regularity with initial density $\rho_0$ from \cite{Lions1989},
thus for any $(x,t)\in \mathcal{O}_t, \omega\in \Omega,$ we obtain
	$$
	0<m\leq \rho^\varepsilon(x,t, \omega)\leq M<\infty,$$
 uniformly in $\varepsilon$.
\end{proof}

\begin{lemma}\label{lem3.1} If $(\mathbf{A.2})$, $(\mathbf{A.4})$ hold and \(\mathbf{u}_0 \in H\), \(0 < m \leq \rho_0 \leq M<\infty\), then for any $T>0$, the sequence of solutions $\{\mathbf{u}^\varepsilon, \rho^\varepsilon\mathbf{u}^\varepsilon\}_{\varepsilon\in (0,1)}$ has the following uniform estimates of $\varepsilon$
	\begin{align}\label{3.1}
		\mathrm{E}\left(\sup_{0\leq t\leq T}\|\sqrt{\rho^{\varepsilon}(t)}\mathbf{u}^\varepsilon(t)\|_{L^2(\mathcal{O})}^2\right)
		+\mathrm{E}\left(\int_{0}^{T}\|\nabla \mathbf{u}^\varepsilon(t)\|_{L^2(\mathcal{O})}^2dt\right)\lesssim_{m, c_2, c_4,T, \kappa, \rho_0, \mathbf{u}_0}  C,
	\end{align}
	and for any $p\geq 2$
	\begin{align}\label{3.2}
		\mathrm{E}\left(\sup_{0\leq t\leq T}\|\sqrt{\rho^{\varepsilon}(t)}\mathbf{u}^\varepsilon(t)\|_{L^2(\mathcal{O})}^p\right)+\mathrm{E}\left(\int_{0}^{T}\|\mathbf{u}^\varepsilon(t)\|_H^{p-2}\|\nabla \mathbf{u}^\varepsilon(t)\|_{L^2(\mathcal{O})}^2dt\right)\lesssim_{m, p, c_2, c_4, T, \kappa, \rho_0, \mathbf{u}_0}  C,
	\end{align}
and
\begin{equation}\label{prior3.45}
	\mathrm{E} \left( \int_0^{T} \| \nabla \mathbf{u}^{\varepsilon}(t) \|^2_{L^2(\mathcal{O})} \, dt \right)^p\lesssim_{m,p, c_2, c_4, T, \kappa, \rho_0, \mathbf{u}_0}  C,
	\end{equation}
	where the positive constant $C$ is independent of  $\varepsilon$. Furthermore, we have

$$
\mathrm{E}\left(\sup_{0\leq t\leq T}\|\rho^{\varepsilon}(t)\mathbf{u}^\varepsilon(t)\|_{L^2(\mathcal{O})}^p\right)+\mathrm{E}\left(\sup_{0\leq t\leq T}\|\mathbf{u}^\varepsilon(t)\|_{H}^p\right)\lesssim_{m, p, c_2, c_4, T, \kappa, \rho_0, \mathbf{u}_0}  C,
$$
where the positive constant $C$ is independent of  $\varepsilon$.
\end{lemma}
\begin{proof}
	Since
$$
\frac{d}{dt} \int_\mathcal{O} |\sqrt{\rho^{\varepsilon}(t)} \mathbf{u}^{\varepsilon}(t)|^2 \, dx=\int_\mathcal{O} \mathbf{u}^{\varepsilon}(t) \frac{d(\rho^{\varepsilon}(t) \mathbf{u}^{\varepsilon}(t))}{dt} \, dx+\int_\mathcal{O} \rho^{\varepsilon}(t) \mathbf{u}^{\varepsilon}(t)\frac{d\mathbf{u}^{\varepsilon}(t)}{dt} \, dx
$$
$$
=\int_\mathcal{O} \mathbf{u}^{\varepsilon}(t) \frac{d(\rho^{\varepsilon}(t) \mathbf{u}^{\varepsilon}(t))}{dt} \, dx+\frac{1}{2}\int_\mathcal{O} \rho^{\varepsilon}(t) \frac{d|\mathbf{u}^{\varepsilon}(t)|^2}{dt} \, dx
$$
$$
=\int_\mathcal{O} \mathbf{u}^{\varepsilon}(t) \frac{d(\rho^{\varepsilon}(t) \mathbf{u}^{\varepsilon}(t))}{dt}\, dx+\frac{1}{2}\frac{d}{dt}\int_\mathcal{O} |\sqrt{\rho^{\varepsilon}(t)} \mathbf{u}^{\varepsilon}(t)|^2 \, dx-\frac{1}{2}\int_{\mathcal{O}}|\mathbf{u}^{\varepsilon}(t)|^2 \frac{d\rho^\varepsilon(t)}{dt} dx.
$$
Hence, we obtain
$$
\frac{d}{dt} \int_\mathcal{O} |\sqrt{\rho^{\varepsilon}(t)} \mathbf{u}^{\varepsilon}(t)|^2 \, dx=2\int_\mathcal{O} \mathbf{u}^{\varepsilon}(t) \frac{d(\rho^{\varepsilon}(t) \mathbf{u}^{\varepsilon}(t))}{dt} \, dx-\int_{\mathcal{O}}|\mathbf{u}^{\varepsilon}(t)|^2 \frac{d\rho^\varepsilon(t)}{dt} dx,
$$
which along with It\^ o's formula gives
	\begin{align}\label{priori2.12}
		&d \int_\mathcal{O} |\sqrt{\rho^{\varepsilon}(s)} \mathbf{u}^{\varepsilon}(s)|^2 \, dx\nonumber\\
 &= -\int_\mathcal{O}\mathbf{u}^{\varepsilon}(s) \mathbf{u}^{\varepsilon}(s) \frac{\partial \rho^{\varepsilon}(s)}{\partial s} \, dx  ds - 2(A^\varepsilon \mathbf{u}^\varepsilon(s), \mathbf{u}^\varepsilon(s))_{V'\times V}ds \nonumber \\
		&\quad + 2 \int_\mathcal{O} (\rho^{\varepsilon}(s) \mathbf{u}^{\varepsilon}(s) \otimes \mathbf{u}^{\varepsilon}(s)) : \nabla \mathbf{u}^{\varepsilon}(s) \, dx ds  + 2 \int_\mathcal{O}  \mathbf{u}^{\varepsilon}(s) f\left(\frac{x}{\varepsilon}, s, \mathbf{u}^{\varepsilon}(s)\right) \, dxds\nonumber \\
		&\quad + 2\sum_{k\geq 1}\int_\mathcal{O}  \mathbf{u}^{\varepsilon}(s) g(\mathbf{u}^{\varepsilon}(s))Q^\frac{1}{2}\mathbf{e}_k \, dx dW_k+ \| g(\mathbf{u}^{\varepsilon}(s))\|^2_{L_2(H;H)} ds.
	\end{align}
	  We can infer that
	\begin{equation}\label{3.4}
		0 = \int_\mathcal{O} \text{div}(\mathbf{u}^{\varepsilon}(s) \mathbf{u}^{\varepsilon} (s)\rho^{\varepsilon} (s)\mathbf{u}^{\varepsilon}(s)) \, dx = \int_\mathcal{O} \left[\mathbf{u}^{\varepsilon}(s) \mathbf{u}^{\varepsilon}(s) \text{div}(\rho^{\varepsilon}(s) \mathbf{u}^{\varepsilon}(s)) + \rho^{\varepsilon} (s)\mathbf{u}^{\varepsilon} (s) \nabla(\mathbf{u}^{\varepsilon}(s))^2\right] \, dx.
	\end{equation}
The first equality follows from the Dirichlet boundary condition of $\mathbf{u}^{\varepsilon}$. By the incompressible condition \( \text{div} \, \mathbf{u}^{\varepsilon} = 0 \), it follows from the continuity equation and \eqref{3.4} that
		$$-\int_\mathcal{O} \mathbf{u}^{\varepsilon}(s) \mathbf{u}^{\varepsilon}(s) \frac{\partial \rho^{\varepsilon}(s)}{\partial s} \, dx 
= \int_\mathcal{O} \mathbf{u}^{\varepsilon}(s) \mathbf{u}^{\varepsilon}(s) (\mathbf{u}^{\varepsilon}(s) \cdot \nabla) \rho^{\varepsilon}(s) \, dx $$
\begin{equation}\label{priori2.14}
= -2 \int_\mathcal{O} \rho^{\varepsilon}(s) \mathbf{u}^{\varepsilon}(s) (\mathbf{u}^{\varepsilon}(s) \cdot \nabla) \mathbf{u}^{\varepsilon}(s)dx = -2 \int_\mathcal{O} (\rho^{\varepsilon}(s) \mathbf{u}^{\varepsilon}(s) \otimes \mathbf{u}^{\varepsilon}(s)) : \nabla \mathbf{u}^{\varepsilon}(s) \, dx.
	\end{equation}
The uniform ellipticity condition  \eqref{1.2} of operator \( A^\varepsilon \)
 leads to
\begin{align}\label{1.6}
	-2(A^\varepsilon \mathbf{u}^\varepsilon(s), \mathbf{u}^\varepsilon(s))_{V'\times V}=-2\sum_{i,j=1}^d\int_{\mathcal{O}}a^\varepsilon_{i,j}\partial_{x_i}\mathbf{u}^\varepsilon (s) \partial_{x_j}\mathbf{u}^\varepsilon (s)dx
\leq -2\kappa \|\nabla \mathbf{u}^\varepsilon(s)\|_{L^2(\mathcal{O})}^2.
\end{align}
Using  $(\mathbf{A.2})$, $(\mathbf{A.4})$, we see
\begin{align}\label{3.11}
2\int_\mathcal{O}  \mathbf{u}^{\varepsilon}(s) f\left(\frac{x}{\varepsilon}, s, \mathbf{u}^{\varepsilon}(s)\right) \, dx \leq c_2(1+3\|\mathbf{u}^{\varepsilon}(s)\|^2_{H})\leq c_2\left(1+\frac{3\|\sqrt{\rho^{\varepsilon}(s)}\mathbf{u}^{\varepsilon}(s)\|^2_{L^2(\mathcal{O})}}{m}\right),
\end{align}
and
\begin{align}\label{3.12}
\| g(\mathbf{u}^{\varepsilon}(s))\|^2_{L_2(H;H)}\leq c_4(1+\|\mathbf{u}^\varepsilon(s)\|_{H}^2)\leq c_4\left(1+\frac{\|\sqrt{\rho^{\varepsilon}(s)}\mathbf{u}^{\varepsilon}(s)\|^2_{L^2(\mathcal{O})}}{m}\right).
\end{align}
By \eqref{priori2.12} and \eqref{priori2.14}-\eqref{3.12}, we obtain  for all \( s \in [0, t] \)
	$$
		 \| \sqrt{\rho^{\varepsilon}(s)} \mathbf{u}^{\varepsilon}(s) \|^2_{L^2(\mathcal{O})} + 2{\kappa} \int_0^s \| \nabla \mathbf{u}^{\varepsilon}(r) \|^2_{L^2(\mathcal{O})} \, dr
		\leq  \| \sqrt{\rho_0} \mathbf{u}_0 \|^2_{L^2(\mathcal{O})}  $$
\begin{eqnarray}\label{prior3.15}
		+ C\int_{0}^{s}\left(1+\|\sqrt{\rho^{\varepsilon}(r)}\mathbf{u}^{\varepsilon}(r)\|^2_{L^2(\mathcal{O})}\right)dr+ 2\int_{0}^{s} (g(\mathbf{u}^{\varepsilon}(r)) \,  dW, \mathbf{u}^{\varepsilon}(r)),
	\end{eqnarray}
where $C=C(m)>0$. Taking the supremum of time over the interval \( [0, t] \) on both sides of \eqref{prior3.15}, and then applying the expectation, we arrive at
	\begin{align}\label{priori3.16}
		&\mathrm{E} \left(\sup_{0 \leq s \leq t} \| \sqrt{\rho^{\varepsilon}(s)} \mathbf{u}^{\varepsilon}(s) \|^2_{L^2(\mathcal{O})}\right) + {2}{\kappa} \mathrm{E}\left( \int_0^{t} \| \nabla \mathbf{u}^{\varepsilon}(s) \|^2_{L^2(\mathcal{O})} \, ds\right) \nonumber \\
		&\leq \| \sqrt{\rho_0} \mathbf{u}_0 \|^2_{L^2(\mathcal{O})} +
 C\mathrm{E}\left(\int_{0}^{t}\left(1+\|\sqrt{\rho^{\varepsilon}(s)}\mathbf{u}^{\varepsilon}(s)\|^2_{L^2(\mathcal{O})}\right)ds\right) \nonumber \\
		&\quad + 2 \mathrm{E} \left(\sup_{0 \leq s \leq t} \left|\int_{0}^{s} (g(\mathbf{u}^{\varepsilon}(r)) \,  dW, \mathbf{u}^{\varepsilon}(r))\right|\right).
	\end{align}
 The $(\mathbf{A.4})$ combined with the Burkholder-Davis-Gundy inequality imply
	$$
		2 \mathrm{E} \left(\sup_{0 \leq s \leq t} \left|\int_{0}^{s} (g(\mathbf{u}^{\varepsilon}(r)) \,  dW, \mathbf{u}^{\varepsilon}(r))\right|\right)$$
$$
 \leq C\mathrm{E} \left( \int_0^{t}  \sum_{k\geq 1}(g( \mathbf{u}^{\varepsilon}(s))Q^\frac{1}{2}\mathbf{e}_k, \mathbf{u}^{\varepsilon}(s) )^2 \, ds \right)^{\frac{1}{2}} $$
		$$\leq C\mathrm{E} \left( \int_0^{t}  c_4(1 + \| \mathbf{u}^{\varepsilon}(s) \|^2_{H}) \|  \mathbf{u}^{\varepsilon}(s) \|^2_{H} \, ds \right)^{\frac{1}{2}} $$
		$$\leq C\sqrt{c_4}\mathrm{E}  \left( \sup_{0 \leq s \leq t} \|  \mathbf{u}^{\varepsilon}(s) \|_{H}\int_0^{t}  (1 + \| \mathbf{u}^{\varepsilon}(s) \|^2_{H}) \, ds \right)^{\frac{1}{2}}$$
\begin{align}\label{3.13}
		&\leq \frac{1}{2}\mathrm{E}\left(\sup_{0 \leq s \leq t} \|  \sqrt{\rho^\varepsilon(s)} \mathbf{u}^{\varepsilon}(s) \|^2_{L^2(\mathcal{O})}\right) + C(m, c_4)\mathrm{E} \left(\int_0^{t} \left(1 + \| \sqrt{\rho^\varepsilon(s)} \mathbf{u}^{\varepsilon}(s) \|^2_{L^2(\mathcal{O})}\right) \, ds\right).
	\end{align}
	Substituting estimate \eqref{3.13} into inequality \eqref{priori3.16}, we have for every $t\in [0,T]$
	\begin{align}
		\mathrm{E}\left( \sup_{0 \leq s \leq t} \left\| \sqrt{\rho^{\varepsilon}(s)} \mathbf{u}^{\varepsilon}(s) \right\|^2_{L^2(\mathcal{O})} \right)+ {4}{\kappa} \mathrm{E} \left(\int_0^{t} \|\nabla \mathbf{u}^{\varepsilon}(s)\|^2_{L^2(\mathcal{O})} \, ds\right) \nonumber\\
		\leq  2\left\| \sqrt{\rho_0} \mathbf{u}_0 \right\|^2_{L^2(\mathcal{O})} + C(m,c_2,c_4)\mathrm{E} \left(\int_0^{t} \left(1 + \| \sqrt{\rho^\varepsilon(s)} \mathbf{u}^{\varepsilon}(s) \|^2_{L^2(\mathcal{O})}\right) \, ds\right).
	\end{align}
	By Gronwall's inequality, we have for every $t\in [0,T]$
	\begin{equation}\label{prior3.28}
		\mathrm{E} \left(\sup_{0 \leq s \leq t} \left\| \sqrt{\rho^{\varepsilon}(s)} \mathbf{u}^{\varepsilon}(s) \right\|^2_{L^2(\mathcal{O})}\right) +  \mathrm{E} \left(\int_0^{t} \|\nabla \mathbf{u}^{\varepsilon}(s)\|^2_{L^2(\mathcal{O})} \, ds\right) \lesssim_{m,\kappa, c_2, c_4, T, \rho_0,\mathbf{u}_0 } C.
	\end{equation}

	Applying It\^ o's formula for \( p \geq 2 \), integrating of time over \( [0, s] \)  we have
	\begin{eqnarray}\label{prior3.30}
		&& \| \sqrt{\rho^{\varepsilon}(s)} \mathbf{u}^{\varepsilon}(s) \|^p_{L^2(\mathcal{O})} + p\int_0^s \| \sqrt{\rho^{\varepsilon}(r)} \mathbf{u}^{\varepsilon}(r) \|^{p-2}_{L^2(\mathcal{O})} (A^\varepsilon \mathbf{u}^\varepsilon(r), \mathbf{u}^\varepsilon(r))_{V'\times V} \, dr \nonumber\\
		&&= \| \sqrt{\rho_0} \mathbf{u}_0\|^p_{L^2(\mathcal{O})} + \frac{p}{2} \int_0^s \|  \sqrt{\rho^{\varepsilon}(r)}\mathbf{u}^{\varepsilon}(r) \|^{p-2}_{L^2(\mathcal{O})} \| g(\mathbf{u}^{\varepsilon}(r))\|^2_{L_2(H;H)}dr \nonumber\\
		&&\quad+ p \int_0^s \| \sqrt{\rho^{\varepsilon}(r)} \mathbf{u}^{\varepsilon}(r) \|^{p-2}_{L^2(\mathcal{O})} \int_\mathcal{O}  \mathbf{u}^{\varepsilon}(r) f\left(\frac{x}{\varepsilon}, r, \mathbf{u}^{\varepsilon}(r)\right) \, dx dr \nonumber\\
&&\quad
+p\int_0^s \|  \sqrt{\rho^{\varepsilon}(r)}\mathbf{u}^{\varepsilon}(r) \|^{p-2}_{L^2(\mathcal{O})}\sum_{k\geq 1}\int_\mathcal{O}  \mathbf{u}^{\varepsilon}(r)g(\mathbf{u}^{\varepsilon}(r))Q^\frac{1}{2}\mathbf{e}_k  \, dx dW_k
\nonumber\\
		&&\quad+ \frac{p(p-2)}{4} \int_0^s \| \sqrt{\rho^{\varepsilon}(r)} \mathbf{u}^{\varepsilon}(r) \|^{p-4}_{L^2(\mathcal{O})} \left(\sum_{k\geq 1}\int_\mathcal{O}  \mathbf{u}^{\varepsilon}(r)g(\mathbf{u}^{\varepsilon}(r))Q^\frac{1}{2}\mathbf{e}_k \, dx\right)^2  \, dr.
	\end{eqnarray}
We shall estimate each term of the equation \eqref{prior3.30} after taking the supremum up to time \( t \) and applying the expectation on both sides.
For the second term on the left-hand side of \eqref{prior3.30}, by the uniform ellipticity condition  \eqref{1.2} of operator \( A^\varepsilon \)
 we have
$$p\int_0^s \| \sqrt{\rho^{\varepsilon}(r)} \mathbf{u}^{\varepsilon}(r) \|^{p-2}_{L^2(\mathcal{O})} (A^\varepsilon \mathbf{u}^\varepsilon(r), \mathbf{u}^\varepsilon(r))_{V'\times V} \, dr
$$
\begin{align}\label{3.17*}
\geq p\kappa\int_0^s \| \sqrt{\rho^{\varepsilon}(r)} \mathbf{u}^{\varepsilon}(r) \|^{p-2}_{L^2(\mathcal{O})}\|\nabla \mathbf{u}^{\varepsilon}(r)\|^2_{L^2(\mathcal{O})} \, dr.
\end{align}
For the second and fifth terms on the right-hand side of \eqref{prior3.30}, using $(\mathbf{A.4})$ we see
$$
\frac{p}{2} \int_0^s \|  \sqrt{\rho^{\varepsilon}(r)}\mathbf{u}^{\varepsilon}(r) \|^{p-2}_{L^2(\mathcal{O})}  \| g(\mathbf{u}^{\varepsilon}(r))\|^2_{L_2(H;H)}dr
$$
$$
\leq \frac{pc_4}{2} \int_0^s \|  \sqrt{\rho^{\varepsilon}(r)}\mathbf{u}^{\varepsilon}(r) \|^{p-2}_{L^2(\mathcal{O})} (1+\|\mathbf{u}^{\varepsilon}(r)\|_{H}^2 )\, dr
$$
$$
\leq \frac{pc_4}{2}\int_0^s \|  \sqrt{\rho^{\varepsilon}(r)}\mathbf{u}^{\varepsilon}(r) \|^{p-2}_{L^2(\mathcal{O})} \left(1+\frac{\|\sqrt{\rho^{\varepsilon}(r)}\mathbf{u}^{\varepsilon}(r)\|_{L^2(\mathcal{O})}^2}{m} \right)\, dr
$$
\begin{align}
\lesssim_{m,p, c_4} C\int_0^s \left(1+\|  \sqrt{\rho^{\varepsilon}(r)}\mathbf{u}^{\varepsilon}(r) \|^{p}_{L^2(\mathcal{O})}\right)\, dr,
\end{align}
and
$$
\frac{p(p-2)}{4} \int_0^s \| \sqrt{\rho^{\varepsilon}(r)} \mathbf{u}^{\varepsilon}(r) \|^{p-4}_{L^2(\mathcal{O})} \left(\sum_{k\geq 1}\int_\mathcal{O}  \mathbf{u}^{\varepsilon}(r)g(\mathbf{u}^{\varepsilon}(r))Q^\frac{1}{2}\mathbf{e}_k \, dx\right)^2  \, dr
$$
$$
\leq \frac{p(p-2)c_4}{4}\int_0^s \| \sqrt{\rho^{\varepsilon}(r)} \mathbf{u}^{\varepsilon}(r) \|^{p-4}_{L^2(\mathcal{O})} \|\mathbf{u}^{\varepsilon}(r)\|_{H}^2(1+
\|\mathbf{u}^{\varepsilon}(r)\|_{H}^2)  \, dr
$$
$$
\lesssim_{m,p, c_4} C \int_0^s \| \sqrt{\rho^{\varepsilon}(r)} \mathbf{u}^{\varepsilon}(r) \|^{p-2}_{L^2(\mathcal{O})}\left(1+\|\sqrt{\rho^{\varepsilon}(r)}\mathbf{u}^{\varepsilon}(r)\|_{L^2(\mathcal{O})}^2\right)  \, dr
$$
\begin{align}\label{3.18}
\lesssim_{m,p,c_4} C\int_0^s \left(1+\|  \sqrt{\rho^{\varepsilon}(r)}\mathbf{u}^{\varepsilon}(r) \|^{p}_{L^2(\mathcal{O})}\right)\, dr.
\end{align}
For the third term on the right-hand side of \eqref{prior3.30}, using $(\mathbf{A.2})$ we see
$$
  p\int_0^s \| \sqrt{\rho^{\varepsilon}(r)} \mathbf{u}^{\varepsilon}(r) \|^{p-2}_{L^2(\mathcal{O})} \int_\mathcal{O}  \mathbf{u}^{\varepsilon}(r) f\left(\frac{x}{\varepsilon}, r, \mathbf{u}^{\varepsilon}(r)\right) \, dx \, dr
$$
$$
\leq  pc_2\int_0^s \| \sqrt{\rho^{\varepsilon}(r)} \mathbf{u}^{\varepsilon}(r) \|^{p-2}_{L^2(\mathcal{O})}\int_{\mathcal{O} } |\mathbf{u}^{\varepsilon}(r)|\left(1+ |\mathbf{u}^{\varepsilon}(r)|\right) dx\, dr
$$
$$
\leq  pc_2\int_0^s \| \sqrt{\rho^{\varepsilon}(r)} \mathbf{u}^{\varepsilon}(r) \|^{p-2}_{L^2(\mathcal{O})}\left(\frac{|\mathcal{O}|}{2}+\frac{3\|\mathbf{u}^{\varepsilon}(r)\|^2_{H}}{2}\right)\, dr
$$
\begin{align}
\lesssim_{m,p,c_2, T} C\int_0^s \| \sqrt{\rho^{\varepsilon}(r)} \mathbf{u}^{\varepsilon}(r) \|^{p}_{L^2(\mathcal{O})} \, dr.
\end{align}
For the fourth term on the right-hand side of \eqref{prior3.30}, using the Burkholder-Davis-Gundy inequality, we have
	$$
		p\mathrm{E}\left( \sup_{0 \leq s \leq t} \left| \int_0^s \| \sqrt{\rho^{\varepsilon}(r)} \mathbf{u}^{\varepsilon}(r) \|^{p-2}_{_{L^2(\mathcal{O})}} \sum_{k\geq 1}\int_\mathcal{O}  \mathbf{u}^{\varepsilon}(r)g(\mathbf{u}^{\varepsilon}(r))Q^\frac{1}{2}\mathbf{e}_k  \, dx dW_k\right|\right)
$$
		$$\leq C(p)\mathrm{E}\left[ \sup_{0 \leq s \leq t} \| \sqrt{\rho^{\varepsilon}(s)} \mathbf{u}^{\varepsilon}(s) \|^{p-2}_{_{L^2(\mathcal{O})}} \left( \int_0^{t}\sum_{k\geq 1}(g( \mathbf{u}^{\varepsilon}(r))Q^\frac{1}{2}\mathbf{e}_k, \mathbf{u}^{\varepsilon}(r) )^2  \, dr \right)^{1/2} \right]
	$$
		$$\leq C(p,c_4) \mathrm{E} \left[ \sup_{0 \leq s \leq t} \| \sqrt{\rho^{\varepsilon}(s)} \mathbf{u}^{\varepsilon}(s) \|^{p-2}_{_{L^2(\mathcal{O})}} \left( \int_0^{t} \|\mathbf{u}^{\varepsilon}(r) \|^2_{H} (1+\|\mathbf{u}^{\varepsilon}(r) \|^2_{H})\, dr \right)^{1/2} \right]
	$$
$$
	\leq \frac{1}{2}\mathrm{E}\left( \sup_{0 \leq s \leq t} \| \sqrt{\rho^{\varepsilon}(s)} \mathbf{u}^{\varepsilon}(s) \|^{p}_{_{L^2(\mathcal{O})}}\right)  + C(p,c_4)\mathrm{E} \left( \int_0^{t} \|\mathbf{u}^{\varepsilon}(r) \|^2_{H} (1+\|\mathbf{u}^{\varepsilon}(r) \|^2_{H}) \, dr \right)^\frac{p}{4}$$
\begin{align}\label{prior3.34}
\quad\leq \frac{1}{2}\mathrm{E}\left( \sup_{0 \leq s \leq t} \| \sqrt{\rho^{\varepsilon}(s)} \mathbf{u}^{\varepsilon}(s) \|^{p}_{_{L^2(\mathcal{O})}}\right)  + C(p, c_4, m)\mathrm{E} \left(\int_{0}^{t} \left(1+\| \sqrt{\rho^{\varepsilon}(r)}\mathbf{u}^{\varepsilon}(r) \|^p_{L^2(\mathcal{O})}\right) \, dr \right).
	\end{align}
Using \eqref{prior3.30}-\eqref{prior3.34}, we arrive at
$$
\mathrm{E} \left(\sup_{0 \leq s \leq t} \| \sqrt{\rho^{\varepsilon}(s)} \mathbf{u}^{\varepsilon}(s) \|^p_{L^2(\mathcal{O})}\right) + 2p\kappa \mathrm{E} \left(\int_0^{t} \| \sqrt{\rho^{\varepsilon}(s)} \mathbf{u}^{\varepsilon}(s) \|^{p-2}_{L^2(\mathcal{O})} \| \nabla \mathbf{u}^{\varepsilon}(s) \|^2_{L^2(\mathcal{O})} \, ds\right)$$
\begin{eqnarray*}
	\leq\| \sqrt{\rho_{0}} \mathbf{u}_0 \|^p_{L^2(\mathcal{O})} + C(p, c_2, c_4, m,\kappa, T)\mathrm{E} \left(\int_{0}^{t} \left(1+\| \sqrt{\rho^{\varepsilon}(s)}\mathbf{u}^{\varepsilon}(s) \|^p_{L^2(\mathcal{O})}\right) \, ds \right).
	\end{eqnarray*}
Using Gronwall's lemma, we have for any $t\in [0,T]$	
\begin{align}\label{prior3.41}
&\mathrm{E} \left(\sup_{0 \leq s \leq t} \| \sqrt{\rho^{\varepsilon}(s)} \mathbf{u}^{\varepsilon}(s) \|^p_{L^2(\mathcal{O})}\right) + \mathrm{E} \left(\int_0^{t} \| \sqrt{\rho^{\varepsilon}(s)} \mathbf{u}^{\varepsilon}(s) \|^{p-2}_{L^2(\mathcal{O})} \| \nabla \mathbf{u}^{\varepsilon}(s) \|^2_{L^2(\mathcal{O})} \, ds\right)\nonumber\\
&\lesssim_{m,c_2, c_4,\kappa,p, T, \rho_{0}, \mathbf{u}_0} C.
\end{align}

 If  we take the power \( p \geq 1 \) in \eqref{prior3.15}, by a same way we could have
	\begin{equation*}
	\mathrm{E} \left( \int_0^{t} \| \nabla \mathbf{u}^{\varepsilon}(s) \|^2_{L^2(\mathcal{O})} \, ds \right)^p \lesssim_{m, c_2, c_4, \kappa,p, T, \rho_{0}, \mathbf{u}_0 }  C.
	\end{equation*}
It follows from \eqref{prior3.41} and Lemma \ref{lem3.1*} that for any $p\geq 2$

\begin{align}\label{3.22}
\rho^{\varepsilon} \mathbf{u}^{\varepsilon} \in L^p(\Omega; L^\infty(0, T; L^2(\mathcal{O}))),
\end{align}
and
$$ \mathbf{u}^{\varepsilon} \in L^p(\Omega; L^\infty(0, T; H)),$$
uniformly in $\varepsilon$ as desired.
\end{proof}

	Next we focus on the temporal regularity of $\rho^\varepsilon, \mathbf{u}^\varepsilon$. 

	

\begin{lemma}\label{lem3.2}If $\mathbf{(A.2)}$, $\mathbf{(A.4)}$ hold,  then there exists a constant $C>0$ which is independent of $\varepsilon, \theta$ such that the time increment of $\mathbf{u}^\varepsilon$ satisfies
	\[
	\mathrm{E} \left(\sup_{\theta\in (0,T)}\left(\theta^{-\frac{1}{4}}\int^{T - \theta}_{0} \| \mathbf{u}^{\varepsilon}(t + \theta) - \mathbf{u}^{\varepsilon}(t) \|^2_{H} \, dt\right)\right) \lesssim_{m, M, c_2, c_4, \kappa, T}  C.
	\]
Moreover, we have for  any $\beta\in  [0,\frac{1}{2})$ and $p>2$, $$\rho^{\varepsilon} \mathbf{u}^{\varepsilon}\in L^2(\Omega; W^{\beta,p}(0,T; W^{-3,2}(\mathcal{O})))$$
and
$$ \rho^{\varepsilon} \in L^p(\Omega; W^{1,\infty}(0, T; W^{-1,2}(\mathcal{O}))),$$
uniformly in $\varepsilon$.
	\end{lemma}
	\begin{proof} According to \eqref{3.22}, we have
	\[
	\nabla(\rho^{\varepsilon} \mathbf{u}^{\varepsilon}) \in L^p(\Omega; L^\infty(0, T; W^{-1,2}(\mathcal{O}))),
	\]
uniformly in $\varepsilon$. Then, from the continuity equation
$
\partial_t\rho^\varepsilon(t)+{\rm div }(\rho^\varepsilon(t) \mathbf{u}^\varepsilon(t))=0,
$
 we see
\begin{align*}\partial_t \rho^{\varepsilon} \in L^p(\Omega; L^\infty(0, T; W^{-1,2}(\mathcal{O}))),\end{align*}
which together with Lemma \ref{lem3.1*} leads to 
$$ \rho^{\varepsilon} \in L^p(\Omega; W^{1,\infty}(0, T; W^{-1,2}(\mathcal{O}))),$$
uniformly in $\varepsilon$.

We next establish the temporal regularity of $\mathbf{u}^\varepsilon$. Note that from the momentum equation, we get
	\begin{eqnarray}
		&&\mathrm{E} \left(\sup_{\theta\in (0,T)}\left(\theta^{-\frac{1}{2}}\int_{0}^{T-\theta} \|\rho^{\varepsilon}(t+\theta) \mathbf{u}^{\varepsilon}(t+\theta) - \rho^{\varepsilon}(t) \mathbf{u}^{\varepsilon}(t)\|^2_{V'} \, dt\right)\right) \nonumber\\
&&= \mathrm{E} \left(\sup_{\theta\in (0,T)}\left(\theta^{-\frac{1}{2}}\int^{T - \theta}_{0} \left\|\int^{t + \theta}_{t} \frac{d(\rho^{\varepsilon}(s) \mathbf{u}^{\varepsilon}(s))}{ds}\, ds\right\|^2_{V'} \, dt\right)\right) \nonumber\\
		&&\leq  c \mathrm{E}  \left(\sup_{\theta\in (0,T)}\theta^{-\frac{1}{2}}\int_{ 0}^{T-\theta}
		\Bigg( \Bigg\|-\int^{t + \theta}_{t} \text{div}(\rho^{\varepsilon}(s) \mathbf{u}^{\varepsilon}(s) \otimes \mathbf{u}^{\varepsilon}(s)) \, ds\right\|^2_{V'}+ \left\|-\int^{t + \theta}_{t} A^\varepsilon\mathbf{u}^{\varepsilon}(s) \, ds\right\|^2_{V'}
		\nonumber\\
		&&\quad+  \left\|\int^{t + \theta}_{t} f\left(\frac{x}{\varepsilon}, s, \mathbf{u}^{\varepsilon}(s)\right) \, ds\right\|^2_{V'}  +\left\|\int^{t + \theta}_{t} g( \mathbf{u}^{\varepsilon}(s)) \, dW\right\|^2_{V'} \Bigg)dt\Bigg).
	\end{eqnarray}	
For the advection term, using the Gagliardo-Nirenberg inequality
$$\|\mathbf{u}^{\varepsilon}\|_{L^4(\mathcal{O})}^2 \leq C\|\mathbf{u}^{\varepsilon}\|^{\frac{4-d}{2}}_{H}\|\nabla\mathbf{u}^{\varepsilon}\|_{L^2(\mathcal{O})}^\frac{d}{2},
$$
we see
	$$
\left\|-\int^{t + \theta}_{t} \text{div}(\rho^{\varepsilon}(s) \mathbf{u}^{\varepsilon}(s) \otimes \mathbf{u}^{\varepsilon}(s)) \, ds\right\|^2_{V'}
$$
	\[
	 = \left(\sup_{\phi \in V; \|\phi\|_V = 1}  \left( \int^{t + \theta}_{t}- (\text{div}(\rho^{\varepsilon}(s) \mathbf{u}^{\varepsilon}(s) \otimes \mathbf{u}^{\varepsilon}(s)) ,  \phi)\, ds \right)\right)^2
	\]
$$
\leq \left(\sup_{\phi \in V; \|\phi\|_V = 1}  \left( \int^{t + \theta}_{t} \|\rho^{\varepsilon}(s) \mathbf{u}^{\varepsilon}(s) \otimes \mathbf{u}^{\varepsilon}(s)\|_{L^2(\mathcal{O})} \|\phi\|_{V} ds \right)\right)^2
$$

$$\leq  \left(\int^{t + \theta}_{t} \|\rho^{\varepsilon}(s) \mathbf{u}^{\varepsilon}(s) \otimes \mathbf{u}^{\varepsilon}(s)\|_{L^2(\mathcal{O})}  \, ds\right)^2
$$
$$
\leq \|\rho^\varepsilon\|^2_{L^\infty(\mathcal{O}_t)}\left(\int^{t + \theta}_{t} \|\mathbf{u}^{\varepsilon}(s)\|_{L^4(\mathcal{O})}^2  \, ds\right)^2
$$
$$
\leq C\|\rho^\varepsilon\|^2_{L^\infty(\mathcal{O}_t)}\left(\int^{t + \theta}_{t} \|\mathbf{u}^{\varepsilon}(s)\|^{\frac{4-d}{2}}_{H}\|\nabla\mathbf{u}^{\varepsilon}(s)\|_{L^2(\mathcal{O})}^\frac{d}{2}  \, ds\right)^2
$$
$$
\leq C\|\rho^\varepsilon\|^2_{L^\infty(\mathcal{O}_t)}\|\mathbf{u}^{\varepsilon}\|^{4-d}_{L^\infty(0,T; H)}\left(\int^{t + \theta}_{t} \|\nabla\mathbf{u}^{\varepsilon}(s)\|_{L^2(\mathcal{O})}^\frac{d}{2}  \, ds\right)^2
$$
$$
\leq C\theta^{\frac{4-d}{2}}\|\rho^\varepsilon\|^2_{L^\infty(\mathcal{O}_t)}\|\mathbf{u}^{\varepsilon}\|^{4-d}_{L^\infty(0,T; H)}\left(\int^{t + \theta}_{t} \|\nabla\mathbf{u}^{\varepsilon}(s)\|_{L^2(\mathcal{O})}^2  \, ds\right)^2,
$$
$$
\leq CM^2\theta^{\frac{4-d}{2}}\|\mathbf{u}^{\varepsilon}\|^{4-d}_{L^\infty(0,T; H)}\left(\int^{t + \theta}_{t} \|\nabla\mathbf{u}^{\varepsilon}(s)\|_{L^2(\mathcal{O})}^2  \, ds\right)^2,
$$
which together with \eqref{prior3.45} leads to
$$
\mathrm{E} \left(\sup_{\theta\in (0,T)}\left(\theta^{-\frac{1}{2}}\int_{0}^{T-\theta}\left\|-\int^{t + \theta}_{t} \text{div}(\rho^{\varepsilon}(s) \mathbf{u}^{\varepsilon}(s) \otimes \mathbf{u}^{\varepsilon}(s)) \, ds\right\|^2_{V'}dt\right)\right)
$$
$$
\leq CM^2\mathrm{E} \left(\sup_{\theta\in (0,T)}\left(\theta^{\frac{3-d}{2}}\int_{0}^{T-\theta}\|\mathbf{u}^{\varepsilon}\|^{4-d}_{L^\infty(0,T; H)}\left(\int^{t + \theta}_{t} \|\nabla\mathbf{u}^{\varepsilon}(s)\|_{L^2(\mathcal{O})}^2  \, ds\right)^2dt\right)\right)
$$
\begin{align}\label{3.25}
\qquad\leq CM^2\left(\mathrm{E}\left(\|\mathbf{u}^{\varepsilon}\|^{2(4-d)}_{L^\infty(0,T; H)}\right)\right)^\frac{1}{2}\left(\mathrm{E}\left(\int^{T}_{0} \|\nabla\mathbf{u}^{\varepsilon}(s)\|_{L^2(\mathcal{O})}^2  \, ds\right)^4\right)^\frac{1}{2} \lesssim_{m,M,c_2,c_4,\kappa,T} C.
\end{align}
	For the diffusion  term, we get
	$$
	\mathrm{E}  \left(\sup_{\theta\in (0,T)}\left(\theta^{-\frac{1}{2}}\int^{T - \theta}_{0} \left\|\int^{t + \theta}_{t} A^\varepsilon \mathbf{u}^{\varepsilon}(s) \, ds\right\|^2_{V'}\, dt\right)\right)$$
$$=\mathrm{E} \left(\sup_{\theta\in (0,T)}\left(\theta^{-\frac{1}{2}}\int^{T - \theta}_{0} \sup_{\phi \in V; \|\phi\|_V = 1} \left(\int^{t + \theta}_{t}( A^\varepsilon \mathbf{u}^{\varepsilon}(s), \varphi)_{V'\times V} \, ds\right)^2\, dt\right)\right)
$$
$$
\leq  \mathrm{E} \left(\sup_{\theta\in (0,T)}\left(\theta^{-\frac{1}{2}}\int^{T - \theta}_{0} \left(\int^{t + \theta}_{t}\sum_{i,j=1}^d\|a_{i,j}^\varepsilon \nabla\mathbf{u}^{\varepsilon}(s)\|_{L^2(\mathcal{O})}\, ds\right)^2\, dt\right)\right)
$$
$$
\leq \sum_{i,j=1}^d\|a_{i,j}^\varepsilon \|^2_{L^\infty(\mathcal{O}_t)} \mathrm{E} \left(\sup_{\theta\in (0,T)}\left(\theta^{-\frac{1}{2}}\theta\int^{T - \theta}_{0} \left(\int^{t + \theta}_{t}\| \nabla\mathbf{u}^{\varepsilon}(s)\|^2_{L^2(\mathcal{O})}\, ds\right)\, dt\right)\right)
$$
\begin{align}
	\leq CT^\frac{3}{2}\mathrm{E} \left(\int^{T}_{0} \|\nabla \mathbf{u}^{\varepsilon}(s)\|^2_{L^2(\mathcal{O})} \, ds \right) \lesssim_{m,M,c_2,c_4,\kappa,T} C.
	\end{align}
	For the external force term, using $(\mathbf{A.2})$ we see
	$$
		\mathrm{E} \left(\sup_{\theta\in (0,T)}\left(\theta^{-\frac{1}{2}}\int^{T - \theta}_{0}\left\|\int^{t + \theta}_{t} f\left(\frac{x}{\varepsilon}, s, \mathbf{u}^{\varepsilon}(s)\right) \, ds\right\|^2_{V'}  \, dt\right)\right) $$
$$ \leq  c^2_2 \mathrm{E} \left(\sup_{\theta\in (0,T)}\left(\theta^{-\frac{1}{2}}\int^{T - \theta}_{0} \left( \int^{t + \theta}_{t}  \left( 1 + \|\mathbf{u}^{\varepsilon}(s)\|_{H} \right) \, ds \right)^{2} \, dt\right)\right) $$
\begin{align}\label{3.27}
	\leq  c^2_2T^\frac{3}{2} \mathrm{E} \left(\int_{0}^{T } \left( 1 + \|\mathbf{u}^{\varepsilon}(s)\|^2_{H} \right) \, ds \right)\lesssim_{m, M, c_2, c_4, \kappa,T} C.
	\end{align}
We next deal with the stochastic integral part. Let $G(t)=\int_{0}^{t}g( \mathbf{u}^{\varepsilon}(s)) \, dW$, then $G(t)-G(s)=\int^{t }_{s} g( \mathbf{u}^{\varepsilon}(r)) \, dW$. By the Burkholder-Davis-Gundy inequality and  $(\mathbf{A.4})$ we have for $p> 2$ and $\alpha\in [0, \frac{1}{2})$
\begin{align}\label{sto}
&\mathrm{E}\left(\int_{0}^{T}\int_{0}^{T}\frac{\|G(t)-G(s)\|_{V'}^p}{|t-s|^{\alpha p+1}}dsdt\right)\nonumber\\
&=\mathrm{E}\left(\int_{0}^{T}\int_{0}^{T}\frac{\left(\sup_{\varphi\in V, \|\varphi\|_{V}=1}\left|\left\langle \int_{s}^{t}g(\mathbf{u}^{\varepsilon}(r)) dW, \varphi \right\rangle\right|\right)^p}{|t-s|^{\alpha p+1}}dsdt\right)\nonumber\\
&\leq \mathrm{E}\left(\int_{0}^{T}\int_{0}^{T}\frac{\left(\sup_{\varphi\in V, \|\varphi\|_{V}=1}\left(\left\| \int_{s}^{t}g(\mathbf{u}^{\varepsilon}(r)) dW\right\|_{H} \|\varphi\|_{V} \right)\right)^p}{|t-s|^{\alpha p+1}}dsdt\right)\nonumber\\
&\leq C\mathrm{E}\left(\int_{0}^{T}\int_{0}^{T}\frac{\left(\int_{s}^{t}\|g(\mathbf{u}^\varepsilon(r))\|^2_{L_2(H; H)}dr\right)^\frac{p}{2}}{|t-s|^{\alpha p+1}}dsdt\right)\nonumber\\
&\leq C\mathrm{E}\left(\int_{0}^{T}\int_{0}^{T}\frac{\left(\int_{s}^{t}c_4(1+\|\mathbf{u}^\varepsilon(r)\|^2_{H})dr\right)^\frac{p}{2}}{|t-s|^{\alpha p+1}}dsdt\right)\nonumber\\
&\leq C\mathrm{E}\left(\sup_{r\in [0,T]}(1+\|\mathbf{u}^\varepsilon(r)\|^p_{H})\right)\int_{0}^{T}\int_{0}^{T}\frac{1}{|t-s|^{(\alpha-\frac{1}{2}) p+1}}dsdt\nonumber\\
&\lesssim_{m,M, c_2, c_4, \kappa,T,p} C.
\end{align}
Then we have for any $\gamma<\alpha-\frac{1}{p}$
\begin{align}\label{K1}
\mathrm{E}\left(\sup_{0\leq s<t\leq T}\frac{\|G(t)-G(s)\|_{V'}^p}{|t-s|^{\gamma p}}\right)\leq C\mathrm{E}\left(\int_{0}^{T}\int_{0}^{T}\frac{\|G(t)-G(s)\|_{V'}^p}{|t-s|^{\alpha p+1}}dsdt\right) \lesssim_{m,M,c_2, c_4, \kappa,T,p} C.
\end{align}
Note that
\begin{align}\label{n1}\|G(t)-G(s)\|_{V'}\leq K(\omega)|t-s|^\gamma,\end{align}
where $K(\omega)=\sup_{0\leq s<t\leq T}\frac{\|G(t)-G(s)\|_{V'}}{|t-s|^\gamma}$ is a random variable. From \eqref{K1} we have for any $p>2$, $\mathrm{E}(K^p(\omega)) \lesssim_{m,M, c_2, c_4,\kappa,T,p} C$, then  $K(\omega)<\infty$, $\mathrm{P}$ a.s. 
Consequently, choose $\alpha, p$ such that $\alpha>\frac{1}{4}+\frac{1}{p}$, we  have for $\gamma\in [\frac{1}{4}, \alpha-\frac{1}{p})$
\begin{align}\label{3.28}	&\mathrm{E} \left(\sup_{\theta\in (0,T)}\left(\theta^{-\frac{1}{2}}\int_{0}^{T-\theta}\left\|\int^{t + \theta}_{t} g( \mathbf{u}^{\varepsilon}(s)) \, dW\right\|^2_{V'}dt\right)\right)\nonumber\\
&\leq \mathrm{E} \left(\sup_{\theta\in (0,T)}\left(\theta^{-\frac{1}{2}}\int_{0}^{T-\theta}K^2(\omega)\theta^{2\gamma} dt\right)\right)\nonumber\\
&\leq \mathrm{E} \left(\sup_{\theta\in (0,T)}\theta^{2\gamma-\frac{1}{2}}\int_{0}^{T}K^2(\omega) dt\right)\nonumber\\
&\leq T^{2\gamma-\frac{1}{2}} \left(\int_{0}^{T}\mathrm{E}(K^2(\omega)) dt\right)
\lesssim_{m,M,c_2, c_4,\kappa,T} C.
\end{align}

Combined the above estimates \eqref{3.25}-\eqref{3.27} and \eqref{3.28}, we obtain
	\begin{align}\label{3.28*}
	\mathrm{E} \left(\sup_{\theta\in (0,T)}\left(\theta^{-\frac{1}{2}}\int^{T - \theta}_{0} \| \rho^{\varepsilon}(t + \theta) \mathbf{u}^{\varepsilon}(t + \theta) -\rho^{\varepsilon}(t) \mathbf{u}^{\varepsilon}(t) \|^2_{V^{\prime}} \, dt\right)\right)
	\lesssim_{m,M,c_2, c_4,\kappa,T} C.
	\end{align}
By \eqref{3.28*} we obtain
\begin{align}\label{3.31}
&	\mathrm{E} \left(\sup_{\theta\in (0,T)}\left(\theta^{-\frac{1}{4}}\int^{T - \theta}_{0} \int_{\mathcal{O}} (\rho^{\varepsilon}(t + \theta) \mathbf{u}^{\varepsilon}(t + \theta) -\rho^{\varepsilon}(t) \mathbf{u}^{\varepsilon}(t))(\mathbf{u}^{\varepsilon}(t + \theta)-\mathbf{u}^{\varepsilon}(t))dx \, dt\right)\right)\nonumber\\
&\leq \mathrm{E} \left(\sup_{\theta\in (0,T)}\left(\theta^{-\frac{1}{4}}\int^{T - \theta}_{0} \|\rho^{\varepsilon}(t + \theta) \mathbf{u}^{\varepsilon}(t + \theta) -\rho^{\varepsilon}(t) \mathbf{u}^{\varepsilon}(t)\|_{V'}\|\mathbf{u}^{\varepsilon}(t + \theta)-\mathbf{u}^{\varepsilon}(t)\|_{V}\, dt\right)\right)\nonumber\\
&\leq  C\mathrm{E} \left(\left(\sup_{\theta\in (0,T)}\left(\theta^{-\frac{1}{2}}\int^{T - \theta}_{0} \|\rho^{\varepsilon}(t + \theta) \mathbf{u}^{\varepsilon}(t + \theta) -\rho^{\varepsilon}(t) \mathbf{u}^{\varepsilon}(t)\|^2_{V'}dt\right)\right)^\frac{1}{2}\left(\int_{0}^{T}\|\mathbf{u}^{\varepsilon}(t)\|^2_{V}\, dt\right)^\frac{1}{2}\right)\nonumber\\
&\leq C\left(\mathrm{E} \left(\sup_{\theta\in (0,T)}\left(\theta^{-\frac{1}{2}}\int^{T - \theta}_{0} \| \rho^{\varepsilon}(t + \theta) \mathbf{u}^{\varepsilon}(t + \theta) -\rho^{\varepsilon}(t) \mathbf{u}^{\varepsilon}(t) \|^2_{V^{\prime}} \, dt\right)\right)\right)^\frac{1}{2}
\left(\mathrm{E}\int_{0}^{T}\|\mathbf{u}^{\varepsilon}(t)\|^2_{V}\, dt\right)^\frac{1}{2}\nonumber\\
& \lesssim_{m,M,c_2, c_4,\kappa,T} C.
	\end{align}
By H\"{o}lder's inequality, the continuity equation and embedding $V\rightarrow L^p(\mathcal{O})$ for $p\in [1,\frac{2d}{d-2}]$  we further have
\begin{align}\label{3.32}
&	\mathrm{E} \left(\sup_{\theta\in (0,T)}\left(\theta^{-\frac{1}{4}}\int^{T - \theta}_{0} \int_{\mathcal{O}} -(\rho^{\varepsilon}(t + \theta) \mathbf{u}^{\varepsilon}(t) -\rho^{\varepsilon}(t) \mathbf{u}^{\varepsilon}(t))(\mathbf{u}^{\varepsilon}(t + \theta)-\mathbf{u}^{\varepsilon}(t))dx \, dt\right)\right)\nonumber\\
&=\mathrm{E} \left(\sup_{\theta\in (0,T)}\left(\theta^{-\frac{1}{4}}\int^{T - \theta}_{0} \int_{\mathcal{O}}\left(\int_{t}^{t+\theta} -\partial_s\rho^{\varepsilon}ds \right) \mathbf{u}^{\varepsilon}(t) (\mathbf{u}^{\varepsilon}(t + \theta)-\mathbf{u}^{\varepsilon}(t))dx \, dt\right)\right)\nonumber\\
&=\mathrm{E} \left(\sup_{\theta\in (0,T)}\left(\theta^{-\frac{1}{4}}\int^{T - \theta}_{0} \int_{t}^{t+\theta}-\int_{\mathcal{O}} \rho^{\varepsilon} (s)\mathbf{u}^\varepsilon(s) \nabla(\mathbf{u}^{\varepsilon}(t) (\mathbf{u}^{\varepsilon}(t + \theta)-\mathbf{u}^{\varepsilon}(t))) dx\, ds \, dt\right)\right)\nonumber\\
&\leq C\mathrm{E} \left(\sup_{\theta\in (0,T)}\left(\theta^{-\frac{1}{4}}\int^{T - \theta}_{0} \int_{t}^{t+\theta}\|\rho^{\varepsilon} (s)\|_{L^\infty(\mathcal{O})}\| \mathbf{u}^\varepsilon(s) \|_{L^6(\mathcal{O})} \|\nabla(\mathbf{u}^{\varepsilon}(t) (\mathbf{u}^{\varepsilon}(t + \theta)-\mathbf{u}^{\varepsilon}(t))) \|_{L^\frac{3}{2}(\mathcal{O})} ds \, dt\right)\right)\nonumber\\
&\leq CM\mathrm{E} \left(\sup_{\theta\in (0,T)}\left(\theta^{-\frac{1}{4}}\int^{T - \theta}_{0} \int_{t}^{t+\theta}\| \mathbf{u}^\varepsilon(s) \|_{V}ds \|\mathbf{u}^{\varepsilon}(t)\|_{V} \|\mathbf{u}^{\varepsilon}(t + \theta)-\mathbf{u}^{\varepsilon}(t) \|_{V} \, dt\right)\right)\nonumber\\
&\leq CM\mathrm{E} \left(\sup_{\theta\in (0,T)}\int^{T - \theta}_{0} \left(\int_{t}^{t+\theta}\| \mathbf{u}^\varepsilon(s) \|^2_{V}ds \right)^\frac{1}{2} \|\mathbf{u}^{\varepsilon}(t)\|_{V} \|\mathbf{u}^{\varepsilon}(t + \theta)-\mathbf{u}^{\varepsilon}(t) \|_{V} \, dt\right)\nonumber\\
&\leq CM\mathrm{E} \left(\int_{0}^{T}\| \mathbf{u}^\varepsilon(s) \|^2_{V}ds \right)^\frac{3}{2} \lesssim_{m,M,c_2, c_4,\kappa,T} C.
\end{align}
By \eqref{3.31} and \eqref{3.32} we see
\begin{align}\label{3.31*}
&	\mathrm{E} \left(\sup_{\theta\in (0,T)}\left(\theta^{-\frac{1}{4}}\int^{T - \theta}_{0} \int_{\mathcal{O}} \rho^{\varepsilon}(t + \theta) (\mathbf{u}^{\varepsilon}(t + \theta) - \mathbf{u}^{\varepsilon}(t))^2dx \, dt\right)\right)\nonumber\\
&\leq \mathrm{E} \left(\sup_{\theta\in (0,T)}\left(\theta^{-\frac{1}{4}}\int^{T - \theta}_{0} \int_{\mathcal{O}} (\rho^{\varepsilon}(t + \theta) \mathbf{u}^{\varepsilon}(t + \theta) -\rho^{\varepsilon}(t) \mathbf{u}^{\varepsilon}(t))(\mathbf{u}^{\varepsilon}(t + \theta)-\mathbf{u}^{\varepsilon}(t))dx \, dt\right)\right)\nonumber\\
&\quad+	\mathrm{E} \left(\sup_{\theta\in (0,T)}\left(\theta^{-\frac{1}{4}}\int^{T - \theta}_{0} \int_{\mathcal{O}} -(\rho^{\varepsilon}(t + \theta) \mathbf{u}^{\varepsilon}(t) -\rho^{\varepsilon}(t) \mathbf{u}^{\varepsilon}(t))(\mathbf{u}^{\varepsilon}(t + \theta)-\mathbf{u}^{\varepsilon}(t))dx \, dt\right)\right)\nonumber\\
& \lesssim_{m,M,c_2,c_4,\kappa,T} C.
\end{align}
Then, by \eqref{3.31*} and the fact that $0<m\leq \rho^\varepsilon(x,t)\leq M<\infty$ we have
\begin{align*}
	\mathrm{E} \left(\sup_{\theta\in (0,T)}\left(\theta^{-\frac{1}{4}}\int^{T - \theta}_{0} \|\mathbf{u}^{\varepsilon}(t + \theta) - \mathbf{u}^{\varepsilon}(t)\|^2_{H} \, dt\right)\right) \lesssim_{m,M,c_2, c_4,\kappa,T} C.
\end{align*}

We finally show $\rho^{\varepsilon} \mathbf{u}^{\varepsilon}\in L^2(\Omega; W^{\beta,p}(0,T; W^{-3,2}(\mathcal{O})))$ with $\beta\in [0,\frac{1}{2}), p>2$. Using the momentum equation we see
\begin{align}\label{2.34}
&\mathrm{E}\left(\|\rho^{\varepsilon} \mathbf{u}^{\varepsilon}\|^p_{ W^{\beta,p}(0,T; W^{-3,2}(\mathcal{O}))}\right)\nonumber\\
&\leq c_p\mathrm{E}\left(\|\rho_0\mathbf{u}_0\|^p_{ W^{\beta,p}(0,T; W^{-3,2}(\mathcal{O}))}+\left\|\int_{0}^{t}A^\varepsilon\mathbf{u}^{\varepsilon}(s)ds\right\|^p_{ W^{\beta,p}(0,T; W^{-3,2}(\mathcal{O}))}\right)\nonumber\\
&\quad+ c_p\mathrm{E}\left(\left\|\int_{0}^{t}\text{div}(\rho^{\varepsilon}(s) \mathbf{u}^{\varepsilon}(s) \otimes \mathbf{u}^{\varepsilon}(s)) ds\right\|^p_{ W^{\beta,p}(0,T; W^{-3,2}(\mathcal{O}))}\right)\nonumber\\
&\quad+c_p\mathrm{E}\left(\left\|\int_{0}^{t}f^\varepsilon(\mathbf{u}^{\varepsilon}(s))ds\right\|^p_{ W^{\beta,p}(0,T; W^{-3,2}(\mathcal{O}))}
+\left\|\int_{0}^{t}g(\mathbf{u}^{\varepsilon}(s))dW\right\|^p_{ W^{\beta,p}(0,T; W^{-3,2}(\mathcal{O}))}\right).
\end{align}
By Lemma \ref{lem3.1} and  $(\mathbf{A.2})$ we have
\begin{align}
&c_p\mathrm{E}\left(\left\|\int_{0}^{t}A^\varepsilon\mathbf{u}^{\varepsilon}(s)ds\right\|^p_{W^{\beta,p}(0,T; W^{-3,2}(\mathcal{O}))}+\left\|\int_{0}^{t}f^\varepsilon(\mathbf{u}^{\varepsilon}(s))ds\right\|^p_{W^{\beta,p}(0,T; W^{-3,2}(\mathcal{O}))}\right)\nonumber\\
&\leq c_p\mathrm{E}\left(\left\|\int_{0}^{t}A^\varepsilon\mathbf{u}^{\varepsilon}(s)ds\right\|^p_{W^{1,p}(0,T; W^{-3,2}(\mathcal{O}))}+\left\|\int_{0}^{t}f^\varepsilon(\mathbf{u}^{\varepsilon}(s))ds\right\|^p_{W^{1,p}(0,T; W^{-3,2}(\mathcal{O}))}\right)\nonumber\\
&\leq C(T, p)\mathrm{E}\left(\int_{0}^{T}\|A^\varepsilon\mathbf{u}^{\varepsilon}(s)\|^p_{W^{-3,2}(\mathcal{O})}ds
+\int_{0}^{T}\|f^\varepsilon(\mathbf{u}^{\varepsilon}(s))\|^p_{W^{-3,2}(\mathcal{O})}ds\right)\nonumber\\
&\leq C(c_2, T,p)\mathrm{E}\left(\int_{0}^{T}(1+\|\mathbf{u}^{\varepsilon}(s)\|^p_{H})ds\right)
\lesssim_{m,M,c_2, c_4,\kappa,T, p} C.
\end{align}
For the third term on the right-hand side of \eqref{2.34}, by Lemma \ref{lem3.1} we have
\begin{align}
&c_p\mathrm{E}\left(\left\|\int_{0}^{t}\text{div}(\rho^{\varepsilon}(s) \mathbf{u}^{\varepsilon}(s) \otimes \mathbf{u}^{\varepsilon}(s)) ds\right\|^p_{ W^{\beta,p}(0,T; W^{-3,2}(\mathcal{O}))}\right)\nonumber\\
&\leq C(p, T)\mathrm{E}\left(\int_{0}^{T}\|\text{div}(\rho^{\varepsilon}(s) \mathbf{u}^{\varepsilon}(s) \otimes \mathbf{u}^{\varepsilon}(s))\|^p_{W^{-3,2}(\mathcal{O})} ds\right)\nonumber\\
&\leq C(p, T)\mathrm{E}\left(\int_{0}^{T}\|\rho^{\varepsilon}(s) \mathbf{u}^{\varepsilon}(s) \otimes \mathbf{u}^{\varepsilon}(s)\|^p_{L^1(\mathcal{O})} ds\right)\nonumber\\
&\leq C(p, T)\mathrm{E}\left(\int_{0}^{T}\|\rho^{\varepsilon}(s)\|^p_{L^\infty(\mathcal{O})} \|\mathbf{u}^{\varepsilon}(s) \|^{2p}_{H} ds\right)\nonumber\\
&\leq C(p, M, T)\mathrm{E}\left(\int_{0}^{T}\|\mathbf{u}^{\varepsilon}(s) \|^{2p}_{H} ds\right)
\lesssim_{m,M,c_2, c_4,\kappa,T, p} C.
\end{align}
For the stochastic integral, similar to \eqref{sto} we see
\begin{align}\label{sto1}
\mathrm{E}\left(\left\|\int_{0}^{t}g(\mathbf{u}^{\varepsilon}(s))dW\right\|^p_{ W^{\beta,p}(0,T; W^{-3,2}(\mathcal{O}))}\right)\lesssim_{m,M,c_2, c_4,\kappa,T,p} C.
\end{align}
By \eqref{2.34}-\eqref{sto1} we obtain the desired estimate $\rho^{\varepsilon} \mathbf{u}^{\varepsilon}\in L^2(\Omega; W^{\beta,p}(0,T; W^{-3,2}(\mathcal{O})))$.
This completes the proof.
\end{proof}

\subsection{Stochastic compactness} With the necessary estimates in hands, we are in a position to show the tightness. 
We first recall the following compactness criterion. 
\begin{lemma}\cite[Theorem 3]{Simon}\label{Simon1} Suppose that $X_{1}\subset X_{0}\subset X_{2}$ are Banach spaces, where $X_{1}$ and $X_{2}$ are reflexive and the embedding of $X_{1}$ into $X_{0}$ is compact. Let $\mathcal{E}$ be a bounded set in $L^{p}(0,T; X_{1})$ for any $1\leq p\leq\infty$, and
$$\|h(t+\theta)-h(t)\|_{L^p(0,T-\theta; X_2)}\rightarrow 0, ~{\rm as}~\theta\rightarrow 0,$$
uniformly in $h\in \mathcal{E}$. Then, $\mathcal{E}$ is relatively compact in $L^p(0,T; X_0)$. Similarly, we have that the embedding of space
$$
L^{p}(0,T; X_{1})\cap W^{\alpha, q}(0,T; X_{2})\rightarrow L^{p}(0,T; X_{0})
$$
is compact where $\alpha > 0$ if $q>p$ and $\alpha > \frac{1}{q}-\frac{1}{p}$ if $q<p$, (and in $C([0,T]; X_{0})$ if $p=\infty$). 
\end{lemma}

Consider the space
$$X=C([0,T]; W^{-\alpha,p}(\mathcal{O}))\times L^2(0,T;H)\times C([0,T]; W^{-\alpha, 2}(\mathcal{O})),~ \alpha\in (0,1), ~p\in (2,\infty).$$
Denote by $\mathfrak{L}_{(\rho^\varepsilon, \mathbf{u}^\varepsilon, \rho^\varepsilon\mathbf{u}^\varepsilon)}$ the joint law of $\rho^\varepsilon, \mathbf{u}^\varepsilon, \rho^\varepsilon\mathbf{u}^\varepsilon$, we next show that the family of measures $\mathfrak{L}_{(\rho^\varepsilon, \mathbf{u}^\varepsilon, \rho^\varepsilon\mathbf{u}^\varepsilon)}$  is tight in $X$.

\begin{lemma} The family of measures $\mathfrak{L}_{(\rho^\varepsilon, \mathbf{u}^\varepsilon, \rho^\varepsilon\mathbf{u}^\varepsilon)}$  is tight in the path space $X$.
\end{lemma}
\begin{proof}
For any $R>0$, $p>2$ and $\beta\in (\frac{1}{p}, \frac{1}{2})$, define the sets
$$
\mathscr{B}^1_R:=\left\{\mathbf{u} \in L^2(0,T;V): \int_{0}^{T}\|\mathbf{u}(t)\|_{V}^2 \, dt
+\sup_{\theta\in (0,T)}\left(\theta^{-\frac{1}{4}}
\int_{0}^{T-\theta}\| \mathbf{u}(t + \theta) - \mathbf{u}(t) \|^2_{H} \, dt\right)
\leq R\right\},
$$
$$
\mathscr{B}^2_R:=\left\{\rho \in L^\infty(\mathcal{O}_t)\cap W^{1,\infty}(0,T; W^{-1,2}(\mathcal{O})): \|\rho\|_{L^\infty(\mathcal{O}_t)}
+ \|\rho\|_{W^{1,\infty}(0,T; W^{-1,2}(\mathcal{O}))}
\leq R\right\},
$$
\begin{align*}
&\mathscr{B}^3_R:=\Big\{\rho \mathbf{u} \in L^\infty(0,T; L^2(\mathcal{O}))\cap W^{\beta,p}(0,T; W^{-3,2}(\mathcal{O})): \nonumber\\
&\qquad\qquad\qquad\qquad\qquad\|\rho \mathbf{u}\|_{L^\infty(0,T; L^2(\mathcal{O}))}
+\|\rho  \mathbf{u}\|_{W^{\beta,p}(0,T; W^{-3,2}(\mathcal{O}))}
\leq R\Big\}.
\end{align*}

According to Lemma \ref{Simon1}, we know the set $\mathscr{B}^i_R, i=1,2,3$ is relative compact in  $L^2(0,T;H)$, $C([0,T]; W^{-\alpha,p}(\mathcal{O}))$, $C([0,T]; W^{-\alpha, 2}(\mathcal{O}))$ respectively. Then the set $\mathscr{B}_R=\mathscr{B}^1_R\times \mathscr{B}^2_R\times\mathscr{B}^3_R$ is relative compact in $X$.  By Lemmas \ref{lem3.1}, \ref{lem3.2} and Chebyschev's inequality, we see
$$
\mathrm{P}(\mathbf{u}^\varepsilon\in \mathscr{B}^1_R)=1-\mathrm{P}(\mathbf{u}^\varepsilon\in (\mathscr{B}^1_R)^c)$$
$$
\geq 1-\frac{1}{R}\mathrm{E}\left(\int_{0}^{T}\|\mathbf{u}^\varepsilon(t)\|_{V}^2
dt+\sup_{\theta\in (0,T)}\left(\theta^{-\frac{1}{4}}\int_{0}^{T-\theta}\| \mathbf{u}^{\varepsilon}(t + \theta) - \mathbf{u}^{\varepsilon}(t) \|^2_{H} dt\right)\right)
$$
\begin{align}\label{com3}
\geq 1-\frac{C(m,M,\kappa,T)}{R}.
\end{align}
Similarly, we have
$$\mathrm{P}(\mathbf{\rho}^\varepsilon\in \mathscr{B}^2_R)\geq 1-\frac{C(m,M,\kappa,T)}{R},~~\mathrm{P}(\mathbf{\rho}^\varepsilon\mathbf{u}^\varepsilon\in \mathscr{B}^3_R)\geq 1-\frac{C(m,M,\kappa,T)}{R},
$$
which along with \eqref{com3} imply that  for any $\epsilon'>0$ and every $\varepsilon$, there exists $R(\epsilon')$ such that
$$\mathrm{P}((\mathbf{\rho}^\varepsilon, \mathbf{u}^\varepsilon, \mathbf{\rho}^\varepsilon\mathbf{u}^\varepsilon)\in \mathscr{B}_{R(\epsilon')})\geq 1-\epsilon',
$$
thus, the family of measures $\mathfrak{L}_{(\rho^\varepsilon, \mathbf{u}^\varepsilon, \rho^\varepsilon\mathbf{u}^\varepsilon)}$ is tight in $X$,  as desired.
\end{proof}
Furthermore, since $ W$ is only one element, also for $\rho_0, \mathbf{u}_0$,  we have that the family of measures $\mathfrak{L}_{(\rho^\varepsilon, \mathbf{u}^\varepsilon, \rho^\varepsilon\mathbf{u}^\varepsilon, \rho_0, \mathbf{u}_0, W)}$ is tight in the path space $X\times L^\infty(\mathcal{O})\times H\times C([0,T]; H_0)$.\\

The following Skorokhod-Jakubowski representation theorem will be used to represent a weakly convergent
probability measure sequence on a topological space as the distribution of a pointwise convergent
random variable sequence.

\begin{proposition}\label{pro3.1}\cite[Theorem 4.6]{brz} If $(E, \mathcal{E})$ is a topological space such that  there exists a sequence
of continuous functions $h_n: E\rightarrow \mathbb{R}$ that separates points of $E$. If the set of probability measures $\{\mu_n\}_{n\geq 1}$ on $(E, \mathcal{E})$ is tight, then there exist a subsequence $\{\mu_{n_k}\}_{k\geq 1}$ of $\{\mu_n\}_{n\geq 1}$, a probability space $(\Omega, \mathcal{F}, \mathrm{P})$ and a sequence of $E$-valued random variables $\{\mathbf{u}_k\}_{k\geq 1}, \mathbf{u}$ such that the law of $\mathbf{u}_k$ is $\mu_{n_k}$, for every $k\in \mathbb{N}^+$ and $\mathbf{u}_k\rightarrow \mathbf{u}$, $\mathrm{P}$ a.s. as $k\rightarrow \infty$.





\end{proposition}

For any $\alpha\in (0,1)$, $p\in (2,\infty)$,  $H$, $L^2(0,T;H)$, $C([0,T]; W^{-\alpha, p}(\mathcal{O}))$,  $C([0,T]; W^{-\alpha, 2}(\mathcal{O}))$ and $C([0,T]; H_0)$ are Polish spaces, hence such continuous functions indeed exist in  $H$, $L^2(0,T;H)$, $C([0,T]; W^{-\alpha, p}(\mathcal{O}))$, $C([0,T]; W^{-\alpha, 2}(\mathcal{O}))$ and $C([0,T]; H_0)$. Moreover, $L^1(\mathcal{O})$ is a Polish space, then there exists a sequence $\{h_n\}_{n\geq 1}$ such that $h_n: L^\infty(\mathcal{O})\rightarrow \mathbb{R}$ is a continuous function, and  $\{h_n\}_{n\geq 1}$ separates points of $L^\infty(\mathcal{O})$.
The conditions in the Proposition \ref{pro3.1} are verified, we infer that there exist a new probability space $\mathcal{S}_1=(\widetilde{\Omega}, \widetilde{\mathcal{F}}, \widetilde{\mathrm{P}})$ and a sequence of random variables $\left\{\widetilde{\rho}^\varepsilon, \widetilde{\mathbf{u}}^\varepsilon, \widetilde{\rho}^\varepsilon\widetilde{\mathbf{u}}^\varepsilon, \widetilde{\rho}^\varepsilon(0), \widetilde{\mathbf{u}}^\varepsilon(0), \widetilde{W}^\varepsilon\right\}_{\varepsilon\in (0,1)}$ and $ \widetilde{\rho}, \widetilde{\mathbf{u}}, \widetilde{\varrho}, \widetilde{\rho}(0), \widetilde{\mathbf{u}}(0), \widetilde{W}$ such that their laws are $\mathfrak{L}_{(\rho^\varepsilon, \mathbf{u}^\varepsilon, \rho^\varepsilon\mathbf{u}^\varepsilon, \rho_0, \mathbf{u}_0, W)}$,
moreover \begin{align}\label{3.29}(\widetilde{\rho}^\varepsilon, \widetilde{\mathbf{u}}^\varepsilon, \widetilde{\rho}^\varepsilon\widetilde{\mathbf{u}}^\varepsilon)\rightarrow  (\widetilde{\rho}, \widetilde{\mathbf{u}}, \widetilde{\varrho}),~ {\rm in}~ X,\end{align}
and
\begin{align}\widetilde{W}^\varepsilon\rightarrow \widetilde{W}, ~{\rm in} ~C([0,T]; H_0),\end{align}
 $\widetilde{\mathrm{P}}$ a.s. as $\varepsilon\rightarrow 0$. Since $\widetilde{W}^\varepsilon$ has the same distribution with $W$, then we can write $\widetilde{W}^\varepsilon(t,\omega)=\sum_{k\geq 1}\sqrt{\lambda_{k}}\mathbf{e}_{k}\widetilde{W}^\varepsilon_k(t,\omega)$ and $\widetilde{W}(t,\omega)=\sum_{k\geq 1}\sqrt{\lambda_{k}}\mathbf{e}_{k}\widetilde{W}_k(t,\omega)$, $\{\widetilde{W}^\varepsilon_{k}\}_{k\geq 1}$, $\{\widetilde{W}_{k}\}_{k\geq 1}$ are two sequences of independent standard $\widetilde{\mathcal{F}}_t$-adapted one-dimension Brownian motions. Since initial data $(\rho_0, \mathbf{u}_0)$ is deterministic, the equivalence of distributions of $(\rho_0, \mathbf{u}_0)$ and $(\widetilde{\rho}^\varepsilon(0), \widetilde{\mathbf{u}}^\varepsilon(0))$ implies that 
 $$(\widetilde{\rho}^\varepsilon(0), \widetilde{\mathbf{u}}^\varepsilon(0))=(\widetilde{\rho}(0), \widetilde{\mathbf{u}}(0))=(\rho_0, \mathbf{u}_0), ~{\rm \widetilde{P}} ~{\rm a.s.}$$
And from the equivalence of distributions of  $(\widetilde{\rho}^\varepsilon, \widetilde{\mathbf{u}}^\varepsilon, \widetilde{\rho}^\varepsilon\widetilde{\mathbf{u}}^\varepsilon)$ and $(\rho^\varepsilon, \mathbf{u}^\varepsilon, \rho^\varepsilon\mathbf{u}^\varepsilon)$, we infer that they share the same estimates
 \begin{align}\label{1.23**}
 \mathrm{E}^{\widetilde{\mathrm{P}}}\left(\sup_{0\leq t\leq T}\|\widetilde{\mathbf{u}}^\varepsilon(t)\|_H^{2p}\right)+\mathrm{E}^{\widetilde{\mathrm{P}}}\left(\int_{0}^{T}\|\nabla \widetilde{ \mathbf{u}}^\varepsilon(t)\|_{L^2(\mathcal{O})}^2dt\right)^p \lesssim_{m,M, c_2, c_4, p,\kappa,T} C,
 \end{align}
 and
 \begin{align}\label{1.23***}
 \|\widetilde{\rho}^\varepsilon\|_{L^\infty(\mathcal{O}_t)}^p \lesssim_{m, M, p}C,
 \end{align}
for any $p\geq 1$, $ \mathrm{E}^{\widetilde{\mathrm{P}}}$  is the expectation with respect to $\widetilde{\mathrm{P}}$, the constant $C$ is independent of $\varepsilon$. Furthermore, by \eqref{1.23**} and \eqref{1.23***} we have
 \begin{align}\label{1.23****}
 \mathrm{E}^{\widetilde{\mathrm{P}}}\left(\sup_{0\leq t\leq T}\|\widetilde{\rho}^\varepsilon(t)\widetilde{\mathbf{u}}^\varepsilon(t)\|_{L^2(\mathcal{O})}^p\right) \lesssim_{m,M, c_2, c_4,p,\kappa,T}C,
 \end{align}
for any  $p\geq 2$.

We verify that actually $\widetilde{\varrho}=\widetilde{\rho}\widetilde{\mathbf{u}}$, $\widetilde{\mathrm{P}}$ a.s. Indeed, by \eqref{1.23***} we infer there exists $\widetilde{\rho}\in L^p(\widetilde{\Omega}; L^\infty(\mathcal{O}_t))$ such that (up to a subsequence)
$$\widetilde{\rho}^\varepsilon\rightarrow \widetilde{\rho},~ {\rm weak^\ast ~in}~ L^\infty(\mathcal{O}_t),~ \widetilde{\mathrm{P}}~{\rm a.s.}$$
which along with $\widetilde{\mathbf{u}}^\varepsilon\rightarrow \widetilde{\mathbf{u}}$ in $L^2(0,T; H)$, $\widetilde{\mathrm{P}}$ a.s. leads to
$$\widetilde{\rho}^\varepsilon\widetilde{\mathbf{u}}^\varepsilon\rightarrow \widetilde{\rho}\widetilde{\mathbf{u}}, ~ {\rm weak ~in}~ L^2(\mathcal{O}_t).$$
Moreover, we know $\widetilde{\rho}^\varepsilon\widetilde{\mathbf{u}}^\varepsilon\rightarrow \widetilde{\varrho}$ in $L^2(0,T; W^{-\alpha,2}(\mathcal{O}))$, then we can identify the limit.

We also have that  on the new probability space $\mathcal{S}_1$, for $\widetilde{\mathrm{P}}$ a.s. $\omega$, it holds for every $\varepsilon\in (0,1)$, $t\in [0,T]$ and $\phi\in H^1, \varphi\in V$
 \begin{eqnarray}\label{equ1*}
\left\{\begin{array}{ll}
\!\!\!(\widetilde{\rho}^{\varepsilon}(t), \phi)-(\rho_0, \phi)-\int_{0}^{t}(\widetilde{\rho}^\varepsilon(s) \widetilde{\mathbf{u}}^{\varepsilon}(s), \nabla\phi)ds=0,\\\\
\!\!\!(\widetilde{\rho}^\varepsilon(t)\widetilde{\mathbf{u}}^\varepsilon(t), \varphi)=(\rho_0\mathbf{u}_0, \varphi)-\int_0^t(A^\varepsilon \widetilde{\mathbf{u}}^\varepsilon(s), \varphi)_{ V'\times V}ds+\int_0^t(\widetilde{\rho}^\varepsilon(s)\widetilde{\mathbf{u}}^\varepsilon(s)\otimes \widetilde{\mathbf{u}}^\varepsilon(s), \nabla\varphi)ds\nonumber\\
\!\!\!\qquad+\int_{0}^{t}(f^\varepsilon( \widetilde{\mathbf{u}}^{\varepsilon}(s)), \varphi)ds+\int_{0}^{t}(g(\widetilde{\mathbf{u}}^{\varepsilon}(s))d\widetilde{W}^\varepsilon, \varphi). 
\end{array}\right.
\end{eqnarray}
For more details of the proof, we refer to \cite{cww, ww}.

\section{Homogenization problem}
Let us discuss the homogenization problem in this section.  We begin by introducing the Sobolev space and the two-scale convergence. 
Denote by $C^\infty_{per}(D)$ all the $D$-periodic infinite differential functions on $\mathbb{R}_y^d$. Let $V(\mathbb{R}_y^d)$ be the space of all functions satisfying $\nabla_y\mathbf{f}\in L^2(\mathbb{R}_y^d)$, and
$V_{per}$ be the space of all the $D$-periodic functions in $V(\mathbb{R}_y^d)$ with the norm
$$\|\mathbf{f}\|^2_{V_{per}}=\int_{D}|\nabla\mathbf{f}(y)|^2dy,$$
which is a Hilbert space. 

The following two-scale convergence results will be used in our setting.
\begin{lemma}\label{lem2.1} \cite[Theorem 4.1]{ngu}
For any $p\in (1,\infty)$, let $\{\mathbf{u}^\varepsilon\}_{\varepsilon\in (0,1)}$ be a bounded sequence in $L^p(\mathcal{O}_t)$ uniformly in $\varepsilon$, then there exists a subsequence of $\{\mathbf{u}^\varepsilon\}_{\varepsilon\in (0,1)}$ which is weak-$\Sigma$ convergent in $L^p( \mathcal{O}_t)$.
\end{lemma}

\begin{lemma}\label{lem4.2*}\cite[Lemma 6.2]{Ngu}
Let $\{\mathbf{u}^\varepsilon\}_{\varepsilon\in (0,1)}$ be a bounded sequence in $L^2(0,T;V)$ uniformly in $\varepsilon$, then there exist a subsequence of $\{\mathbf{u}^\varepsilon\}_{\varepsilon\in (0,1)}$ (still denoted by $\{\mathbf{u}^\varepsilon\}_{\varepsilon\in (0,1)}$) and a $L^2(\mathcal{O}_t; L^2_{per}(D))$-valued function $\overline{\mathbf{u}}$ such that
\begin{align*}
\frac{\partial \mathbf{u}^\varepsilon}{\partial x_i}\rightarrow \frac{\partial \mathbf{u}}{\partial x_i}+\frac{\partial \overline{\mathbf{u}}}{\partial y_i}, ~in ~ L^2(\mathcal{O}_t), ~weak-\Sigma.
\end{align*}
\end{lemma}


The following result  provides a way to passage to the limit of a sequence of product functions.
\begin{lemma}\label{lem4.3} \cite[Proposition 4.7]{ngu} For any $r, p,q>1$ with $\frac{1}{r}=\frac{1}{p}+\frac{1}{q}$, if the following two conditions hold

i. a sequence of functions $\{\mathbf{u}^\varepsilon\}_{\varepsilon\in (0,1)}$ is weak-$\Sigma$ convergent to some certain $\mathbf{u}\in L^p( \mathcal{O}_t; L^p_{per}(D))$;

ii. a sequence of functions $\{\mathbf{v}^\varepsilon\}_{\varepsilon\in (0,1)}$ is strong-$\Sigma$ convergent to some certain $\mathbf{v}\in L^q( \mathcal{O}_t; L^q_{per}(D))$.

Then, we have that the sequence of $\{\mathbf{u}^\varepsilon\mathbf{v}^\varepsilon\}_{\varepsilon\in (0,1)}$ is weak-$\Sigma$ convergent to $\mathbf{u}\mathbf{v}$ in $L^r( \mathcal{O}_t)$.
\end{lemma}

The Vitali convergence theorem is applied to identifying the limit.
\begin{theorem}\cite[Chapter 3]{Kallenberg}\label{thm4.1} Let $p\geq 1$, $\{\mathbf{u}_\varepsilon\}_{\varepsilon\in(0,1)}\in L^p$ and $\mathbf{u}_\varepsilon\rightarrow \mathbf{u}$ in probability. Then, the following are equivalent:\\
i. $\mathbf{u}_\varepsilon\rightarrow \mathbf{u}$ in $L^p$;\\
ii. the sequence $|\mathbf{u}_\varepsilon|^p$ is uniformly integrable;\\
iii. $\mathrm{E}\left(|\mathbf{u}_\varepsilon|^p\right)\rightarrow \mathrm{E}\left(|\mathbf{u}|^p\right)$.
\end{theorem}

Recall that $\{\widetilde{\rho}^\varepsilon, \widetilde{\mathbf{u}}^{\varepsilon}, \widetilde{\rho}^\varepsilon\widetilde{\mathbf{u}}^{\varepsilon}, \widetilde{W}^\varepsilon\}_{\varepsilon\in (0,1)}$ is the sequence we chosen from the Skorokhod-Jakubowski representation theorem, which satisfies equations \eqref{equ1*} with uniform estimates \eqref{1.23**}-\eqref{1.23****}. If no confusion occurs, we still use $\rho^\varepsilon, \mathbf{u}^\varepsilon, \rho^\varepsilon\mathbf{u}^\varepsilon, W^\varepsilon, \mathrm{E}$ instead of $\widetilde{\rho}^\varepsilon, \widetilde{\mathbf{u}}^{\varepsilon}, \widetilde{\rho}^\varepsilon\widetilde{\mathbf{u}}^{\varepsilon}, \widetilde{W}^\varepsilon, \mathrm{E}^{\widetilde{\mathrm{P}}}$. From Proposition \ref{pro3.1}, we see $\mathrm{P}$ a.s.
\begin{align}\label{2.1}
(\rho^\varepsilon, \mathbf{u}^\varepsilon, \rho^\varepsilon\mathbf{u}^\varepsilon)\rightarrow  (\rho, \mathbf{u}, \rho\mathbf{u}),~ {\rm in}~ X,~ W^\varepsilon\rightarrow W,~ {\rm in}~ C([0,T]; H_0).
\end{align}
Then, the convergence $\mathbf{u}^\varepsilon\rightarrow \mathbf{u}$ in $L^2(\mathcal{O}_t)$, $\mathrm{P}$ a.s. implies 
\begin{align}\label{4.1*}
\int_{0}^{T}\|\mathbf{u}^\varepsilon(t)\|^2_Hdt\lesssim_{\omega, m, M, c_2, c_4, \kappa,T} C,
\end{align}
where $C$ is independent of $\varepsilon$.

We next establish the following uniform regularity estimates of the sequence of new processes $\{\mathbf{u}^\varepsilon\}_{\varepsilon\in (0,1)}$.
\begin{lemma}\label{lemapx.1} If $(\mathbf{A.2})-(\mathbf{A.4})$ hold and \(\mathbf{u}_0 \in H\), \(0 < m \leq \rho_0 \leq M<\infty\), then for any $T>0$, every  $s\in[0,T]$,  the following uniform estimate of the new process $\mathbf{u}^\varepsilon$ in $\varepsilon$ holds
	\begin{align*}
		\int_{0}^{s}\|\nabla \mathbf{u}^\varepsilon(r)\|_{L^2(\mathcal{O})}^2dr\lesssim_{\omega, m, M, c_2, c_4, T, \kappa} C,
	\end{align*}
	where the positive constant $C$ is independent of  $\varepsilon$.
\end{lemma}
	\begin{proof} Similar to \eqref{priori2.12}-\eqref{1.6} we see for every $s\in [0,T]$
		\begin{align}\label{priori2.12_r}
			&\int_\mathcal{O} |\sqrt{\rho^{\varepsilon}(s)} \mathbf{u}^{\varepsilon}(s)|^2 \, dx+2\kappa\int_0^s \|\nabla \mathbf{u}^\varepsilon(r)\|_{L^2 (\mathcal{O})}\,dr \nonumber\\
			& \leq 
			\int_\mathcal{O} |\sqrt{\rho^{\varepsilon}(s)} \mathbf{u}^{\varepsilon}(s)|^2 \, dx+2\int_0^s (A^\varepsilon \mathbf{u}^\varepsilon(r),
			\mathbf{u}^\varepsilon(r))_{V'\times V}\,dr \nonumber\\
			&=  \int_\mathcal{O} |\sqrt{\rho_0} \mathbf{u}_0|^2 \, dx + 2\int_0^s \int_\mathcal{O}
			\mathbf{u}^{\varepsilon}(r)
			f\!\left(\frac{x}{\varepsilon},r,\mathbf{u}^{\varepsilon}(r)\right)
			\, dx dr \nonumber\\
			&\quad + 2\int_0^s 
			(g(\mathbf{u}^{\varepsilon}(r)) dW^\varepsilon, \mathbf{u}^{\varepsilon}(r)) + \int_0^s  \| g(\mathbf{u}^{\varepsilon}(r))\|^2_{L_2(H;H)}dr .
		\end{align}
		Using  $(\mathbf{A.2})$, $(\mathbf{A.4})$, we see
		\begin{align}\label{3.11*}
			\left|2\int_0^s\int_\mathcal{O}  \mathbf{u}^{\varepsilon}(r) f\left(\frac{x}{\varepsilon}, r, \mathbf{u}^{\varepsilon}(r)\right) \, dx dr\right|\leq c_2\int_0^s (1+3\|\mathbf{u}^{\varepsilon}(r)\|^2_{H})dr,
		\end{align}
		and
		\begin{align}\label{3.12*}
			\int_0^s \| g(\mathbf{u}^{\varepsilon}(r))\|^2_{L_2(H;H)}dr\leq c_4\int_0^s(1+\|\mathbf{u}^\varepsilon(r)\|_{H}^2)dr.
		\end{align}
By \eqref{4.1*} we know that for $\mathrm{P}$ a.s. $\omega$
\begin{align*}
\int_0^s\|\mathbf{u}^\varepsilon(r)\|_{H}^2dr ~{\rm is~ bounded~ uniformly~ in}~ \varepsilon,
\end{align*}
which along with \eqref{3.11*}, \eqref{3.12*} yields that  for $\mathrm{P}$ a.s. $\omega$
\begin{align}\label{6.4}
\left|2\int_0^s\int_\mathcal{O}  \mathbf{u}^{\varepsilon}(r) f\left(\frac{x}{\varepsilon}, r, \mathbf{u}^{\varepsilon}(r)\right) \, dx dr\right|+\int_0^s \| g(\mathbf{u}^{\varepsilon}(r))\|^2_{L_2(H;H)}dr ~{\rm is~ bounded~ uniformly~ in}~ \varepsilon.
\end{align}

Let $\widetilde{G}^\varepsilon(s)=2\int_0^s 
			(g(\mathbf{u}^{\varepsilon}(r)) dW^\varepsilon, \mathbf{u}^{\varepsilon}(r))$, as \eqref{n1} we have
\begin{align}\label{n1*}\left|\widetilde{G}^\varepsilon(t)-\widetilde{G}^\varepsilon(s)\right|\leq K(\omega)|t-s|^\gamma,\end{align}
where $K(\omega)=\sup_{0\leq s<t\leq T}\frac{\left|\widetilde{G}^\varepsilon(t)-\widetilde{G}^\varepsilon(s)\right|}{|t-s|^\gamma}$ is a random variable. Similar to \eqref{K1} we have for any $p>2$, $\gamma<\alpha-\frac{1}{p}$ and $\alpha\in (\frac{1}{p}, \frac{1}{2})$, $\mathrm{E}(K^p(\omega)) \lesssim_{m,M, c_2, c_4,\kappa,T,p} C$. From \eqref{n1*} we obtain
\begin{align*}\mathrm{E}\left(\sup_{|t-s|<\delta}\left|\widetilde{G}^\varepsilon(t)-\widetilde{G}^\varepsilon(s)\right|^p\right)\leq \delta^{\gamma p} \mathrm{E}(K^p(\omega))\lesssim_{m,M, c_2, c_4,\kappa,T,p} C \delta^{\gamma p},\end{align*}
which along with Chebyshev's inequality leads to
\begin{align}
\mathrm{P}\left(\sup_{|t-s|<\delta}\left|\widetilde{G}^\varepsilon(t)-\widetilde{G}^\varepsilon(s)\right|>\epsilon'\right)\leq \frac{\mathrm{E}\left(\sup_{|t-s|<\delta}\left|\widetilde{G}^\varepsilon(t)-\widetilde{G}^\varepsilon(s)\right|^p\right)}{(\epsilon')^p}\lesssim_{m,M, c_2, c_4,\kappa,T, p} \frac{ C \delta^{\gamma p}}{(\epsilon')^p},
\end{align}
thus we establish the equicontinuous of $\widetilde{G}^\varepsilon(s)$ in $C([0,T])$ in probability. Moreover, by \eqref{1.23**}, $(\mathbf{A.4})$ and the Burkholder-Davis-Gundy inequality we have
\begin{align*}
\mathrm{E}\left(\sup_{s\in [0,T]}\left|\widetilde{G}^\varepsilon(s)\right|^2\right)&\leq C\mathrm{E}\left(\int_{0}^{T}\sum_{k\geq 1}\left(\int_\mathcal{O}  \mathbf{u}^\varepsilon(r)g(\mathbf{u}^\varepsilon(r))Q^{1/2}\mathbf{e}_k  \, dx\right)^2 dr\right)\nonumber\\
&\leq Cc_4\mathrm{E}\left(\int_{0}^{T}\| \mathbf{u}^\varepsilon(r)\|_{H}^2(1+\| \mathbf{u}^\varepsilon(r)\|_{H}^2 )dr\right) \nonumber\\
&\lesssim_{m,M, c_2, c_4,\kappa,T,p} C,
\end{align*}
which implies the boundedness of $\widetilde{G}^\varepsilon(s)$  in $C([0,T])$  in  probability. The Arzelà-Ascoli theorem gives that there exists $\widetilde{G}$ such that $\widetilde{G}^\varepsilon\rightarrow \widetilde{G}$ in $C([0,T])$  in  probability.  Then we can abstract a subsequence of $\left\{\widetilde{G}^\varepsilon\right\}_{\varepsilon\in (0,1)}$ (which is not relabelled) such that for  P a.s. $\omega$
$$\widetilde{G}^\varepsilon(s)\rightarrow \widetilde{G}(s), ~{\rm in}~   C([0,T]).$$ We further obtain for every $s\in [0,T]$, and  P a.s. $\omega$, 
\begin{align}\label{n1**}\widetilde{G}^\varepsilon(s)\rightarrow \widetilde{G}(s), ~{\rm as}~ \varepsilon\rightarrow 0.\end{align}

We need to confirm that $$\widetilde{G}(s)=2\int_0^s 
			(g(\mathbf{u}(r)) dW, \mathbf{u}(r)).$$
By \eqref{2.1} and  $(\mathbf{A.3})$, we see $g(\mathbf{u}^{\varepsilon}(r))\rightarrow g(\mathbf{u}(r))$ in $L^2(0,T; L_2(H;H))$, $\mathrm{P}$ a.s. which together with $W^\varepsilon\rightarrow W,~ {\rm in}~ C([0,T]; H_0)$ and \cite[Lemma 2.1]{ANR} leads to
\begin{align*}
2\int_0^s 
			g(\mathbf{u}^{\varepsilon}(r)) \, dW^\varepsilon \rightarrow 2\int_0^s 
			g(\mathbf{u}(r)) \, dW
\end{align*}
in probability in $L^2(\mathcal{O}_t)$. Hence, there exists a subsequence (which is not relabelled) such that 
\begin{align}\label{6.5}
2\int_0^s 
			g(\mathbf{u}^{\varepsilon}(r)) \, dW^\varepsilon \rightarrow 2\int_0^s 
			g(\mathbf{u}(r)) \, dW
\end{align}
 in $L^2(\mathcal{O}_t)$, $\mathrm{P}$ a.s.  Then, the convergence \eqref{6.5} together with $\mathbf{u}^{\varepsilon}\rightarrow \mathbf{u}$ in $L^2(\mathcal{O}_t)$ , $\mathrm{P}$ a.s. gives
 \begin{align}\label{6.51*}
\left| \int_{0}^{T}\widetilde{G}^\varepsilon(s)ds-2\int_{0}^{T}\left(\int_0^s
			(g(\mathbf{u}(r)) dW, \mathbf{u}(r)\right)ds\right|\rightarrow 0.
\end{align}
We further infer from \eqref{6.51*} that there exists a subsequence (which is not relabelled) such that for almost everywhere $s\in [0,T]$ and $\mathrm{P}$ a.s. $\omega$
 \begin{align}\label{6.5*}
 \widetilde{G}^\varepsilon(s) \rightarrow 2\int_0^s
			\left(g(\mathbf{u}(r)) dW, \mathbf{u}(r)\right).
\end{align}
 By \eqref{n1**} and \eqref{6.5*}, we get the conclusion.

By $(\mathbf{A.4})$,  \eqref{1.23**} and the Burkholder-Davis-Gundy inequality we also have for every $s\in [0,T]$
		$$
		\mathrm{E}\left|\int_{0}^{s}(g(\mathbf{u}(r)) dW, \mathbf{u}(r))\right|^2 
\lesssim_{m, M, c_2, c_4, \kappa, T} C,$$
which implies that for  $\mathrm{P}$ a.s. $\omega$
\begin{align}\label{6.6}
		\left|\int_{0}^{s}(g(\mathbf{u}(r)) dW, \mathbf{u}(r))\right| ~{\rm is~ bounded~ uniformly~ in}~ \varepsilon.
\end{align}
By \eqref{n1**} and \eqref{6.6} we have  that for every $s\in [0,T]$ and $\mathrm{P}$ a.s. $\omega$ the sequence 
\begin{align}\label{6.5**}
\left|\widetilde{G}^\varepsilon(s)\right| ~{\rm is~ bounded~ uniformly~ in}~ \varepsilon.
\end{align}
	Finally, by \eqref{priori2.12_r}, \eqref{6.4} and \eqref{6.5**} we conclude that for every $s\in [0,T]$ and $\mathrm{P}$ a.s. $\omega$, the sequence
\begin{align*}
		\int_{0}^{s}\|\nabla \mathbf{u}^\varepsilon(r)\|_{L^2(\mathcal{O})}^2dr ~{\rm is~ bounded~ uniformly~ in}~ \varepsilon.
	\end{align*}
This completes the proof.
	\end{proof}

Then, by  Lemma \ref{lem4.2*} and Lemma \ref{lemapx.1}, we have the following convergence immediately.
\begin{corollary}  There exists a $L^2(\mathcal{O}_t; L^2_{per}(D))$-valued process $\overline{\mathbf{u}}$  such that for $\mathrm{P}$ a.s. $\omega$
\begin{align}\label{4.4}
\frac{\partial \mathbf{u}^\varepsilon}{\partial x_i}\rightarrow \frac{\partial \mathbf{u}}{\partial x_i}+\frac{\partial \overline{\mathbf{u}}}{\partial y_i}, ~{\rm in} ~ L^2(\mathcal{O}_t), ~{\rm weak}-\Sigma,
\end{align}
as $\varepsilon\rightarrow 0$.
\end{corollary}


$\mathbf{Homogenization~ result}$. Based on \eqref{2.1}-\eqref{4.1*} and \eqref{4.4}, we show that the  quadruple $(\rho, \mathbf{u}, \overline{\mathbf{u}}, \rho \mathbf{u})$ solves the following equations. For every $t\in [0,T]$, denote $\mathcal{O}_s=\mathcal{O}\times [0,t]$.

\begin{proposition}\label{pro2.1} If $(\mathbf{A.1})$-$(\mathbf{A.4})$ hold,  then for every $t\in [0,T]$, and for any $\phi\in H^{1}(\mathcal{O})$, $(\varphi,\psi)\in V\times V(\mathcal{O}; V_{per})$, the quintuple $(\rho, \mathbf{u}, \overline{\mathbf{u}}, \rho \mathbf{u}, W)$ satisfies the following non-homogeneous incompressible Navier-Stokes equations  $\mathrm{P}$ a.s. in the new probability space $\mathcal{S}_1$

\begin{align}
(\rho(t), \phi)-(\rho_0, \phi)-\int_{0}^{t}(\rho(s) \mathbf{u}(s), \nabla\phi)ds=0,
\end{align}
and
\begin{align}\label{4.3}
&
(\rho(t)\mathbf{u}(t), \varphi)=(\rho_0\mathbf{u}_0, \varphi)\nonumber\\
&\quad-\sum_{i,j=1}^d\int_{\mathcal{O}_s}\int_{D}a_{i,j}(y,s)\left(\frac{\partial \mathbf{u}(x,s)}{\partial x_i}+\frac{\partial \overline{\mathbf{u}}(x, y, s)}{\partial y_i}\right)\left(\frac{\partial \varphi(x)}{\partial x_j}+\frac{\partial \psi(x,y)}{\partial y_j}\right)dxdyds\nonumber\\
&\quad+\int_0^t(\rho(s)\mathbf{u}(s)\otimes \mathbf{u}(s), \nabla\varphi)ds
+\int_{\mathcal{O}_s}\int_{D}f(y, s, \mathbf{u}(s)) \varphi dxdyds+\int_{0}^{t}(g(\mathbf{u}(s))dW, \varphi).
\end{align}
\end{proposition}
\begin{proof} Let 
	$$\Phi^\varepsilon(x): =\phi(x)+\varepsilon\chi\left(x,\frac{x}{\varepsilon}\right),$$
$$\Psi^\varepsilon(x): =\varphi(x)+\varepsilon\psi\left(x,\frac{x}{\varepsilon}\right),$$
for $x\in \mathcal{O}$, in which $\phi\in C_0^\infty(\mathcal{O})$, $\chi\in C_0^\infty(\mathcal{O})\times C^\infty_{per}( D)$, $\varphi\in C_{0,div}^\infty(\mathcal{O})$, $\psi\in C_{0, div}^\infty(\mathcal{O})\times C^\infty_{per}( D) $. Recall that for every $t\in [0,T]$, and P a.s. $\omega$, $(\rho^\varepsilon, \mathbf{u}^\varepsilon, \rho^\varepsilon\mathbf{u}^\varepsilon, W^\varepsilon)$ satisfies
\begin{eqnarray}\label{equ2}
\left\{\begin{array}{ll}
\!\!\!(\rho^{\varepsilon}(t), \Phi^\varepsilon)-(\rho_0, \Phi^\varepsilon)-\int_{0}^{t}(\rho^\varepsilon(s) \mathbf{u}^{\varepsilon}(s), \nabla\Phi^\varepsilon)ds=0,\\\\
\!\!\!(\rho^\varepsilon(t)\mathbf{u}^\varepsilon(t), \Psi^\varepsilon)=(\rho_0\mathbf{u}_0, \Psi^\varepsilon)-\int_0^t(A^\varepsilon \mathbf{u}^\varepsilon(s), \Psi^\varepsilon)_{ V'\times V}ds+\int_0^t(\rho^\varepsilon(s)\mathbf{u}^\varepsilon(s)\otimes \mathbf{u}^\varepsilon(s), \nabla\Psi^\varepsilon)ds\nonumber\\
\!\!\!\qquad+\int_{0}^{t}(f^\varepsilon( \mathbf{u}^{\varepsilon}(s)), \Psi^\varepsilon)ds+\int_{0}^{t}(g(\mathbf{u}^{\varepsilon}(s))dW^\varepsilon, \Psi^\varepsilon). 
\end{array}\right.
\end{eqnarray}

$\mathbf{Step~ 1}$. We first consider the limit of the continuity equation. 
Observe that
\begin{align}\label{rho3}
(\rho^{\varepsilon}(t), \Phi^\varepsilon)=\left(\rho^{\varepsilon}(t), \phi(x)+\varepsilon\chi\left(x,\frac{x}{\varepsilon}\right)\right).\end{align}
Applying  $\rho^\varepsilon\rightarrow \rho $ in $C([0,T]; W^{-\alpha, p}(\mathcal{O}))$, for $\mathrm{P}$ a.s. $\omega$ we have $\varepsilon\rightarrow 0$
$$\left(\rho^{\varepsilon}(t), \phi(x)\right)\rightarrow\left(\rho(t), \phi(x)\right).$$
By $\rho^\varepsilon\in L^\infty(\mathcal{O}_s)$ we get $\mathrm{P}$ a.s.
$$\left|\left(\rho^{\varepsilon}(t), \varepsilon\chi\left(x,\frac{x}{\varepsilon}\right)\right)\right|\rightarrow 0.$$
Then by \eqref{rho3} we have P a.s.
\begin{align}\label{p1}
(\rho^{\varepsilon}(t), \Phi^\varepsilon)\rightarrow \left(\rho(t), \phi(x)\right).
\end{align}
Similarly we have P a.s.
\begin{align}\label{p2}
(\rho_0, \Phi^\varepsilon)\rightarrow \left(\rho_0, \phi(x)\right).
\end{align}
Next, we focus on passing to the limit of third term in the continuity equation. 
By Lemma \ref{lem2.1} and \eqref{1.23***} we infer that there exists $\rho$ such that for $\mathrm{P}$ a.s. $\omega$ and $p\in [1,\infty)$
 \begin{align}\label{tc1}\rho^\varepsilon\rightarrow \rho, ~{\rm weak}-\Sigma ~{\rm in}~ L^p(\mathcal{O}_s). \end{align}
 Observe that $\rho$ is actually independent of $y$. 
By \eqref{2.1} we know that $\mathbf{u}^\varepsilon\rightarrow \mathbf{u}$ in $L^2(\mathcal{O}_s)$, $\mathrm{P}$ a.s. Then, we have $\mathrm{P}$ a.s.
\begin{align}\label{tc2}\mathbf{u}^\varepsilon\rightarrow \mathbf{u}, ~{\rm strong}-\Sigma~{\rm in} ~L^2(\mathcal{O}_s).\end{align}
Note that $\mathbf{u}$ is also independent of variable $y$. 
By \eqref{tc1} and \eqref{tc2}, we infer from Lemma \ref{lem4.3} that $\mathrm{P}$ a.s.
\begin{align}\label{tc3}
\rho^\varepsilon\mathbf{u}^\varepsilon\rightarrow \rho\mathbf{u}, ~{\rm weak}-\Sigma ~{\rm in}~ L^\frac{2p}{p+2}( \mathcal{O}_s).
\end{align}
Since
$$
\int_{0}^{t}(\rho^\varepsilon(s) \mathbf{u}^{\varepsilon}(s), \nabla \Phi^\varepsilon)ds
=\int_{0}^{t}(\rho^\varepsilon(s) \mathbf{u}^{\varepsilon}(s), \nabla_x \phi(x))ds$$
$$
+\int_{0}^{t}\left(\rho^\varepsilon(s) \mathbf{u}^{\varepsilon}(s), \varepsilon\nabla_x\chi\left(x,\frac{x}{\varepsilon}\right)\right)ds
$$
\begin{align}\label{tc4}
+\int_{0}^{t}\left(\rho^\varepsilon(s) \mathbf{u}^{\varepsilon}(s), \nabla_y\chi\left(x,\frac{x}{\varepsilon}\right)\right)ds.
\end{align}
For the first term on the right-hand side of \eqref{tc4}, by \eqref{tc3} we have for $\mathrm{P}$ a.s.  $\omega$ as $\varepsilon\rightarrow 0$
\begin{align}\int_{0}^{t}(\rho^\varepsilon(s) \mathbf{u}^{\varepsilon}(s), \nabla_x \phi(x))ds\rightarrow\int_{0}^{t}(\rho(s) \mathbf{u}(s), \nabla_x \phi(x))ds.\end{align}
For the second term on the right-hand side of \eqref{tc4}, we have  $\mathrm{P}$ a.s. 
$$
\left|\int_{0}^{t}\left(\rho^\varepsilon(s) \mathbf{u}^{\varepsilon}(s), \varepsilon\nabla_x\chi\left(x,\frac{x}{\varepsilon}\right)\right)ds\right|$$
$$\leq \varepsilon\left\|\nabla_x\chi\left(x,\frac{x}{\varepsilon}\right)\right\|_{L^2(\mathcal{O}; L^2_{per}( D))}\int_{0}^{t}\left\|\rho^\varepsilon(s) \mathbf{u}^{\varepsilon}(s)\right\|_{L^2(\mathcal{O})}ds.
$$
From \eqref{1.23***} and \eqref{4.1*} we see $\int_{0}^{t}\left\|\rho^\varepsilon(s) \mathbf{u}^{\varepsilon}(s)\right\|_{L^2(\mathcal{O})}ds<\infty$, $\mathrm{P}$ a.s. Then the right-hand side term converges to zero as $\varepsilon\rightarrow 0$, thus  $\mathrm{P}$ a.s. 
\begin{align}\label{tc5}
\int_{0}^{t}\left(\rho^\varepsilon(s) \mathbf{u}^{\varepsilon}(s), \varepsilon\nabla_x\chi\left(x,\frac{x}{\varepsilon}\right)\right)ds\rightarrow 0.\end{align}
For the third term on the right-hand side of \eqref{tc4}, by \eqref{tc3} we have $\mathrm{P}$ a.s. 
$$\int_{0}^{t}\left(\rho^\varepsilon(s) \mathbf{u}^{\varepsilon}(s), \nabla_y\chi\left(x,\frac{x}{\varepsilon}\right)\right)ds
$$
$$
\rightarrow\int_{\mathcal{O}_s}\int_{D}\rho(s) \mathbf{u}(s)\nabla_y\chi\left(x,y\right)dxdyds
$$
\begin{align}\label{tc6}
=\int_{\mathcal{O}_s}\rho(s) \mathbf{u}(s)\left(\int_{D}\nabla_y\chi\left(x,y\right)dy\right)  dxds=0.
\end{align}
Using \eqref{tc4}-\eqref{tc6} we obtain for $\mathrm{P}$ a.s.  $\omega$ as $\varepsilon\rightarrow 0$
\begin{align}\label{tc7}\int_{0}^{t}(\rho^\varepsilon(s) \mathbf{u}^{\varepsilon}(s), \nabla \Phi^\varepsilon)ds\rightarrow\int_{0}^{t}(\rho(s) \mathbf{u}(s), \nabla_x \phi(x))ds.\end{align}
By \eqref{p1}, \eqref{p2} and \eqref{tc7}, we see $\mathrm{P}$ a.s.
$$(\rho^{\varepsilon}(t), \Phi^\varepsilon)-(\rho_0, \Phi^\varepsilon)-\int_{0}^{t}(\rho^\varepsilon(s) \mathbf{u}^{\varepsilon}(s), \nabla\Phi^\varepsilon)ds
$$
$$
\rightarrow (\rho(t), \phi(x))- (\rho_0, \phi(x))-\int_{0}^{t}(\rho(s) \mathbf{u}(s), \nabla_x \phi(x))ds=0.
$$

$\mathbf{Step~ 2}$. We proceed to pass to the limit of the momentum equation. 
Define the functional
$$
(\rho^\varepsilon, \mathbf{u}^{\varepsilon})\mapsto\mathcal{G}_t^\varepsilon(\rho^\varepsilon, \mathbf{u}^{\varepsilon}):=
(\rho^\varepsilon(t)\mathbf{u}^\varepsilon(t), \Psi^\varepsilon)-(\rho_0\mathbf{u}_0, \Psi^\varepsilon)+\int_0^t(A^\varepsilon \mathbf{u}^\varepsilon(s), \Psi^\varepsilon)_{ V'\times V}ds
$$
$$-\int_0^t(\rho^\varepsilon(s)\mathbf{u}^\varepsilon(s)\otimes \mathbf{u}^\varepsilon(s), \nabla\Psi^\varepsilon)ds
-\int_{0}^{t}(f^\varepsilon( \mathbf{u}^{\varepsilon}(s)), \Psi^\varepsilon)ds.
$$
Then, by the momentum equation we know  $\mathrm{P}$ a.s.
$$
\mathcal{G}_t^\varepsilon(\rho^\varepsilon, \mathbf{u}^{\varepsilon})=\int_{0}^{t}(g(\mathbf{u}^{\varepsilon}(s))dW^\varepsilon,  \Psi^\varepsilon),
$$
which is a square integrable martingale with quadratic variation
$$\langle\!\langle\mathcal{G}_t^\varepsilon(\rho^\varepsilon, \mathbf{u}^{\varepsilon}) \rangle\!\rangle=\sum_{k\geq 1}\int_{0}^{t}(g(\mathbf{u}^\varepsilon(s))Q^\frac{1}{2}\mathbf{e}_k, \Psi^\varepsilon)^2ds.$$
Also, define the functional
\begin{align*}
&(\rho, \mathbf{u})\mapsto \mathcal{G}_t(\rho, \mathbf{u}):=(\rho(t)\mathbf{u}(t), \varphi)- (\rho_0\mathbf{u}_0, \varphi)  \nonumber\\
&\quad+\sum_{i,j=1}^d\int_{\mathcal{O}_s}\int_{D}a_{i,j}(y, s)\left(\frac{\partial \mathbf{u}(x,s)}{\partial x_i}+\frac{\partial \overline{\mathbf{u}}(x,y,s)}{\partial y_i}\right)\left(\frac{\partial \varphi(x)}{\partial x_j}+\frac{\partial \psi(x,y)}{\partial y_j}\right)dxdyds\nonumber\\
&\quad-\int_{0}^{t}(\rho(s)\mathbf{u}(s)\otimes \mathbf{u}(s), \nabla_x\varphi)ds-\int_{\mathcal{O}_s}\int_{D}f(y, s, \mathbf{u}(s)) \varphi dxdyds.
\end{align*}
Our goal is to show that
$$\mathcal{G}_t(\rho, \mathbf{u})=\int_{0}^{t}(g(\mathbf{u}(s))dW,  \varphi),$$
which can be obtained once we show that $\mathcal{G}_t(\rho, \mathbf{u})$ is a martingale with quadratic variation
$$
\langle\!\langle \mathcal{G}_t(\rho, \mathbf{u}) \rangle\!\rangle=\sum_{k\geq 1}\int_{0}^{t}(g(\mathbf{u}(s))Q^\frac{1}{2}\mathbf{e}_k, \varphi)^2ds,
$$
and the cross variation
$$
\langle\!\langle \mathcal{G}_t(\rho, \mathbf{u}),  W_k\rangle\!\rangle=\int_{0}^{t}(g(\mathbf{u}(s))Q^\frac{1}{2}\mathbf{e}_k, \varphi)ds.
$$

For the first term in the momentum equation, by $\rho^\varepsilon\mathbf{u}^\varepsilon\rightarrow \rho \mathbf{u}$ in $C([0,T]; W^{-\alpha, 2}(\mathcal{O}))$, \eqref{1.23***}, \eqref{4.1*}, we could have in the same way as \eqref{p1} and \eqref{p2} for $\mathrm{P}$ a.s. $\omega$
\begin{align}\label{tc8}
	(\rho^\varepsilon(t)\mathbf{u}^{\varepsilon}(t), \Psi^\varepsilon)-(\rho_0\mathbf{u}_0,  \Psi^\varepsilon)\rightarrow (\rho(t)\mathbf{u}(t), \varphi(x))-(\rho_0\mathbf{u}_0, \varphi(x)).
\end{align}
For the diffusion term, we see
$$\int_{0}^{t}(A^\varepsilon \mathbf{u}^\varepsilon(s), \Psi^\varepsilon)_{V'\times V}ds=-\int_{0}^{t}\sum_{i,j=1}^d\left(a_{i,j}\left(\frac{x}{\varepsilon}, s\right) \nabla\mathbf{u}^\varepsilon(s), \nabla\Psi^\varepsilon\right)ds$$
$$=-\int_{0}^{t}\sum_{i,j=1}^d\left(a^\varepsilon_{i,j}(x,s)\frac{\partial\mathbf{u}^\varepsilon(x,s)}{\partial x_i}, \frac{\partial\varphi(x)}{\partial x_j}+\frac{\partial\psi\left(x,\frac{x}{\varepsilon}\right)}{\partial y_j}+\varepsilon\frac{\partial\psi\left(x,\frac{x}{\varepsilon}\right)}{\partial x_j}\right)ds.$$
Using \eqref{4.4} we get $\mathrm{P}$ a.s.
$$
\int_{0}^{t}\sum_{i,j=1}^d\left(a^\varepsilon_{i,j}(x,s)\frac{\partial\mathbf{u}^\varepsilon(x,s)}{\partial x_i}, \frac{\partial\varphi(x)}{\partial x_j}+\frac{\partial\psi\left(x,\frac{x}{\varepsilon}\right)}{\partial y_j}\right)ds\rightarrow
$$
$$
 \int_{\mathcal{O}_s}\int_{D}\sum_{i,j=1}^da_{i,j}(y,s)\left(\frac{\partial \mathbf{u}(x,s)}{\partial x_i}+\frac{\partial \overline{\mathbf{u}}(x,y,s)}{\partial y_i}\right)\left(\frac{\partial \varphi(x)}{\partial x_j}+\frac{\partial \psi(x,y)}{\partial y_j}\right)dxdyds.
$$
Moreover, by Lemma \ref{lemapx.1} and $a_{i,j}\in L^\infty (\mathbb{R}^d_y\times \mathbb{R}^+)$ we obtain $\mathrm{P}$ a.s.
$$
\int_{0}^{t}\sum_{i,j=1}^d\left(a^\varepsilon_{i,j}(x,s)\frac{\partial\mathbf{u}^\varepsilon(x,s)}{\partial x_i}, \varepsilon\frac{\partial\psi\left(x,\frac{x}{\varepsilon}\right)}{\partial x_j}\right)ds\rightarrow 0.
$$
We obtain for $\mathrm{P}$ a.s. $\omega$ as $\varepsilon\rightarrow 0$
$$
\int_{0}^{t}(A^\varepsilon \mathbf{u}^\varepsilon(s), \Psi^\varepsilon)_{V'\times V}ds\rightarrow
$$
\begin{align}\label{tc8*} -\int_{\mathcal{O}_s}\int_{D}\sum_{i,j=1}^da_{i,j}(y,s)\left(\frac{\partial \mathbf{u}(x,s)}{\partial x_i}+\frac{\partial \overline{\mathbf{u}}(x,y,s)}{\partial y_i}\right)\left(\frac{\partial \varphi(x)}{\partial x_j}+\frac{\partial \psi(x,y)}{\partial y_j}\right)dxdyds.
\end{align}

By
using  \eqref{tc2}, \eqref{tc3} and Lemma \ref{lem4.3}, we infer $\mathrm{P}$ a.s.
\begin{align*}
\rho^\varepsilon\mathbf{u}^\varepsilon\otimes \mathbf{u}^\varepsilon\rightarrow \rho\mathbf{u}\otimes \mathbf{u}, ~{\rm weak}-\Sigma ~{\rm in}~ L^\frac{p+1}{p}( \mathcal{O}_s),
\end{align*}
which follows
$$
\int_{0}^{t}\left(\rho^\varepsilon(s)\mathbf{u}^\varepsilon(s)\otimes \mathbf{u}^\varepsilon(s), \nabla_y\psi\left(x,\frac{x}{\varepsilon}\right)\right)ds$$
$$
\rightarrow\int_{\mathcal{O}_s}\int_{D}(\rho(s)\mathbf{u}(s)\otimes \mathbf{u}(s)) \nabla_y\psi\left(x,y\right)dxdyds$$
\begin{align}\label{tc10}
=\int_{\mathcal{O}_s}\rho(s)\mathbf{u}(s)\otimes \mathbf{u}(s)\left(\int_{D}\nabla_y\psi\left(x,y\right)dy\right) dx ds=0.
\end{align}
Moreover, by \eqref{1.23***} and \eqref{4.1*} we also have for $\mathrm{P}$ a.s. $\omega$  as $\varepsilon\rightarrow 0$
\begin{align}\int_{0}^{t}\left(\rho^\varepsilon(s)\mathbf{u}^\varepsilon(s)\otimes \mathbf{u}^\varepsilon(s), \varepsilon\nabla_x\psi\left(x,\frac{x}{\varepsilon}\right)\right)ds\rightarrow 0,
\end{align}
and
\begin{align}\label{tc11}
\int_{0}^{t}\left(\rho^\varepsilon(s)\mathbf{u}^\varepsilon(s)\otimes \mathbf{u}^\varepsilon(s), \nabla_x \varphi(x)\right)ds\rightarrow\int_{0}^{t}\left(\rho(s)\mathbf{u}(s)\otimes \mathbf{u}(s), \nabla_x \varphi(x)\right)ds.
\end{align}
From \eqref{tc10}-\eqref{tc11}, we arrive at for $\mathrm{P}$ a.s. $\omega$ as $\varepsilon\rightarrow 0$
\begin{align}\label{tc11*}
\int_{0}^{t}\left(\rho^\varepsilon(s)\mathbf{u}^\varepsilon(s)\otimes \mathbf{u}^\varepsilon(s), \nabla \Psi^\varepsilon\right)ds\rightarrow\int_{0}^{t}\left(\rho(s)\mathbf{u}(s)\otimes \mathbf{u}(s), \nabla_x \varphi(x)\right)ds.
\end{align}
By $\mathbf{u}^\varepsilon\rightarrow \mathbf{u}$ in $L^2(\mathcal{O}_s)$, $\mathrm{P}$ a.s., \cite[Lemma 7]{raz} and $(\mathbf{A.1})$ we obtain
\begin{align}\label{4.28} f^\varepsilon(\mathbf{u}^\varepsilon)\rightarrow f(\cdot, \cdot, \mathbf{u}), ~{\rm weak}-\Sigma,~ {\rm in}~L^2( \mathcal{O}_s), \end{align}
which implies
\begin{align}\label{tc12}
\int_{0}^{t}(f^\varepsilon( \mathbf{u}^{\varepsilon}(s)), \Psi^\varepsilon)ds\rightarrow\int_{\mathcal{O}_s}\int_{D}f(y, s, \mathbf{u}(s)) \varphi(x) dxdyds.
\end{align}
Combining \eqref{tc8}, \eqref{tc8*}, \eqref{tc11*} and \eqref{tc12}, we get for $\mathrm{P}$ a.s. $\omega$
\begin{align}\label{2.15}
\mathcal{G}_t^\varepsilon(\rho^\varepsilon, \mathbf{u}^\varepsilon)\rightarrow\mathcal{G}_t(\rho, \mathbf{u}), ~{\rm as} ~\varepsilon\rightarrow 0.
\end{align}
By the Burkholder-Davis-Gundy inequality,  $(\mathbf{A.4})$ and \eqref{1.23**} we have for any $p>2$
\begin{align*}
\mathrm{E}\left(|\mathcal{G}_t^\varepsilon(\rho^\varepsilon, \mathbf{u}^\varepsilon)|^p\right)\leq C\mathrm{E}\left|\int_{0}^{t}\sum_{k\geq 1}(g( \mathbf{u}^\varepsilon(r))Q^\frac{1}{2}\mathbf{e}_k, \mathbf{u}^\varepsilon(r))^2dr\right|^\frac{p}{2}\lesssim_{m, M, c_2, c_4, \kappa, p, T} C,
\end{align*}
where $C>0$ is independent of $\varepsilon$. Then, by \eqref{2.15} we infer from the Vitali convergence theorem 
\begin{align}\label{4.24**}\mathrm{E}(\left|\mathcal{G}_t^\varepsilon(\rho^\varepsilon, \mathbf{u}^\varepsilon)-\mathcal{G}_t(\rho, \mathbf{u})\right|)\rightarrow 0, ~{\rm as} ~\varepsilon\rightarrow 0.\end{align}

Let
$\mathbf{h}$ be any bounded continuous functional on $X|_{[0,s]} \times C([0, s]; H_0)$, by \eqref{4.24**} and the martingale property we have
$$\mathrm{E}\left((\mathcal{G}_t(\rho, \mathbf{u})-\mathcal{G}_s(\rho, \mathbf{u}))\mathbf{h}\left((\rho, \mathbf{u}, \rho \mathbf{u}, W)|_{[0,s]}\right)
\right)$$\begin{align*}=\lim_{\varepsilon\rightarrow 0}\mathrm{E}\left((\mathcal{G}_t^\varepsilon(\rho^\varepsilon, \mathbf{u}^\varepsilon)-\mathcal{G}_s^\varepsilon(\rho^\varepsilon, \mathbf{u}^\varepsilon))\mathbf{h}\left((\rho^{\varepsilon}, \mathbf{u}^{\varepsilon}, \rho^{\varepsilon}\mathbf{u}^{\varepsilon}, W^{\varepsilon})|_{[0,s]}\right)
	\right)=0.
\end{align*}
The arbitrariness of $\mathbf{h}$ implies
\begin{align}\label{2.17}
\mathrm{E}(\mathcal{G}_t(\rho, \mathbf{u})|\mathscr{F}_s)=\mathcal{G}_s(\rho, \mathbf{u}),
\end{align}
where the  filtration $\{\mathscr{F}_t\}_{t\geq 0}$ is generated by $\sigma\{\rho(s), \mathbf{u}(s), \rho(s)\mathbf{u}(s), W(s), s\leq t\}$ satisfying the usual conditions.

We proceed to show that
$$\mathrm{E}\left((\mathcal{G}_t(\rho, \mathbf{u}))^2-\sum_{k\geq 1}\int_{0}^{t}(g(\mathbf{u}(r))Q^\frac{1}{2}\mathbf{e}_k, \varphi)^2dr\bigg|\mathscr{F}_s\right)$$
\begin{align}\label{2.18}
=(\mathcal{G}_s(\rho, \mathbf{u}))^2-\sum_{k\geq 1}\int_{0}^{s}(g(\mathbf{u}(r))Q^\frac{1}{2}\mathbf{e}_k, \varphi)^2dr.
\end{align}
Again, the Vitali convergence theorem implies
\begin{align}\label{4.24*}\mathrm{E}(\mathcal{G}_t^\varepsilon(\rho^\varepsilon, \mathbf{u}^\varepsilon)-\mathcal{G}_t(\rho, \mathbf{u}))^2\rightarrow 0, ~{\rm as} ~\varepsilon\rightarrow 0.\end{align}
We also need to show that
\begin{align}\label{4.26*}\mathrm{E}\left|\sum_{k\geq 1}\int_{0}^{t}(g(\mathbf{u}^\varepsilon(r))Q^\frac{1}{2}\mathbf{e}_k, \Psi^\varepsilon)^2
-(g(\mathbf{u}(r))Q^\frac{1}{2}\mathbf{e}_k, \varphi)^2dr\right|\rightarrow 0,\end{align}
as $\varepsilon\rightarrow 0$.
From  $({\bf A.4})$ and \eqref{1.23**}, we have
$$\mathrm{E}\left|\sum_{k\geq 1}\int_{0}^{t}(g(\mathbf{u}^\varepsilon(r))Q^\frac{1}{2}\mathbf{e}_k,\varepsilon\psi)^2dr\right|\leq \varepsilon\mathrm{E}\left|\int_{0}^{t}\|\psi\|^2_{L^2(\mathcal{O}; L^2_{per}(D))}\|g(\mathbf{u}^\varepsilon(r))\|^2_{L_2(H;H)}dr\right|$$
\begin{align}\label{4.27*}\leq \varepsilon c_4\|\psi\|^2_{L^2(\mathcal{O}; L^2_{per}(D))}\mathrm{E}\left(\int_{0}^{t}(1+\|\mathbf{u}^\varepsilon(r)\|_{H}^2)dr\right)\rightarrow 0,\end{align}
as $\varepsilon\rightarrow 0$. It remains to show that
\begin{align}\label{4.22}\mathrm{E}\left|\sum_{k\geq 1}\int_{0}^{t}(g(\mathbf{u}^\varepsilon(r))Q^\frac{1}{2}\mathbf{e}_k, \varphi)^2
-(g(\mathbf{u}(r))Q^\frac{1}{2}\mathbf{e}_k, \varphi)^2dr\right|\rightarrow 0.\end{align}
By \eqref{2.1} and $({\bf A.3})$, we have $\mathrm{P}$ a.s.
$$\left|\sum_{k\geq 1}\int_{0}^{t}(g(\mathbf{u}^\varepsilon(r))Q^\frac{1}{2}\mathbf{e}_k, \varphi)^2
-(g(\mathbf{u}(r))Q^\frac{1}{2}\mathbf{e}_k, \varphi)^2dr\right|$$
$$\leq \int_{0}^{t}\|\varphi\|^2_{C^\infty_{0, div}(\mathcal{O})}\|g(\mathbf{u}^\varepsilon(r))-g(\mathbf{u}(r))\|_{L_2(H;H)}(\|g(\mathbf{u}^\varepsilon(r))\|_{L_2(H;H)}+\|g(\mathbf{u}(r))\|_{L_2(H;H)})dr$$
$$\leq C(c_3,c_4)\|\varphi\|^2_{C^\infty_{0, div}(\mathcal{O})}\left(\int_{0}^{t}\|\mathbf{u}^\varepsilon(r)-\mathbf{u}(r)\|^2_{H}dr\right)^\frac{1}{2}
\left(\int_{0}^{t}(1+\|\mathbf{u}^\varepsilon(r)\|_{H}^2+\|\mathbf{u}(r)\|^2_{H})dr\right)^\frac{1}{2}$$
$$\lesssim_{\omega, m, M, \kappa,  T, c_3,c_4} C\|\varphi\|^2_{C^\infty_{0, div}(\mathcal{O})}\left(\int_{0}^{t}\|\mathbf{u}^\varepsilon(r)-\mathbf{u}(r)\|^2_{H}dr\right)^\frac{1}{2}\rightarrow 0.$$
The dominated convergence theorem gives
$$\mathrm{E}\left|\sum_{k\geq 1}\int_{0}^{t}(g(\mathbf{u}^\varepsilon(r))Q^\frac{1}{2}\mathbf{e}_k, \varphi)^2
-(g(\mathbf{u}(r))Q^\frac{1}{2}\mathbf{e}_k, \varphi)^2dr\right|\rightarrow 0,$$
as desired.  \eqref{4.26*} is a consequence of \eqref{4.27*} and \eqref{4.22}.

Using \eqref{4.24*} and \eqref{4.26*}, we further obtain
\begin{align}\label{2.18*}
&\mathrm{E}\left(\!\left(\left(\mathcal{G}_t^\varepsilon(\rho^\varepsilon, \mathbf{u}^\varepsilon)\right)^2-\left(\mathcal{G}_s^\varepsilon(\rho^\varepsilon, \mathbf{u}^\varepsilon)\right)^2
-\sum_{k\geq 1}\int_{s}^{t}(g(\mathbf{u}^\varepsilon(r))Q^\frac{1}{2}\mathbf{e}_k, \Psi^\varepsilon)^2dr
\right)\mathbf{h}\left((\rho^\varepsilon, \mathbf{u}^\varepsilon, \rho^{\varepsilon}\mathbf{u}^{\varepsilon}, W^\varepsilon)|_{[0,s]}\right)\right)\nonumber\\
&\rightarrow\mathrm{E}\left(\!\left(\left(\mathcal{G}_t(\rho, \mathbf{u})\right)^2-\left(\mathcal{G}_s(\rho, \mathbf{u})\right)^2-\sum_{k\geq 1}\int_{s}^{t}(g(\mathbf{u}(r))Q^\frac{1}{2}\mathbf{e}_k, \varphi)^2dr\right)\!\mathbf{h}\left((\rho, \mathbf{u}, \rho\mathbf{u}, W)|_{[0,s]}\right)\!\right).
\end{align}
By the martingale property, we deduce
$$\mathrm{E}\left(\!\left(\left(\mathcal{G}_t^\varepsilon(\rho^\varepsilon, \mathbf{u}^\varepsilon)\right)^2-\left(\mathcal{G}_s^\varepsilon(\rho^\varepsilon, \mathbf{u}^\varepsilon)\right)^2
-\sum_{k\geq 1}\int_{s}^{t}(g(\mathbf{u}^\varepsilon(r))Q^\frac{1}{2}\mathbf{e}_k, \Psi^\varepsilon)^2dr
\right)\mathbf{h}\left((\rho^\varepsilon, \mathbf{u}^\varepsilon, \rho^{\varepsilon}\mathbf{u}^{\varepsilon}, W^\varepsilon)|_{[0,s]}\right)\right)=0,$$
then, by \eqref{2.18*} we further obtain
$$\mathrm{E}\left(\!\left(\left(\mathcal{G}_t(\rho, \mathbf{u})\right)^2-\left(\mathcal{G}_s(\rho, \mathbf{u})\right)^2-\sum_{k\geq 1}\int_{s}^{t}(g(\mathbf{u}(r))Q^\frac{1}{2}\mathbf{e}_k, \varphi)^2dr
\right)\!\mathbf{h}\left((\rho, \mathbf{u}, \rho\mathbf{u}, W)|_{[0,s]}\right)\!\right)=0.$$
The arbitrariness of $\mathbf{h}$ yields \eqref{2.18}.

Moreover, by an easier argument than \eqref{2.18} we have
$$
\mathrm{E}\left(\mathcal{G}_t(\rho, \mathbf{u})W_k(t)-\int_{0}^{t}(g(\mathbf{u}(r))Q^\frac{1}{2}\mathbf{e}_k, \varphi)dr\bigg|\mathscr{F}_s\right)$$
\begin{align}\label{4.37}
=\mathcal{G}_s(\rho, \mathbf{u})W_k(s)-\int_{0}^{s}(g(\mathbf{u}(r))Q^\frac{1}{2}\mathbf{e}_k, \varphi)dr.
\end{align}

By \eqref{2.17}, \eqref{2.18}, \eqref{4.37} and \eqref{1.23**}, we finally infer from the martingale representation theory
\begin{align*}
\mathcal{G}_t(\rho, \mathbf{u})=\int_{0}^{t}(g(\mathbf{u}(r))dW, \varphi).
\end{align*}
By the density of $C_{0, div}^\infty(\mathcal{O}),  C_{0,div}^\infty(\mathcal{O})\times C^\infty_{per}( D) $ in $V, V(\mathcal{O}; V_{per})$ and $C_{0}^\infty(\mathcal{O})$ in $ H^1(\mathcal{O})$, we complete the proof of Proposition \ref{pro2.1}.
\end{proof}

{\bf Recover the representation of $\overline{\mathbf{u}}$}. We are going to give the specific expression of the corrector $\overline{\mathbf{u}}$.

\begin{lemma}\label{lem4.6} The corrector $\overline{\mathbf{u}}$ is given by
\begin{align*}
\overline{\mathbf{u}}(x,y,t)=-\sum_{i,k=1}^d \frac{\partial \mathbf{u}}{\partial x_i}(x,t)\eta_{i,k}(y,t),~ \mathrm{P} ~a.s.
\end{align*}
where $\eta_{i,k}$ is the solution of variational problem
\begin{align}\label{var1}
\begin{cases}
\mathcal{K}(\eta_{i,k}, \mathbf{w})=\sum_{j=1}^d \int_{D}a_{i,j}\frac{\partial\mathbf{w}^k}{\partial y_j}dy,\\\\
\int_{D}\eta_{i,k}dy=0,
\end{cases}
\end{align}
for any $\mathbf{w}\in V_{per}$, $\mathrm{P}\times \mathcal{G}, a.e. ~(\omega, y)$, $\mathcal{G}$ is the Lebesgue measure and the bilinear operator $\mathcal{K}$ is defined by
$$\mathcal{K}(\mathbf{v}, \mathbf{w})=\sum_{i,j=1}^d\int_{D}a_{i,j}\frac{\partial \mathbf{v}}{\partial y_i}\frac{\partial \mathbf{w}}{\partial y_j}dy.$$
\end{lemma}
\begin{proof} Similar to \cite[Lemma 4.5]{chen}, choosing $\varphi=0$ in equations \eqref{4.3}, and $\psi=\mathbf{w}$ for $ \mathbf{w}\in V_{per}$, we have
\begin{align}\label{v3}
\mathcal{K}(\overline{\mathbf{u}}, \mathbf{w})=\sum_{i, j, k=1}^d \int_{\mathcal{O}_t} \frac{\partial\mathbf{u}}{\partial x_i}\left(\int_{D}a_{i,j}\frac{\partial\mathbf{w}^k}{\partial y_j}dy \right) dxds, ~ \mathrm{P} ~a.s.
\end{align}
For the existence of solutions to the variational  problem \eqref{var1}, the readers are  referred to \cite{lions}. We sketch the proof of uniqueness. Assume that $\mathbf{v}_1, \mathbf{v}_2$ are two solutions, then
$$\mathcal{K}(\mathbf{v}_1, \mathbf{w})-\mathcal{K}(\mathbf{v}_2, \mathbf{w})=\sum_{i,j=1}^d\int_{D}a_{i,j}\frac{\partial (\mathbf{v}_1-\mathbf{v}_2)}{\partial y_i}\frac{\partial \mathbf{w}}{\partial y_j}dy=0.$$
Let $\mathbf{w}=\mathbf{v}_1-\mathbf{v}_2$, we see
$$\mathcal{K}(\mathbf{v}_1, \mathbf{w})-\mathcal{K}(\mathbf{v}_2, \mathbf{w})=\sum_{i,j=1}^d\int_{D}a_{i,j}\frac{\partial (\mathbf{v}_1-\mathbf{v}_2)}{\partial y_i}\cdot
\frac{\partial (\mathbf{v}_1-\mathbf{v}_2)}{\partial y_j}  dy=0,$$
which along with $$\sum_{i,j=1}^d\int_{D}a_{i,j}\frac{\partial (\mathbf{v}_1-\mathbf{v}_2)}{\partial y_i}\cdot \frac{\partial (\mathbf{v}_1-\mathbf{v}_2)}{\partial y_j} dy\geq \kappa\|\mathbf{v}_1-\mathbf{v}_2\|^2_{V_{per}}\geq c\|\mathbf{v}_1-\mathbf{v}_2\|^2_{ L^2_{per}(D)},$$
leads to $\mathbf{v}_1=\mathbf{v}_2$. Note that the process $\overline{\mathbf{u}}$ is a solution to the variational problem. Compared with \eqref{var1}, we find that
$$\widetilde{\mathbf{v}}=-\sum_{i,k=1}^d \frac{\partial \mathbf{u}}{\partial x_i}(x,t)\eta_{i,k}(y,t),  ~ \mathrm{P} ~a.s.$$
is also a solution of \eqref{v3}. We obtain $\overline{\mathbf{u}}=\widetilde{\mathbf{v}}$ from the uniqueness.
\end{proof}

\textbf{Proof of Theorem \ref{thm2.1}}. Denote the function
\begin{align}\label{t2.1}\mathbf{a}_{i,j,k,l}=\int_{D}a_{i,j}(y,s)dy-\int_{D}a_{i,j}(y, s)\frac{\partial\eta^l_{i,k}(y,s)}{\partial y_j}dy,\end{align}
for $1\leq i,j,k,l \leq d$. Corresponding to the function, we denote by $\overline{A}=(\overline{A}_{kl})_{k,l=1,\cdots, d}$ the differential homogenized operator
\begin{align}\label{4.39}\overline{A}_{kl}=-\sum_{i,j=1}^d\mathbf{a}_{i,j,k,l}\frac{\partial^2}{\partial x_i \partial x_j},~~ k,l=1,2,\cdots,d.\end{align}
From Lemma \ref{lem4.6} and Proposition \ref{pro2.1}, we finally obtain $(\rho, \rho\mathbf{u})$ satisfies the homogenized Navier-Stokes equations in Theorem \ref{thm2.1}.
We emphasize that the homogenized operator $\overline{A}$ also satisfies the condition of uniform ellipticity see \cite{ben}, thus, there exists constant $\kappa>0$ such that
$$\sum_{i,j,k,l=1}^d\mathbf{a}_{i,j,k,l}\xi_{i,k}\xi_{j,l}\geq \kappa \sum_{k,l=1}^d|\xi_{k,l}|^2.$$
Moreover, we could easily verify that
$$|\overline{f}(\mathbf{u}_1)-\overline{f}(\mathbf{u}_2)|\leq c_1|\mathbf{u}_1-\mathbf{u}_2|.$$
Using the uniform ellipticity condition, ({\bf A.3})-({\bf A.4}) and the Lipschitz continuity of $\overline{f}$, we could infer that the homogenized Navier-Stokes equations admit a solution $(\rho, \rho\mathbf{u})$ with the regularity as in Proposition \ref{lem2.1*}. 

Using It\^{o}'s formula to $\frac{1}{2}\|\sqrt{\rho(t)} \mathbf{u}(t)\|_{L^2(\mathcal{O})}^2$, we obtain for every $t\in [0,T]$, and P a.s. $\omega$
\begin{align*}
&\| \sqrt{\rho(t)}\mathbf{u}(t)\|^2_{L^2(\mathcal{O})}- \| \sqrt{\rho_0}\mathbf{u}_0\|^2_{L^2(\mathcal{O})}\nonumber\\
&=2\int_{0}^{t}(\overline{A}\mathbf{u}(s), \mathbf{u}(s))_{V'\times V}ds+2\int_{0}^{t}\int_{\mathcal{O}}\int_{D}f(y,s, \mathbf{u}(s)) \mathbf{u}(s)dxdyds\nonumber\\
&\quad+2\int_{0}^{t}(g(\mathbf{u}(s))dW,  \mathbf{u}(s))+
\int_{0}^{t}\|g(\mathbf{u}(s))\|_{L_2(H;H)}^2ds.
\end{align*}
Then by \eqref{t2.1}, \eqref{4.39} and Lemma \ref{lem4.6} we see
\begin{align*}&\int_{0}^{t}(\overline{A}\mathbf{u}(s), \mathbf{u}(s))_{V'\times V}ds=\nonumber\\
&-\sum_{i,j=1}^d\int_{0}^{t}\int_{\mathcal{O}}\int_{D}a_{i,j}(y, s)\left(\frac{ \partial\mathbf{u}(x,s)}{\partial{x_i}}+\frac{ \partial\overline{\mathbf{u}}(x,y,s)}{\partial {y_i}}\right)\cdot \left(\frac{ \partial\mathbf{u}(x,s)}{\partial x_j}+\frac{ \partial\overline{\mathbf{u}}(x,y,s)}{\partial y_j}\right) dxdyds.
\end{align*}
Thus, we obtain the energy equation \eqref{1.4}. This completes the proof.    $\Box$

\section{A corrector result}
A corrector result is established in this section which strengthens the convergence of $\nabla \mathbf{u}^\varepsilon$ in $L^2(\Omega\times \mathcal{O}_t)$, weak-$\Sigma$ to the $L^2(\Omega\times \mathcal{O}_t)$, strong-$\Sigma$. We first establish a stochastic version of the lower semicontinuity.
\begin{lemma}\label{lem5.1} 
If the weak-$\Sigma$ in $L^2(\Omega; L^2(0,T;H))$ convergence of $\mathbf{v}^\varepsilon$ to $\mathbf{v}$ holds, and $b\in (L^\infty (\mathbb{R}^d_y\times \mathbb{R}^+))^{d\times d}$  is a symmetric matrix satisfying the periodicity and uniform ellipticity conditions, then we have
$$
\liminf_{\varepsilon\rightarrow 0}\mathrm{E}\left(\int_{\mathcal{O}_t}b\left(\frac{x}{\varepsilon},t\right)\mathbf{v}^\varepsilon(x,t)\cdot\mathbf{v}^\varepsilon(x,t)dxdt\right)
$$
$$\geq \mathrm{E}\left(\int_{\mathcal{O}_t}\int_{D}b(y, t)\mathbf{v}(x,y,t)\cdot\mathbf{v}(x,y,t)dxdydt\right).
$$
\end{lemma}
\begin{proof} Inspired by \cite[Section 7]{zhi}, we  choose $h^\varepsilon(x,t)=h_1\left(x, \frac{x}{\varepsilon}\right)h_2\left(t\right)1_{\mathcal{A}}(\omega)$, where $h_1(x,y)\in C_0^\infty(\mathcal{O})\times C_{per}^\infty(D)$, $h_2(t)\in C_0^\infty([0,T])$,  $\mathcal{A}\in \mathcal{B}(\Omega)$.
Then, by the uniform ellipticity condition of $b$ we see P a.s.
\begin{align}\label{5.1*}
	&0\leq  b\left(\frac{x}{\varepsilon},t\right)(\mathbf{v}^\varepsilon(x,t)-h^\varepsilon(x,t))\cdot(\mathbf{v}^\varepsilon(x,t)-h^\varepsilon(x,t))
	\nonumber\\
	&= b\left(\frac{x}{\varepsilon},t\right)\mathbf{v}^\varepsilon(x,t)\cdot\mathbf{v}^\varepsilon(x,t) -2b\left(\frac{x}{\varepsilon},t\right)\mathbf{v}^\varepsilon(x,t)h^\varepsilon(x,t)+b\left(\frac{x}{\varepsilon},t\right)h^\varepsilon(x,t)\cdot h^\varepsilon(x,t).
\end{align}
Furthermore, by \eqref{5.1*} and 	the weak-$\Sigma$  in $L^2(\Omega, L^2(0,T;H))$ convergence of $\mathbf{v}^\varepsilon$ to $\mathbf{v}$  we obtain
\begin{eqnarray*}
	&&	\liminf_{\varepsilon\rightarrow 0}\mathrm{E}\left(\int_{\mathcal{O}_t}b\left(\frac{x}{\varepsilon},t\right)\mathbf{v}^\varepsilon(x,t)\cdot\mathbf{v}^\varepsilon(x,t)dxdt\right)\nonumber\\
	&&\geq \liminf_{\varepsilon\rightarrow 0}\mathrm{E}\left(\int_{\mathcal{O}_t}2b\left(\frac{x}{\varepsilon},t\right)\mathbf{v}^\varepsilon(x,t)h^\varepsilon(x,t)-
	b\left(\frac{x}{\varepsilon},t\right)h^\varepsilon(x,t)\cdot h^\varepsilon(x,t)dxdt\right)\nonumber\\
	&&=\mathrm{E}\left(\int_{\mathcal{O}_t}\int_{D}2b\left(y,t\right)\mathbf{v}(x, y,t)h(x,y,t)-
	b\left(y,t\right)h(x,y,t)\cdot h(x,y,t)dxdydt\right),
\end{eqnarray*}
where $h(x,y,t)=h_1(x,y)h_2(t)1_{\mathcal{A}}(\omega)$.
Define the operator $\digamma:L^2(\Omega; L^2(\mathcal{O}_t; L^2_{per} (D)))\longrightarrow \mathbb{R}$ by
$$\digamma(h)=\mathrm{E}\left(\int_{\mathcal{O}_t}\int_{D}2b\left(y,t\right)\mathbf{v}(x,y,t)h(x,y,t)-
	b\left(y,t\right)h(x,y,t)\cdot h(x,y,t)dxdydt\right).$$
We see that  $\digamma$ is continuous with respect to $h$ due to $b\in (L^\infty (\mathbb{R}^d_y\times \mathbb{R}^+))^{d\times d}$. Then by taking $h=\mathbf{v}$, we obtain the desired result.
\end{proof}


We introduce the following two-scale convergence of stochastic version given in \cite[Theorem 4]{raz}, \cite{zhi}, which will be used for establishing the strong-$\Sigma$ convergence.
\begin{lemma}\label{lem4.2} Suppose that $\{\mathbf{u}^\varepsilon\}_{\varepsilon\in(0,1)}$ is a sequence of $L^2(0,T;V)$-valued random variables with the regularity
$$\mathrm{E}\left(\int_{0}^{T}\|\mathbf{u}^\varepsilon(t)\|_{V}^2dt\right)\leq C,$$
and the convergence
$$\mathbf{u}^\varepsilon\rightarrow \mathbf{u}, ~in ~L^2(\mathcal{O}_t),~ \mathrm{P}~~ a.s.$$
then, there exist a subsequence (still denoted by $\{\mathbf{u}^\varepsilon\}_{\varepsilon\in(0,1)}$) and a $L^2(\mathcal{O}_t; L^2_{per}(D))$-valued random variable $\overline{\mathbf{u}}$ such that
\begin{align*}
\frac{\partial \mathbf{u}^\varepsilon}{\partial x_i}\rightarrow \frac{\partial \mathbf{u}}{\partial x_i}+\frac{\partial \overline{\mathbf{u}}}{\partial y_i}, ~in ~ L^2(\Omega\times\mathcal{O}_t), ~weak-\Sigma.
\end{align*}
\end{lemma}

\begin{lemma}\label{lem5.4} Assume that for any $r, p,q>1$ with $\frac{1}{r}=\frac{1}{p}+\frac{1}{q}$, if the following two conditions hold

i. a sequence of $L^p(\mathcal{O}_t)$-valued random variables $\{\mathbf{u}^\varepsilon\}_{\varepsilon\in(0,1)}$ is weak-$\Sigma$ convergent to some certain $\mathbf{u}\in L^p(\Omega\times \mathcal{O}_t; L^p_{per}(D))$;

ii. a sequence of $L^q(\mathcal{O}_t)$-valued random variables $\{\mathbf{v}^\varepsilon\}_{\varepsilon\in(0,1)}$ is strong-$\Sigma$ convergent to some certain $\mathbf{v}\in L^q(\Omega\times \mathcal{O}_t; L^q_{per}(D))$.

Then, we have that the sequence of $\{\mathbf{u}^\varepsilon\mathbf{v}^\varepsilon\}_{\varepsilon\in (0,1)}$ is weak-$\Sigma$ convergent to $\mathbf{u}\mathbf{v}$ in $L^r(\Omega\times \mathcal{O}_t)$.
\end{lemma}

\begin{lemma}\label{lem5.3}  A sequence of $L^p(\mathcal{O}_t)$-valued random variables $\{\mathbf{u}^\varepsilon\}_{\varepsilon\in(0,1)}$ is  strong-$\Sigma$ convergent in $L^p(\Omega\times \mathcal{O}_t)$ if  there exists a certain $L^p(\mathcal{O}_t; L^p_{per}(D))$-valued random variable $\mathbf{u}$ such that

i. the weak-$\Sigma$ convergence holds;

ii. it satisfies
$$\|\mathbf{u}^\varepsilon\|_{L^p(\Omega\times \mathcal{O}_t)}\rightarrow \|\mathbf{u}\|_{L^p(\Omega\times \mathcal{O}_t; L^p_{per}(D))}.$$
\end{lemma}

We end the paper by showing Theorem \ref{thm2} on the basis of Lemma \ref{lem5.1}-Lemma \ref{lem5.3}. By Lemma \ref{lem4.2} and \eqref{1.23**} we have 
\begin{align}\label{w1}
\frac{\partial \mathbf{u}^\varepsilon}{\partial x_i}\rightarrow \frac{\partial \mathbf{u}}{\partial x_i}+\frac{\partial \overline{\mathbf{u}}}{\partial y_i}, ~in ~ L^2(\Omega\times\mathcal{O}_t), ~weak-\Sigma.
\end{align}
We then improve the weak-$\Sigma$ convergence to the strong-$\Sigma$ convergence.

$\mathbf{Proof~ of~ Theorem ~\ref{thm2}}$. We first show that for every $t\in [0,T]$
\begin{align}\label{5.1}
\mathrm{E}\left(\|\sqrt{\rho^\varepsilon(t) }\mathbf{u}^\varepsilon(t)\|_{L^2(\mathcal{O})}^2\right)\rightarrow\mathrm{E}\left(\|\sqrt{\rho(t) }\mathbf{u}(t)\|_{L^2(\mathcal{O})}^2\right).
\end{align}
Since $\{\rho^{\varepsilon}\mathbf{u}^{\varepsilon}\}_{\varepsilon\in(0,1)}$ is uniformly bounded in 
$L^{p}\!\left(\Omega;L^{2}(\mathcal{O})\right)$ for any $p\ge 2$,  we have the weak convergence
$$\rho^\varepsilon\mathbf{u}^\varepsilon\rightharpoonup \rho\mathbf{u},~ {\rm in}~ L^p(\Omega; L^2(\mathcal{O})),$$
which implies that
$$\liminf_{\varepsilon\rightarrow 0}\mathrm{E}\left(\|\sqrt{\rho^\varepsilon(t) }\mathbf{u}^\varepsilon(t)\|_{L^2(\mathcal{O})}^2\right)\geq \mathrm{E}\left(\|\sqrt{\rho(t) }\mathbf{u}(t)\|_{L^2(\mathcal{O})}^2\right).$$
Then, the convergence \eqref{5.1} will follow from
\begin{align}\label{5.2}
\limsup_{\varepsilon\rightarrow 0}\mathrm{E}\left(\|\sqrt{\rho^\varepsilon(t) }\mathbf{u}^\varepsilon(t)\|_{L^2(\mathcal{O})}^2\right)\leq \mathrm{E}\left(\|\sqrt{\rho(t) }\mathbf{u}(t)\|_{L^2(\mathcal{O})}^2\right).
\end{align}

By \eqref{priori2.12} we have
\begin{eqnarray}\label{4.5}
	&&\|\sqrt{\rho^\varepsilon(t) }\mathbf{u}^\varepsilon(t)\|_{L^2(\mathcal{O})}^2+2\int_0^t(A^\varepsilon \mathbf{u}^\varepsilon(r), \mathbf{u}^\varepsilon(r))_{V'\times V}dr\nonumber\\
	&&=\|\sqrt{\rho_0 }\mathbf{u}_0\|_{L^2(\mathcal{O})}^2+2\int_{0}^{t}\left(f\left(\frac{x}{\varepsilon}, r, \mathbf{u}^\varepsilon(r)\right), \mathbf{u}^\varepsilon(r)\right)dr
	\nonumber\\
	&&\quad+\int_{0}^{t}\|g(\mathbf{u}^\varepsilon(r))\|_{L_2(H; H)}^2 dr+2\int_{0}^{t}(g(\mathbf{u}^\varepsilon(r))dW^\varepsilon, \mathbf{u}^\varepsilon(r)).
\end{eqnarray}
Taking expectation on both sides of \eqref{4.5} we obtain
$$
\mathrm{E}\left(\|\sqrt{\rho^\varepsilon(t) }\mathbf{u}^\varepsilon(t)\|_{L^2(\mathcal{O})}^2\right)+2\mathrm{E}\left(\int_0^t(A^\varepsilon \mathbf{u}^\varepsilon(r), \mathbf{u}^\varepsilon(r))_{V'\times V}dr\right)
$$
\begin{align}\label{5.3}
=\|\sqrt{\rho_0 }\mathbf{u}_0\|_{L^2(\mathcal{O})}^2+2\mathrm{E}\left(\int_{0}^{t}\left(f\left(\frac{x}{\varepsilon}, r, \mathbf{u}^\varepsilon(r)\right), \mathbf{u}^\varepsilon(r)\right)dr\right)+\mathrm{E}\left(\int_{0}^{t}\|g(\mathbf{u}^\varepsilon(r))\|_{L_2(H; H)}^2 dr\right).
\end{align}
For the second term on the left-hand side of \eqref{5.3}, using Lemma \ref{lem5.1} we see
$$\liminf_{\varepsilon\rightarrow 0}\mathrm{E}\left(\int_0^t(A^\varepsilon \mathbf{u}^\varepsilon(r), \mathbf{u}^\varepsilon(r))_{V'\times V}dr\right)$$
$$
\geq \mathrm{E}\left(\sum_{i,j=1}^d\int_{\mathcal{O}_r}\int_{D}a_{i,j}(y,r)\left(\frac{ \partial\mathbf{u}(x,r)}{\partial{x_i}}+\frac{ \partial\overline{\mathbf{u}}(x,y,r)}{\partial {y_i}}\right)\cdot \left(\frac{ \partial\mathbf{u}(x,r)}{\partial x_j}+\frac{ \partial\overline{\mathbf{u}}(x,y,r)}{\partial y_j}\right) dxdydr\right),
$$ for every $t \in [0,T].$
Then we have
$$-\limsup_{\varepsilon\rightarrow 0}\mathrm{E}\left(\int_0^t(A^\varepsilon \mathbf{u}^\varepsilon(r), \mathbf{u}^\varepsilon(r))_{V'\times V}dr\right)$$
\begin{align}\label{5.4*}\leq -\mathrm{E}\left(\sum_{i,j=1}^d\int_{\mathcal{O}_r}\int_{D}a_{i,j}(y, r)\left(\frac{ \partial\mathbf{u}(x,r)}{\partial{x_i}}+\frac{ \partial\overline{\mathbf{u}}(x,y,r)}{\partial {y_i}}\right)\cdot \left(\frac{ \partial\mathbf{u}(x,r)}{\partial x_j}+\frac{ \partial\overline{\mathbf{u}}(x,y,r)}{\partial y_j}\right) dxdydr\right).
\end{align}
By $\mathbf{u}^\varepsilon\rightarrow\mathbf{u}$ in $L^2(0,T;H)$, P a.s. and the uniform estimates $\mathbf{u}^\varepsilon \in L^p(\Omega; L^2(0,T;H))$ for $p>2$, we deduce from the Vitali convergence theorem 
\begin{align}\label{4.8}\mathbf{u}^\varepsilon\rightarrow\mathbf{u}, ~{\rm in} ~L^2(\Omega; L^2(0,T;H)).\end{align}
Then, for the second term on the right-hand side of \eqref{5.3}, using \eqref{4.28}, \eqref{4.8} and Lemma \ref{lem5.4} we have
\begin{align}\label{5.5*}
\lim_{\varepsilon\rightarrow 0}\mathrm{E}\left(\int_{0}^{t}\left(f\left(\frac{x}{\varepsilon}, r, \mathbf{u}^\varepsilon(r)\right), \mathbf{u}^\varepsilon(r)\right)dr\right)=\mathrm{E}\left(\int_{\mathcal{O}_r}\int_{D}f\left(y, r, \mathbf{u}(r)\right) \mathbf{u}(r)dxdydr\right).
\end{align}
For the last term on the right-hand side of \eqref{5.3}, by \eqref{4.8}  and $(\mathbf{A.3})$ we have
$$\left|\mathrm{E}\left(\int_{0}^{t}\|g(\mathbf{u}^\varepsilon(r))\|_{L_2(H; H)}^2 dr\right)-\mathrm{E}\left(\int_{0}^{t}\|g(\mathbf{u}(r))\|_{L_2(H; H)}^2 dr\right)\right|
$$
$$
\leq \mathrm{E}\left(\int_{0}^{t}\|g(\mathbf{u}^\varepsilon(r))-g(\mathbf{u}(r))\|_{L_2(H; H)}(\|g(\mathbf{u}^\varepsilon(r))\|_{L_2(H; H)}+\|g(\mathbf{u}(r))\|_{L_2(H; H)}) dr\right)
$$
$$
\leq c\mathrm{E}\left(\int_{0}^{t}\|\mathbf{u}^\varepsilon(r)-\mathbf{u}(r)\|_{H}(1+\|\mathbf{u}^\varepsilon(r)\|_{H}+\|\mathbf{u}(r)\|_{H}) dr\right)
$$
\begin{align}\label{5.6*}
\leq c\left(\mathrm{E}\left(\int_{0}^{t}\|\mathbf{u}^\varepsilon(r)-\mathbf{u}(r)\|_{H}^2dr\right)\right)^\frac{1}{2}
\left(\mathrm{E}\left(\int_{0}^{t}(1+\|\mathbf{u}^\varepsilon(r)\|_{H}+\|\mathbf{u}(r)\|_{H})^2 dr\right)\right)^\frac{1}{2}
\rightarrow 0,
\end{align}
where $c>0$ relies on $c_3, c_4$.

Combining \eqref{1.4} and \eqref{5.3}-\eqref{5.6*}, we have
\begin{eqnarray}
	&&\limsup_{\varepsilon\rightarrow 0}\mathrm{E}\left(\|\sqrt{\rho^\varepsilon(t) }\mathbf{u}^\varepsilon(t)\|_{L^2(\mathcal{O})}^2\right)\nonumber\\
	&&=\|\sqrt{\rho_0 }\mathbf{u}_0\|_{L^2(\mathcal{O})}^2-2\limsup_{\varepsilon\rightarrow 0}\mathrm{E}\left(\int_0^t(A^\varepsilon \mathbf{u}^\varepsilon(r), \mathbf{u}^\varepsilon(r))_{V'\times V}dr\right)\nonumber\\
	&&\quad+2\limsup_{\varepsilon\rightarrow 0}\mathrm{E}\left(\int_{0}^{t}\left(f\left(\frac{x}{\varepsilon}, r, \mathbf{u}^\varepsilon(r)\right), \mathbf{u}^\varepsilon(r)\right)dr\right)+\limsup_{\varepsilon\rightarrow 0}\mathrm{E}\left(\int_{0}^{t}\|g(\mathbf{u}^\varepsilon(r))\|_{L_2(H; H)}^2 dr\right)\nonumber\\
	&&\leq \|\sqrt{\rho_0 }\mathbf{u}_0\|_{L^2(\mathcal{O})}^2\nonumber\\
&&\quad-2\mathrm{E}\left(\sum_{i,j=1}^d\int_{\mathcal{O}_r}\int_{D}a_{i,j}(y, r)\left(\frac{ \partial\mathbf{u}(x,r)}{\partial{x_i}}+\frac{ \partial\overline{\mathbf{u}}(x,y,r)}{\partial {y_i}}\right)\cdot \left(\frac{ \partial\mathbf{u}(x,r)}{\partial x_j}+\frac{ \partial\overline{\mathbf{u}}(x,y,r)}{\partial y_j}\right)  dxdydr\right)\nonumber\\
	&&\quad+2\mathrm{E}\left(\int_{\mathcal{O}_r}\int_{D}f\left(y, r, \mathbf{u}(r)\right) \mathbf{u}(r)dxdydr\right)+\mathrm{E}\left(\int_{0}^{t}\|g(\mathbf{u}(r))\|_{L_2(H;H)}^2 dr\right)\nonumber\\
	&&=\mathrm{E}\left(\|\sqrt{\rho(t) }\mathbf{u}(t)\|_{L^2(\mathcal{O})}^2\right),
\end{eqnarray}
thus, \eqref{5.2} holds.

We next use \eqref{5.1} to prove the strong-$\Sigma$ convergence. From \eqref{5.3}, we have
\begin{align}\label{2.14*}
\mathrm{E}\left(\int_0^t(A^\varepsilon \mathbf{u}^\varepsilon(r), \mathbf{u}^\varepsilon(r))_{V'\times V}dr\right)=-\frac{1}{2}\mathrm{E}\left(\|\sqrt{\rho^\varepsilon(t) }\mathbf{u}^\varepsilon(t)\|_{L^2(\mathcal{O})}^2\right)+\frac{1}{2} \|\sqrt{\rho_0 }\mathbf{u}_0\|_{L^2(\mathcal{O})}^2\nonumber\\
+\mathrm{E}\left(\int_{0}^{t}\left(f\left(\frac{x}{\varepsilon}, r, \mathbf{u}^\varepsilon(r)\right), \mathbf{u}^\varepsilon(r)\right)dr\right)+\frac{1}{2}\mathrm{E}\left(\int_{0}^{t}\|g(\mathbf{u}^\varepsilon(r))\|_{L_2(H;H)}^2 dr\right).
\end{align}
By \eqref{2.14*}, we further have
\begin{eqnarray}
	&&\mathrm{E}\left(\int_0^t(A^\varepsilon (\mathbf{u}^\varepsilon(r)-\overline{\Psi}^\varepsilon), \mathbf{u}^\varepsilon(r)-\overline{\Psi}^\varepsilon)_{V'\times V}dr\right)\nonumber\\
	&&=-\frac{1}{2}\mathrm{E}\left(\|\sqrt{\rho^\varepsilon(t) }\mathbf{u}^\varepsilon(t)\|_{L^2(\mathcal{O})}^2\right)+\frac{1}{2} \|\sqrt{\rho_0 }\mathbf{u}_0\|_{L^2(\mathcal{O})}^2
	\nonumber\\
	&&\quad+\mathrm{E}\left(\int_{0}^{t}\left(f\left(\frac{x}{\varepsilon}, r, \mathbf{u}^\varepsilon(r)\right), \mathbf{u}^\varepsilon(r)\right)dr\right)+\frac{1}{2}\mathrm{E}\left(\int_{0}^{t}\|g(\mathbf{u}^\varepsilon(r))\|_{L_2(H;H)}^2 dr\right)\nonumber\\
	&&
	\quad-2\mathrm{E}\left(\int_0^t(A^\varepsilon \mathbf{u}^\varepsilon(r), \overline{\Psi}^\varepsilon)_{V'\times V}dr\right)+\mathrm{E}\left(\int_0^t(A^\varepsilon \overline{\Psi}^\varepsilon, \overline{\Psi}^\varepsilon)_{V'\times V}dr\right),\label{6.3}
\end{eqnarray}
where $\overline{\Psi}^\varepsilon(x,t):=\left(\varphi(x,t)+\varepsilon\psi(x, \frac{x}{\varepsilon}, t)\right)\cdot 1_{\mathcal{A}}(\omega)$, $\mathcal{A}\in \mathcal{B}(\Omega)$, and $\varphi\in C_{0,div}^\infty(\mathcal{O}_t)$, $\psi\in C_{0, div}^\infty(\mathcal{O}_t)\times C^\infty_{per}( D) $.

Note that as $\varepsilon\rightarrow0$
$$
\mathrm{E}\left(\int_0^t(A^\varepsilon \mathbf{u}^\varepsilon(r), \overline{\Psi}^\varepsilon)_{V'\times V}dr\right)\rightarrow
$$
\begin{align}\label{5.4}
 \sum_{i,j=1}^d \mathrm{E}\left(\int_{\mathcal{O}_r}\int_{D}a_{i,j}(y, r)\left(\frac{\partial \mathbf{u}(x, r)}{\partial x_i}+\frac{\partial \overline{\mathbf{u}}(x,y,r)}{\partial y_i}\right)\left(\frac{\partial \varphi(x,r)}{\partial x_j}+\frac{\partial \psi(x,y,r)}{\partial y_j}
\right)1_{\mathcal{A}}(\omega)dxdydr\right),
\end{align}
and
$$\mathrm{E}\left(\int_0^t(A^\varepsilon \overline{\Psi}^\varepsilon, \overline{\Psi}^\varepsilon)_{V'\times V}dr\right)\rightarrow $$
\begin{align}\label{5.10}
\sum_{i,j=1}^d \mathrm{E}\left(\int_{\mathcal{O}_r}\int_{D}a_{i,j}(y, r)\left(\frac{\partial \varphi(x,r)}{\partial x_i}+\frac{\partial \psi(x,y,r)}{\partial y_i}\right)\left(\frac{\partial \varphi(x,r)}{\partial x_j}+\frac{\partial \psi(x,y,r)}{\partial y_j}
\right)1_{\mathcal{A}}(\omega)dxdydr\right).
\end{align}

Combining \eqref{5.1}, \eqref{5.5*}, \eqref{5.6*}, \eqref{5.4} and \eqref{5.10},  we find
\begin{eqnarray}
	&&\lim_{\varepsilon\rightarrow 0}\mathrm{E}\left(\int_0^t(A^\varepsilon (\mathbf{u}^\varepsilon(r)-\overline{\Psi}^\varepsilon), \mathbf{u}^\varepsilon(r)-\overline{\Psi}^\varepsilon)_{V'\times V}dr\right)\nonumber\\
	&&=-\frac{1}{2}\mathrm{E}\left(\|\sqrt{\rho(t)}\mathbf{u}(t)\|_{L^2(\mathcal{O})}^2\right)+\frac{1}{2}\|\sqrt{\rho_0}\mathbf{u}_0\|_{L^2(\mathcal{O})}^2\nonumber\\
	&&\quad+\mathrm{E}\left(\int_{\mathcal{O}_r}\int_{D}f\left(y, r, \mathbf{u}(r)\right) \mathbf{u}(r)dxdydr\right) +\frac{1}{2}\mathrm{E}\left(\int_{0}^{t}\|g(\mathbf{u}(r))\|_{L_2(H; H)}^2 dr\right)\nonumber\\
	&&\quad-2\sum_{i,j=1}^d \mathrm{E}\left(\int_{\mathcal{O}_r}\int_{D}a_{i,j}(y,r)\left(\frac{\partial \mathbf{u}(x,r)}{\partial x_i}+\frac{\partial \overline{\mathbf{u}}(x,y,r)}{\partial y_i}\right)\left(\frac{\partial \varphi(x,r)}{\partial x_j}+\frac{\partial \psi(x,y,r)}{\partial y_j}
	\right)1_{\mathcal{A}}(\omega)dxdydr\right)\nonumber\\
	&&\quad+\sum_{i,j=1}^d \mathrm{E}\left(\int_{\mathcal{O}_r}\int_{D}a_{i,j}(y,r)\left(\frac{\partial \varphi(x,r)}{\partial x_i}+\frac{\partial \psi(x,y,r)}{\partial y_i}\right)\left(\frac{\partial \varphi(x,r)}{\partial x_j}+\frac{\partial \psi(x,y,r)}{\partial y_j}
	\right)1_{\mathcal{A}}(\omega)dxdydr\right)\nonumber\\
	&&=\sum_{i,j=1}^d \mathrm{E}\left(\int_{\mathcal{O}_r}\int_{D}a_{i,j}(y,r)\left(\frac{\partial \mathbf{u}(x,r)}{\partial x_i}+\frac{\partial \overline{\mathbf{u}}(x,y,r)}{\partial y_i}\right)\left(\frac{\partial \mathbf{u}(x,r)}{\partial x_j}+\frac{\partial \overline{\mathbf{u}}(x,y,r)}{\partial y_j}
	\right)dxdydr\right)\nonumber\\
	&&\quad-2\sum_{i,j=1}^d \mathrm{E}\left(\int_{\mathcal{O}_r}\int_{D}a_{i,j}(y,r)\left(\frac{\partial \mathbf{u}(x,r)}{\partial x_i}+\frac{\partial \overline{\mathbf{u}}(x,y,r)}{\partial y_i}\right)\left(\frac{\partial \varphi(x,r)}{\partial x_j}+\frac{\partial \psi(x,y,r)}{\partial y_j}
	\right)1_{\mathcal{A}}(\omega)dxdydr\right)\nonumber\\
	&&\quad+\sum_{i,j=1}^d \mathrm{E}\left(\int_{\mathcal{O}_r}\int_{D}a_{i,j}(y,r)\left(\frac{\partial \varphi(x,r)}{\partial x_i}+\frac{\partial \psi(x,y,r)}{\partial y_i}\right)\left(\frac{\partial \varphi(x,r)}{\partial x_j}+\frac{\partial \psi(x,y,r)}{\partial y_j}
	\right)1_{\mathcal{A}}(\omega)dxdydr\right)\nonumber\\
	&&=:\mathrm{E}\left(\int_{\mathcal{O}_r}\int_{D}a\partial(\widetilde{\mathbf{u}}-\overline{\Psi})\cdot \partial(\widetilde{\mathbf{u}}-\overline{\Psi})dxdydr\right),\label{5.7}
\end{eqnarray}
where $\widetilde{\mathbf{u}}:=(\mathbf{u}, \overline{\mathbf{u}})$, $\overline{\Psi}:=(\varphi1_{\mathcal{A}}(\omega), \psi1_{\mathcal{A}}(\omega))$. Since   $C_{0,div}^\infty(\mathcal{O}_t)\times 1_{\cdot}(\omega)$ and $C_{0,div}^\infty(\mathcal{O}_t)\times C^\infty_{per}(D)\times 1_{\cdot}(\omega)$ are dense in $L^2(\Omega; L^2(0,T; H)), L^2(\Omega; L^2(0,T; H\times L^2_{per}(D)))$, then for any \( \varepsilon_1 > 0 \), we can choose suitable \( \varphi \) and \( \psi \) such that
\begin{align}\label{5.8}
\mathrm{E}\left(\int_{\mathcal{O}_r}\int_{D}a\partial(\widetilde{\mathbf{u}}-\overline{\Psi})\cdot \partial(\widetilde{\mathbf{u}}-\overline{\Psi})dxdydr\right)\leq \varepsilon_1.
\end{align}
By \eqref{5.7} and \eqref{5.8}, we infer that there exists  $\eta>0$ such that for all $\varepsilon<\eta$
\begin{align*}
\mathrm{E}\left(\int_0^t(A^\varepsilon (\mathbf{u}^\varepsilon(r)-\overline{\Psi}^\varepsilon), \mathbf{u}^\varepsilon(r)-\overline{\Psi}^\varepsilon)_{V'\times V}dr\right)\leq 2\varepsilon_1.
\end{align*}
By condition \eqref{1.2}, we further obtain
\begin{align}\label{5.9}
\mathrm{E}\left(\int_0^t(\mathbf{u}^\varepsilon(r)-\overline{\Psi}^\varepsilon, \mathbf{u}^\varepsilon(r)-\overline{\Psi}^\varepsilon)_{V}dr\right)\leq \frac{2}{\kappa}\varepsilon_1,
\end{align}
also,
\begin{align}\label{5.88}
\mathrm{E}\left(\int_{\mathcal{O}_r}\int_{D}a\partial(\widetilde{\mathbf{u}}-\overline{\Psi})\cdot \partial(\widetilde{\mathbf{u}}-\overline{\Psi})dxdydr\right)\leq \frac{\varepsilon_1}{\kappa}.
\end{align}

Following \eqref{w1} and Lemma \ref{lem5.3}, the strong-$\Sigma$ convergence will hold once we show that
$$\left\|\frac{ \partial\mathbf{u}^\varepsilon}{\partial x_i}\right\|_{L^2(\Omega\times \mathcal{O}_r)}\rightarrow \left\|\frac{\partial\mathbf{u}}{\partial x_i}+\frac{\partial\overline{\mathbf{u}}}{\partial y_i}\right\|_{L^2(\Omega\times \mathcal{O}_r; L^2_{per}(D))}.$$
First, note that
\begin{align*}
\frac{ \partial\overline{\Psi}^\varepsilon}{\partial x_i}\rightarrow \left(\frac{\partial\varphi}{\partial x_i}+\frac{\partial\psi}{\partial y_i}\right)1_{\mathcal{A}}(\omega), ~{\rm in }~ L^2(\Omega\times \mathcal{O}_r),~{\rm strong}-\Sigma,
\end{align*}
then
\begin{align*}
\left\|\frac{ \partial\overline{\Psi}^\varepsilon}{\partial x_i}\right\|_{L^2(\Omega\times \mathcal{O}_r)}\rightarrow\left\|\left(\frac{\partial\varphi}{\partial x_i}+\frac{\partial\psi}{\partial y_i}\right)1_{\mathcal{A}}(\omega)\right\|_{L^2(\Omega\times \mathcal{O}_r; L^2_{per}(D))}.
\end{align*}
Thus, for any $\varepsilon_2>0$, there exists  $\delta>0$ such that $\varepsilon<\delta$, we have
\begin{align}\label{5.11}
\left|\left\|\frac{ \partial\overline{\Psi}^\varepsilon}{\partial x_i}\right\|_{L^2(\Omega\times \mathcal{O}_r)}-\left\|\left(\frac{\partial\varphi}{\partial x_i}+\frac{\partial\psi}{\partial y_i}\right)1_{\mathcal{A}}(\omega)\right\|_{L^2(\Omega\times \mathcal{O}_r; L^2_{per}(D))}\right|\leq \varepsilon_2.
\end{align}

Using the triangle inequality and \eqref{5.9}-\eqref{5.11}, we conclude
\begin{align*}
&\left|\left\|\frac{ \partial\mathbf{u}^\varepsilon}{\partial x_i}\right\|_{L^2(\Omega\times \mathcal{O}_r)}-\left\|\frac{\partial\mathbf{u}}{\partial x_i}+\frac{\partial\overline{\mathbf{u}}}{\partial y_i}\right\|_{L^2(\Omega\times \mathcal{O}_r; L^2_{per}(D))}\right|\nonumber\\
&\leq \left|\left\|\frac{ \partial\mathbf{u}^\varepsilon}{\partial x_i}\right\|_{L^2(\Omega\times \mathcal{O}_r)}-\left\|\frac{ \partial\overline{\Psi}^\varepsilon}{\partial x_i}\right\|_{L^2(\Omega\times \mathcal{O}_r)}\right|\nonumber\\
&\quad+\left|\left\|\frac{ \partial\overline{\Psi}^\varepsilon}{\partial x_i}\right\|_{L^2(\Omega\times \mathcal{O}_r)}-\left\|\left(\frac{\partial\varphi}{\partial x_i}+\frac{\partial\psi}{\partial y_i}\right)1_{\mathcal{A}}(\omega)\right\|_{L^2(\Omega\times \mathcal{O}_r; L^2_{per}(D))}\right|\nonumber\\
&\quad+\left|\left\|\frac{\partial\mathbf{u}}{\partial x_i}+\frac{\partial\overline{\mathbf{u}}}{\partial y_i}\right\|_{L^2(\Omega\times \mathcal{O}_r; L^2_{per}(D))}-\left\|\left(\frac{\partial\varphi}{\partial x_i}+\frac{\partial\psi}{\partial y_i}\right)1_{\mathcal{A}}(\omega)\right\|_{L^2(\Omega\times \mathcal{O}_r; L^2_{per}(D))}\right|\nonumber\\
&\leq \frac{3}{\kappa}\varepsilon_1+\varepsilon_2,
\end{align*}
the arbitrariness of $\varepsilon_1, \varepsilon_2$ leads to the desired result.
\qquad$\Box$

\section*{Acknowledgments}
 Z. Qiu is supported by the National Natural Science Foundation of China
(Grant No. 12401305), the National Science Foundation for Colleges and Universities in Jiangsu Province (Grant No. 24KJB110011) and the National Science Foundation of Jiangsu
Province (Grant No. BK20240721).
J. Chen and J. Duan are partly supported by the NSFC  W2541005, the Guangdong Provincial Key Laboratory of Mathematical and Neural Dynamical Systems (Grant 2024B1212010004),   the Cross Disciplinary Research Team on Data Science and Intelligent Medicine (2023KCXTD054), and the Guangdong-Dongguan Joint Research Fund (Grant 2023A151514 0016).
\section*{Data availability}
Data sharing not applicable to this article as no datasets were generated or analysed
during the current study.

\section*{Statements and Declarations}
The authors have no relevant financial or non-financial interests to disclose.

\end{document}